\documentclass{article}

\RequirePackage{amsthm,amsmath,amsfonts,amssymb}
\RequirePackage[authoryear]{natbib}
\RequirePackage[colorlinks,citecolor=blue,urlcolor=blue]{hyperref}
\RequirePackage{graphicx}

\RequirePackage{algorithm}
\RequirePackage{algpseudocode}
\RequirePackage{enumitem}
\RequirePackage{centernot}

\RequirePackage{authblk}

\RequirePackage{geometry}
\newenvironment{itemize*}%
  {\vspace*{-2mm}
  \begin{itemize}%
    \setlength{\itemsep}{0.5pt}}%
  {\end{itemize}
  \vspace*{-2mm}}

\newenvironment{enumerate*}%
  {\vspace*{-3mm}
  \begin{enumerate}%
    \setlength{\itemsep}{0.5pt}}%
  {\end{enumerate}
  \vspace*{-2mm}}

\newenvironment{enumerate**}%
  {\begin{enumerate}%
    \setlength{\itemsep}{0.5pt}}%
  {\end{enumerate}
  \vspace*{-2mm}}

\theoremstyle{plain}
\newtheorem{theorem}{Theorem}
\newtheorem{proposition}{Proposition}[section]
\newtheorem{corollary}[proposition]{Corollary}
\newtheorem{lemma}[proposition]{Lemma}

\theoremstyle{remark}
\newtheorem{problem}{Problem}
\newtheorem{definition}[proposition]{Definition}
\newtheorem{example}[proposition]{Example}
\newtheorem{remark}[proposition]{Remark}

\newtheorem*{continuancex}{Example \continuanceref \ (cont.)}
\newenvironment{continuance}[1]
  {\newcommand\continuanceref{\ref{#1}}\continuancex}
  {\endcontinuancex}

\newcommand{\Graph}{\mathcal{G}} % Graph, here always a DMG
\newcommand{\DAG}{\mathcal{D}} % DAG
\newcommand{\ADMG}{\mathcal{A}} % ADMG
\newcommand{\DMAG}{\mathcal{M}} % DMAG
\newcommand{\InducedSubgraph}[2]{\Graph_{sub}(#1, #2)} % induced subgraph

\newcommand{\ADMGprojection}[2]{\sigma_{\text{ADMG}}(#1, #2)}
\newcommand{\DMAGprojection}[2]{\sigma_{\text{DMAG}}(#1, #2)}

\newcommand{\Vertices}{\mathbf{V}} % vertex set
\newcommand{\DirectedEdges}{\mathbf{E}_{\rightarrow}} % set of directed edges
\newcommand{\BidirectedEdges}{\mathbf{E}_{\leftrightarrow}} % set of bidirected edges

\newcommand{\VarIndices}{\mathbf{I}} % variable index set
\newcommand{\TimeIndices}{\mathbf{T}} % time index set

\newcommand{\ObservedVertices}{\mathbf{O}} % observed vertices
\newcommand{\LatentVertices}{\mathbf{L}} % unobserved vertices

\newcommand{\Zbold}{\mathbf{Z}} % Vertex set, typically used as conditioning set

\newcommand{\tailhead}{{\,\rightarrow\,}} % directed edge to the right
\newcommand{\headtail}{{\,\leftarrow\,}} % directed edge to the left
\newcommand{\headhead}{{\,\leftrightarrow\,}} % bidirected edge
\newcommand{\asthead}{{\ast\!\!\!\rightarrow}} % wildcard 1
\newcommand{\headast}{{\leftarrow\!\!\!\ast}} % wildcard 2
\newcommand{\astast}{{\ast\!{-}\!\ast}} % wildcard 3

\newcommand{\length}[1]{len(#1)} % length of a path
\newcommand{\weight}[1]{w(#1)} % weight of a path in a graph with time series structure
\newcommand{\weightset}[1]{\mathbf{w}(#1)} % sets of weights of a path in a summary graph

\newcommand{\parents}[2]{pa(#1, #2)} % parents of a vertex in a graph
\newcommand{\children}[2]{ch(#1, #2)} % children of a vertex in a graph
\newcommand{\ancestors}[2]{an(#1, #2)} % ancestors of a vertex in a graph
\newcommand{\descendants}[2]{de(#1, #2)} % descendants of a vertex in a graph
\newcommand{\spouses}[2]{sb(#1, #2)} % spouses of a vertex in a graph

\newcommand{\CanonicalDAG}[1]{\DAG_{c}(#1)} % canonical DAG
\newcommand{\CanonicaltsDAG}[1]{\DAG^{ts}_{c}(#1)} % canonical DAG

\newcommand{\VarIndicesObserved}{\VarIndices_{\ObservedVertices}} % variable indices of observed component time series
\newcommand{\VarIndicesUnobserved}{\VarIndices_{\LatentVertices}} % variable indices of observed component time series
\newcommand{\TimeIndicesObserved}{\TimeIndices_{\ObservedVertices}} % variable indices of observed component time series

\newcommand{\tsDMAG}[2]{\DMAG_{#1}(#2)} % time series DMAG
\newcommand{\tsADMG}[2]{\ADMG_{#1}(#2)} % time series DMAG

\newcommand{\ptsDAG}{p_{\ADMG}} % maximal lag in a ts-DAG
\newcommand{\ptimewindow}{p} % length of the observed time window
\newcommand{\pcutoff}{p^\prime} % length of an arbitrary window

\newcommand{\LatentVerticesTemporally}{\LatentVertices^{temp}} % temporally unobserved vertices
\newcommand{\LatentVerticesUnobservable}{\LatentVertices^{unob}} % temporally observed unobservable vertices

\newcommand{\SummaryGraph}[1]{\mathcal{S}(#1)} % summary graph

\newcommand{\con}[1]{\mathrm{con}(#1)} % integer conical hull

\newcommand{\cG}{\mathcal{G}}

\newcommand{\cW}{\mathcal{W}}
\newcommand{\cC}{\mathcal{C}}
\newcommand{\cS}{\mathcal{S}}
\newcommand{\cM}{\mathcal{M}}

\newcommand{\w}{\mathbf{w}}
\renewcommand{\c}{\mathbf{c}}

\title{Projecting infinite time series graphs to finite marginal graphs using number theory}

\author[1]{Andreas Gerhardus\thanks{\texttt{andreas.gerhardus@dlr.de}}}
\author[2, 1]{Jonas Wahl\thanks{\texttt{wahl@tu-berlin.de}}}
\author[1]{Sofia Faltenbacher\thanks{\texttt{sofia.faltenbacher@dlr.de}}}
\author[2, 1]{Urmi Ninad\thanks{\texttt{urmi.ninad@tu-berlin.de}}}
\author[1, 2]{Jakob Runge\thanks{\texttt{jakob.runge@dlr.de}}}

\affil[1]{German Aerospace Center, Institute of Data Science}
\affil[2]{Technische Universität Berlin, Institute of Computer Engineering and Microelectronics}

\begin{document}

\maketitle

\begin{abstract}
In recent years, a growing number of method and application works have adapted and applied the causal-graphical-model framework to time series data. Many of these works employ time-resolved causal graphs that extend infinitely into the past and future and whose edges are repetitive in time, thereby reflecting the assumption of stationary causal relationships. However, most results and algorithms from the causal-graphical-model framework are not designed for infinite graphs. In this work, we develop a method for projecting infinite time series graphs with repetitive edges to marginal graphical models on a finite time window. These finite marginal graphs provide the answers to $m$-separation queries with respect to the infinite graph, a task that was previously unresolved. Moreover, we argue that these marginal graphs are useful for causal discovery and causal effect estimation in time series, effectively enabling to apply results developed for finite graphs to the infinite graphs. The projection procedure relies on finding common ancestors in the to-be-projected graph and is, by itself, not new. However, the projection procedure has not yet been algorithmically implemented for time series graphs since in these infinite graphs there can be infinite sets of paths that might give rise to common ancestors. We solve the search over these possibly infinite sets of paths by an intriguing combination of path-finding techniques for finite directed graphs and solution theory for linear Diophantine equations. By providing an algorithm that carries out the projection, our paper makes an important step towards a theoretically-grounded and method-agnostic generalization of a range of causal inference methods and results to time series.
\end{abstract}

\tableofcontents

\section{Introduction}
Many research questions, from the social and life sciences to the natural sciences and engineering, are inherently causal. \emph{Causal inference} provides the theoretical foundations and a variety of methods to combine statistical or machine learning models with domain knowledge in order to quantitatively answer causal questions based on experimental and/or observational data, see for example \citet{pearl2009causality}, \citet{imbens2015causal}, \citet{Spirtes2000}, \citet{peters2017elements} and \citet{hernan2020causal}. Since domain knowledge often exists in the form causal graphs that assert qualitative cause-and-effect relationships, the \emph{causal-graphical-model framework} \citep{pearl2009causality} has become increasingly popular during the last few decades.

By now, there is a vast body of literature on causal-graphical modeling. Broadly speaking, the framework subsumes the subfields \emph{causal discovery} and \emph{causal effect identification}. In causal discovery, see for example \citet{Spirtes2000} and \citet{peters2017elements}, the goal is to learn qualitative cause-and-effect relationships (that is, the causal graph) by leveraging appropriate enabling assumptions on the data-generating process. In causal effect identification, the goal is to predict the effect of \emph{interventions} \citep{pearl2009causality}, which are idealized abstractions of experimental manipulations of the system under study, by leveraging knowledge of (or assumptions on) the causal graph. There are diverse methods and approaches for this purpose, such as the famous \emph{backdoor criterion} \citep{pearl1993bayesian}, the \emph{generalized adjustment criterion} \citep{shpitser2010validity, perkovic2018complete}, the \emph{instrumental variables} approach \citep{sargan1958estimation,angrist2009mostly}, the \emph{$do$-calculus} \citep{pearl1995causal, huang2006pearls, shpitser2006identification, shpitser2008complete} and \emph{causal transportability} \citep{bareinboim2016causal}, to name a few. In recent years, \emph{causal representation learning}, see for example \citet{schoelkopf2021toward}, has emerged as another branch and gained increasing popularity in the machine learning community. In causal representation learning, the goal is to learn causally meaningful variables that can serve as the nodes of a causal graph at an appropriate level of abstraction.

As one of its major achievements, causal-graphical modeling does not necessarily require temporal information for telling apart cause and effect. In fact, most of the causal-graphical-model framework was originally developed without reference to time \citep{pearl2009causality}. However, as many research fields specifically concern causal questions about dynamic phenomena, such as in Earth Sciences \citep{runge2019inferring}, ecology \citep{runge2023modern} or neuroscience \citep{danks2023causal}, in recent years there is a growing interest in adapting the framework to time series. More specifically, in this paper we draw our motivation from adaptations of causal-graphical modeling to the discrete-time domain; see for example \citet{runge2023causal} and \citet{camps2023discovering} for recent reviews of this setting.

In the discrete-time setting, there are, broadly, two different approaches to causal-graphical modeling. The first approach uses time-collapsed graphs, also known as \emph{summary graphs}, which represent each component time series by a single vertex and summarize causal influences across all time lags by a single edge. \emph{Granger causality} \citep{granger1969investigating} uses this approach, and various works refined and extended Granger's work, see for example \citet[see the notion of \emph{causality graphs}]{dahlhaus2003causality}, \citet{eichler2007causal}, \citet{eichler2010granger}, \citet{eichler2010graphical}, and further works discussed in \citet{assaad2022survey}. This first approach is ideally suited to problem settings in which one is not interested in the specific time lags of the causal relationships. The second approach uses time-resolved graphs, which represent each time step of each component time series by a separate vertex and thus explicitly resolve the time lags of causal influences. Examples of works that employ this approach are \citet[see the notion of \emph{time series chain graphs}]{dahlhaus2003causality}, \citet{chu2008search}, \citet{hyvarinen2010estimation}, \citet{runge2019detecting}, \citet{runge2020discovering}, and \citet{thams2022identifying}. This second approach is ideally suited to problem settings in which the specific time lags of the causal relationships are of importance. As opposed to time-collapsed graphs, time-resolved graphs extend to the infinite past and future and thus have an infinite number of vertices. However, adopting the assumption of time-invariant qualitative causal relationships (often referred to as \emph{causal stationarity}), the edges of the infinite time-resolved graphs are repetitive in time. As a result, despite being infinite, causally stationary time-resolved graphs admit a finite description. Throughout this paper, unless explicitly stated otherwise, we only consider causally stationary time-resolved graphs.

Time-resolved graphs are useful for a number of reasons: 1) knowledge of time lags is crucial for a deeper process understanding, for example, in the case of time delays of atmospheric teleconnections \citep{runge2019inferring}, and is relevant for tasks such as climate model evaluation \citep{nowack2020causal}; 2) knowledge of specific time lags allows for a more parsimonious and low-dimensional representation as opposed to modeling all past influences up to some maximal time lag; 3) causally-informed forecasting models benefit from precise time lag information \citep{runge2015optimal}; and 4) the lag-structure can render particular causal effect queries identifiable.

This paper draws its motivation from the time-resolved approach to causal-graphical modeling. Since this setting is conceptually close to the non-temporal modeling framework, one can in principle hope to straightforwardly generalize the wealth of causal effect identification methods developed for non-time-series data. This task is much less straightforward in the time-collapsed approach, where it is harder to in the first place define and then interpret causal effects between time-collapsed nodes (see for example \citet{reiter2023formalising} for a recent work in this direction). However, for causally stationary time-resolved graphs one still faces the technical complication of having to deal with infinite graphs, whereas most causal effect identification methods are designed for finite graphs. Therefore, these methods still need specific modifications for making them applicable to infinite graphs.

In this paper, we provide a new approach that resolves this complication of having to deal with infinite graphs: A method for projecting infinite causally stationary time-resolved graphs to finite marginal graphs on arbitrary finite time windows. Since this projection preserves \emph{$m$-separations} \citep{richardson2002ancestral, richardson2003markov} (cf.~\citet{pearl1988} for the related notion of \emph{$d$-separations}) as well as causal ancestral relationships, one can equivalently check the graphical criteria of many causal effect identification methods on appropriate finite marginal graphs instead of on the infinite time-resolved graph itself. In particular, one can answer $m$-separation queries with respect to an infinite time-resolved graph by asking the same query for any of its finite marginal graphs for given finite time-window lengths that contain all vertices involved in the query. Figure~\ref{fig:introduction-figure} illustrates one example of an infinite time-resolved graphs together with one of its finite marginal graphs. Intuitively speaking, the projection method implicitly takes care of the infiniteness of the time-resolved graphs and thereby relieves downstream applications (such as methods for causal effect identification and answering $m$-separation queries) from having to deal with infinite graphs. By providing this reformulation, our paper makes an important step towards a theoretically-grounded and method-agnostic generalization of causal effect identification methods to time series.

\begin{figure}[tp]
    \centering
    \includegraphics[scale=0.25]{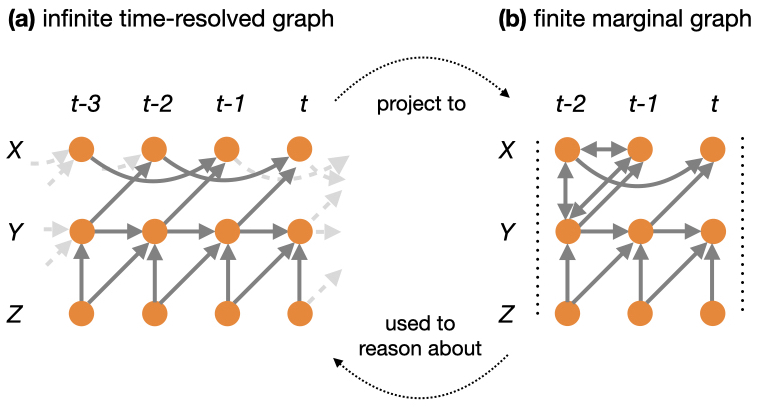}
    \caption{Part \textbf{(a)} shows an infinite time-resolved graph for a time series with three components $X$, $Y$ and $Z$. By the assumption of causal stationarity, the edges in this graph are repetitive in time, that is, the graph is invariant under shifts along its temporal dimension. Moreover, as the dashed light gray edges indicate, the graph extends infinitely both into the past and into the future. The contribution of this paper is the development of an algorithmic method that projects such infinite graphs to finite marginal graphs on arbitrary finite time windows. To illustrate, part \textbf{(b)} shows the finite marginal graph that results from projecting the infinite graph in part (a) to the time window $[t-2, t]$. Note that this finite marginal graph does not extend to the infinite past or future, as the vertical dashed lines indicate. Since the projection preserves relevant graphical properties, one can use the finite marginal graph in part (b) to reason about the infinite graph in part (a).}
    \label{fig:introduction-figure}
\end{figure}

Besides causal effect identification and $m$-separation queries, our projection method is also useful for causal discovery. The reason is that (equivalence classes of) finite marginal graphs of infinite time-resolved graphs are the natural targets of time-resolved time series causal discovery, see \citet{gerhardus2021characterization} for more details on this matter. Therefore, in order to obtain a conceptual understanding of the very targets of time-resolved causal discovery, one needs a method for constructing these finite marginal graphs---which our work provides.

As the projection procedure, we here employ the widely-used ADMG latent projection \citep{pearl1995theory} (see also for example \citet{richardson2023nested}), thereby giving rise to what we below call \emph{marginal time series ADMGs (marginal ts-ADMGs)}, see Def.~\ref{def:tsADMG-marginal} below. In Section~\ref{sec:generalize-to-DMAGs}, we then extend our results to the DMAG latent projection \citep{richardson2002ancestral, zhang2008causal}, which too is widely-used and gives rise to what we call \emph{marginal time series DMAGs (marginal ts-DMAGs)}, see Definition~\ref{def:tsDMAG} (which is adapted and generalized from \citet{gerhardus2021characterization}). While both of these projection procedures themselves are not new, their practical application to infinite time-resolved graphs is non-trivial and is, to the authors' knowledge, not yet solved in generality. The issue is that, when applied to infinite time-resolved graphs, the projection procedures require a search over a potentially infinite number of paths. In this paper, we show how to circumvent this issue by making use of the repetitive structure of the infinite time-resolved graphs in combination with the evaluation of a finite number of number-theoretic solvability problems. Thus, as an important point to note, our solution crucially relies on the assumption of causal stationarity.

\paragraph*{Related works}
The work of \citet{gerhardus2021characterization} already defines marginal ts-DMAGs and utilizes these finite graphs for the purpose of time-resolved time series causal discovery. This work also presents ts-DMAGs for several examples of infinite time-resolved graphs, but does not give a general method for their construction. The work of \citet{thams2022identifying} already considers finite marginal graphs obtained by the ADMG latent projection of infinite time series graphs and utilizes these finite graphs for the purpose of causal effect identification in time series. However, this work presents the finite marginal graphs of only a single infinite time-resolved graph (which, in addition, is of a certain special type that comes with significant simplifications) and it too does not give a general method for their construction. As opposed to these two works, we here consider a general class of infinite time-resolved graphs and derive an algorithmic method for constructing their finite marginals---both for the ADMG latent projection, see the main paper, and the DMAG latent projection, see Section~\ref{sec:generalize-to-DMAGs}.

\paragraph*{Structure and main results of this work}
In Section~\ref{sec:preliminaries}, we summarize the necessary preliminaries. In Section~\ref{sec:problem-formulation}, we first elaborate in more detail on the above-mentioned reasons for why finite marginal graphs of infinite time-resolved graphs are useful in the context of causal inference. We then formally define the finite marginal graphs obtained by the ADMG latent projection and reduce their construction to the search for common ancestors of pairs of vertices in certain infinite DAGs. This common-ancestor search is non-trivial because, in general, there is an infinite number of paths that might give rise to common ancestors. In Section~\ref{sec:refined-common-ancestor}, we first solve the common-ancestor search in a significantly simplified yet important special case and also show why the same strategy does not work in the general case. We then, for the general case, map the common-ancestor search to the number-theoretic problem of deciding whether at least one of a finite collection of linear Diophantine equations has a non-negative integer solution (\textbf{Theorem~\ref{cor.prob1}}). This result establishes an intriguing connection between graph theory and number theory, which might be of interest in its own right. Intuitively speaking, the mapping reformulates the problem of searching over the potentially infinite number of to-be-considered paths to a geometric intersection problem for still infinite but \emph{finite-dimensional} affine cones over non-negative integers. Deciding whether such cones intersect is then equivalent to deciding whether a certain linear Diophantine equation admits a non-negative integer solution. Next, building on well-established results from number theory, we provide a criterion that answers the resulting number-theoretic solvability problem in finite time (\textbf{Theorem~\ref{thm:our-linear-diophantine}}). Thereby, we obtain an algorithmic and provably correct finite-time solution to the task of constructing finite marginal graphs of infinite time-resolved graphs. As a corollary, we also present an upper bound on a finite time window to which one can restrict the infinite time-resolved graphs before projecting them to the finite marginal graphs (\textbf{Theorem~\ref{thm.upper-bound-main-theorem}}). This result provides a second solution to the problem of constructing finite marginal graphs, which appears conceptually simpler as it conceals the underlying number-theoretic problem but might computationally more expensive. In Section~\ref{sec:summary}, we present the conclusions. In Sections~\ref{sec:generalize-to-DMAGs} and \ref{sec:counterexamples}, we respectively extend our results to the finite marginal graphs obtained by the DMAG latent projection and present two examples that we omit from the main paper for brevity. In Sections~\ref{sec:proofs} and \ref{sec:pseudocode}, we provide the proofs of all theoretical claims and pseudocode for our number-theoretic solution to the projection task.

\section{Preliminaries}\label{sec:preliminaries}
In this section, we summarize graphical terminology and notation that we use and build upon throughout this work. Our notation takes inspiration from \citet{mooij2020constraint} and \citet{gerhardus2021characterization}, among others.

\subsection{Basic graphical concepts and notation}
A \emph{directed mixed graph (DMG)} is a triple $\Graph = (\Vertices, \DirectedEdges, \BidirectedEdges)$ where $\Vertices$ is the set of \emph{vertices} (also referred to as \emph{nodes}), $\DirectedEdges \subseteq \Vertices \times \Vertices$ is the set of \emph{directed edges}, and $\BidirectedEdges \subseteq \left(\Vertices \times \Vertices\right) \!/ \,\mathbb{Z}_2$ is the set of \emph{bidirected edges}. Here, the $\mathbb{Z}_2$ action identifies the bidirected edges $(i, j) \in \BidirectedEdges$ and $(j, i) \in \BidirectedEdges$ with each other. We denote a directed edge $(i, j) \in \DirectedEdges$ as $i \tailhead j$ or $j \headtail i$ and a bidirected edge $(i, j) \in \BidirectedEdges$ as $i \headhead j$ or $j \headhead i$. We use $i \asthead j$ resp.~$i \headast j$ as a wildcard for $i \tailhead j$ or $i \headhead j$ resp.~$i \headtail j$ or $i \headhead j$, and use $i \astast j$ as a wildcard for $i \asthead j$ or $i \headast j$. We say that two vertices $i, j \in \Vertices$ are \emph{adjacent} in a DMG $\Graph$ if there is an edge of any type between them, that is, if $i \astast j$ in $\Graph$. This definition allows \emph{self edges}, that is, edges of the form $i \tailhead i$ and $i \headhead i$. Moreover, the definition allows a pair of vertices $i, j \in \Vertices$ to be connected by more than one edge.

The \emph{induced subgraph} $\InducedSubgraph{\Vertices^\prime}{\Graph}$ of a DMG $\Graph = (\Vertices, \DirectedEdges, \BidirectedEdges)$ on a subset $\Vertices^\prime \subseteq \Vertices$ of the vertices is the DMG $\InducedSubgraph{\Vertices^\prime}{\Graph} = (\Vertices^\prime, \DirectedEdges^\prime, \BidirectedEdges^\prime)$ where $\DirectedEdges^\prime = \DirectedEdges \cap \left( \Vertices^\prime \times \Vertices^\prime \right)$ and $\BidirectedEdges^\prime = \BidirectedEdges \cap \left( \Vertices^\prime \times \Vertices^\prime \right) \!/ \,\mathbb{Z}_2$. Intuitively, $\InducedSubgraph{\Vertices^\prime}{\Graph}$ contains all and only those vertices in the subset $\Vertices^\prime$ as well as the edges between them.

A \emph{walk} is a finite ordered sequence $\pi = (\pi(1), e_1, \pi(2), e_2, \pi(3), \ldots, e_{n-1}, \pi(n))$ where $\pi(1), \ldots, \pi(n)$ are vertices and where for all $k = 1, \ldots, n-1$ the edge $e_k$ connects $\pi(k)$ and $\pi(k+1)$. We say that $\pi$ is \emph{between $\pi(1)$ and $\pi(n)$}. The integer $n$ is the \emph{length} of the walk $\pi$, which we also denote as $\length{\pi}$. We call the vertices $\pi(1)$ and $\pi(n)$ the \emph{endpoint vertices} on $\pi$ and call the vertices $\pi(k)$ with $1 < k < n$ the \emph{middle vertices} on $\pi$. The definition allows \emph{trivial} walks, that is, walks which consist of a single vertex and no edges. If all vertices $\pi(1), \ldots, \pi(n)$ are distinct, then $\pi$ is a \emph{path}. We can represent and specify a walk graphically, for example $v_1 \headtail v_2 \headhead v_3 \tailhead v_4$.

For $k$ and $l$ with $1 \leq k < l \leq \length{\pi}$, we let $\pi(k, l)$ denote the walk $\pi^\prime = (\pi(k), e_k, \pi(k+1), \ldots, e_{l-1}, \pi(l))$. We say that a walk (path) $\pi^\prime$ is a \emph{subwalk (subpath)} of a walk (path) $\pi$ if there are $k$ and $l$ with $1 \leq k < l \leq \length{\pi}$ such that $\pi^\prime = \pi(k, l)$. A subwalk (subpath) $\pi^\prime$ of a walk (path) $\pi$ is \emph{proper} if $\pi^\prime \neq \pi$.

A walk is \emph{into} its first vertex $\pi(1)$ if its first edge has an arrowhead at $\pi(1)$, that is, if $\pi(1, 2)$ is of the form $\pi(1) \headast \pi(2)$. If a walk is not into its first vertex, then it is \emph{out of} its first vertex. Similarly, a walk is into (resp.~out of) its last vertex $\pi(\length{\pi})$ if its last edge has (resp.~does not have) an arrowhead at $\pi(\length{\pi})$.

A middle vertex $\pi(k)$ on a walk $\pi$ is a \emph{collider on $\pi$} if the edges on $\pi$ meet head-to-head at $\pi(k)$, that is, if the subwalk $\pi(k-1, k+1)$ is of the form $\pi(k-1) \asthead \pi(k) \headast \pi(k+1)$. A middle vertex on a walk is a \emph{non-collider on $\pi$} if it is not a collider on $\pi$.

A walk is \emph{directed} if it is non-trivial and takes the form $v_1 \headtail v_2 \headtail\ldots \headtail v_n$ or $v_1 \tailhead v_2 \tailhead \ldots \tailhead v_n$. A non-trivial walk $\pi$ is a \emph{confounding walk} if, first, no middle vertex on $\pi$ is a collider and, second, $\pi$ is into both its endpoint vertices.

A walk is a \emph{cycle} if it is non-trivial and $\pi(1) = \pi(\length{\pi})$. A cycle is \emph{irreducible} if it does not have a proper subwalk that is also a cycle. Thus, a cycle is irreducible if and only if, first, no vertex other than $\pi(1)$ appears twice and, second, $\pi(1)$ appears not more than twice. A cycle that is not irreducible is \emph{reducible}. A walk is \emph{cycle-free} if it does not have a subwalk which is a cycle. A directed \emph{path} is always cycle-free. Two cycles $c_1$ and $c_2$ are \emph{equivalent to each other} if and only if $c_2$ can be obtained by i) revolving the vertices on $c_1$ or by ii) reversing the order of the vertices on $c_1$ or by iii) by a combination of these two operations.

If in a DMG $\Graph$ there is an edge $i \tailhead j$, then $i$ is a \emph{parent} of $j$ and $j$ is a \emph{child} of $i$. We denote the sets of parents and children of $i$ as, respectively, $\parents{i}{\Graph}$ and $\children{i}{\Graph}$. If there is a directed walk from $i$ to $j$ or $i = j$, then $i$ is an \emph{ancestor} of $j$ and $j$ is a \emph{descendant} of $i$. We denote the sets of ancestors and descendants of $i$ as, respectively, $\ancestors{i}{\Graph}$ and $\descendants{i}{\Graph}$. If $i \headhead j$, then $i$ and $j$ are \emph{spouses} of each other. We denote the set of spouses of $i$ as $\spouses{i}{\Graph}$.

A path $\pi$ between $i$ and $j$ is an \emph{inducing path} if, first, all its middle vertices are ancestors of $i$ or $j$ and, second, all its middle vertices are colliders on $\pi$.

An \emph{acyclic directed mixed graph (ADMG)}, which we here typically denote as $\ADMG$, is a DMG without directed cycles. A \emph{directed graph} is a DMG $(\Vertices, \DirectedEdges, \BidirectedEdges)$ without bidirected edges, that is, a DMG with $\BidirectedEdges = \emptyset$. For simplicity, we identify a directed graph with the pair $(\Vertices, \DirectedEdges)$. A \emph{directed acyclic graph (DAG)}, which we here typically denoted as $\DAG$, is an ADMG that is also a directed graph. An ADMG $\ADMG$ is \emph{ancestral} if it satisfies two conditions: First, $\ADMG$ does not have self edges. Second, $i \notin \spouses{j}{\ADMG}$ if $i \in \ancestors{j}{\ADMG}$. It follows that an ancestral ADMG has most one edge between any pair of vertices.

The $m$-separation criterion \citep{richardson2002ancestral} extends the $d$-separation criterion \citep{pearl1988} from DAGs to ADMGs: A path $\pi$ between the vertices $i$ and $j$ in an ADMG with vertex set $\Vertices$ is \emph{$m$-connecting} given a set $\Zbold \subseteq \Vertices \setminus \{i, j\}$ if, first, no non-collider on $\pi$ is in $\Zbold$ and, second, every collider on $\pi$ is an ancestor of some element in $\Zbold$. If a path is not $m$-connecting given $\Zbold$, then the path is \emph{$m$-blocked}. The vertices $i$ and $j$ are \emph{$m$-connected} given $\Zbold$ if there is at least one path between $i$ and $j$ that is $m$-connecting given $\Zbold$. If the vertices $i$ and $j$ are not $m$-connected given $\Zbold$, then they are \emph{$m$-separated} given $\Zbold$.

Let $\ADMG = (\Vertices, \DirectedEdges, \BidirectedEdges)$ be an ADMG without self edges. Then, its \emph{canonical DAG} $\CanonicalDAG{\ADMG}$ is the directed graph $(\Vertices^\prime, \DirectedEdges^\prime)$ where $\Vertices^\prime = \Vertices \cup \LatentVertices$ with $\LatentVertices = \{l_{ij} ~|~ (i, j) \in \BidirectedEdges\}$ and $\DirectedEdges^\prime = \DirectedEdges \cup \{(l_{ij}, i) ~|~ l_{ij} \in \LatentVertices\} \cup \{(l_{ij}, j) ~|~ l_{ij} \in \LatentVertices\}$ (cf.~Section 6.1 of \citet{richardson2002ancestral}).\footnote{The definition of canonical DAGs in \citet{richardson2002ancestral} requires the ADMG to be ancestral. However, we can use the very same definition also for non-ancestral ADMGs without self edges.} Intuitively, we obtain $\CanonicalDAG{\ADMG}$ from $\ADMG$ by replacing each bidirected edge $i \headhead j$ in $\ADMG$ with $i \headtail l_{ij} \tailhead j$. It follows that acyclicity of $\ADMG$ carries over to $\CanonicalDAG{\ADMG}$, which means that $\CanonicalDAG{\ADMG}$ is indeed a DAG.

\subsection{The ADMG latent projection}\label{sec:latent-projections}
In many applications, some of the vertices $\Vertices$ of a graph $\Graph$ serving as a graphical model might correspond to unobserved variables. We can formalize this situation by a partition $\Vertices = \ObservedVertices \,\dot{\cup}\, \LatentVertices$ of the vertices $\Vertices$ into the \emph{observed vertices $\ObservedVertices$} and the \emph{unobserved / latent vertices $\LatentVertices$}. If one is predominantly interested in reasoning about the observed vertices, then it is often convenient to project $\Graph$ to a \emph{marginal graph} on the observed vertices only---provided the projection preserves certain graphical properties of interest. In this paper, we consider the following widely-used projection procedure.

\begin{definition}[ADMG latent projection \citep{pearl1995theory}, see also for example \citet{richardson2023nested}]\label{def:ADMG-latent-projection}
Let $\ADMG$ be an ADMG with vertex set $\Vertices = \ObservedVertices \,\dot{\cup}\, \LatentVertices$ that has no self edges. Then, its \emph{marginal ADMG $\ADMGprojection{\ObservedVertices}{\ADMG}$ on $\ObservedVertices$} is the graph with vertex set $\ObservedVertices$ such that
\begin{enumerate}
    \item there is a directed edge $i \tailhead j$ in $\ADMGprojection{\ObservedVertices}{\ADMG}$ if and only if in $\ADMG$ there is at least one directed path from $i$ to $j$ such that all middle vertices on this path are in $\LatentVertices$, and
    \item there is a bidirected edge $i \headhead j$ in $\ADMGprojection{\ObservedVertices}{\ADMG}$ if and only if in $\ADMG$ there is at least one confounding path between $i$ and $j$ such that all middle vertices on this path are in $\LatentVertices$.
\end{enumerate}
\end{definition}

\noindent It follows that $i \in \ancestors{j}{\ADMGprojection{\ObservedVertices}{\ADMG}}$ if and only if $i, j \in \ObservedVertices$ and $i \in \ancestors{j}{\ADMG}$. The acyclicity of $\ADMG$ thus carries over to $\ADMGprojection{\ObservedVertices}{\ADMG}$, so that $\ADMGprojection{\ObservedVertices}{\ADMG}$ is an ADMG indeed. Moreover, the definitions imply that $\ADMGprojection{\ObservedVertices}{\ADMG}$ does not have self edges. There can be more than one edge between a pair of vertices in $\ADMGprojection{\ObservedVertices}{\ADMG}$, namely $i \headhead j$ plus $i \tailhead j$ or $i \headtail j$. Thus, in particular, the marginal ADMG $\ADMGprojection{\ObservedVertices}{\ADMG}$ is not necessarily ancestral. Two observed vertices $i$ and $j$ are $m$-separated given $\Zbold$ in $\ADMG$ if and only if $i$ and $j$ are $m$-separated given $\Zbold$ in $\ADMGprojection{\ObservedVertices}{\ADMG}$, see Proposition 1 in \citet{richardson2023nested}.\footnote{Proposition 1 in \citet{richardson2023nested} only applies to ADMG latent projections of DAGs rather than of ADMGs without self edges, but the proof in \citet{richardson2023nested} also works for the latter more general case.}

In Section~\ref{sec:generalize-to-DMAGs}, we further consider the DMAG latent projection \citep{richardson2002ancestral, zhang2008causal}, which is a different widely-used projection procedure. We also show how our results adapt to that case.

\subsection{Causal graphical time series models}\label{sec:background-time-series-models}
This paper draws its motivation from causally stationary structural vector autoregressive processes with acyclic contemporaneous interactions (see for example \citet{malinsky2018causal} and \citet{gerhardus2021characterization} for formal definitions). We employ these potentially multivariate and potentially non-linear $\mathbb{Z}$-indexed stochastic processes with causal meaning by considering them as structural causal models \citep{bollen1989structural, pearl2009causality, peters2017elements}. Moreover, we allow for dependence between the noise variables (also known as innovation terms). The property of causal stationarity then requires that both the qualitative cause-and-effect relationships as well as the qualitative dependence structure between the noise variables are invariant in time.

In particular, we are interested in the causal graphs (see for example \citet{pearl2009causality}) of such processes. These causal graphs represent the processes' qualitative cause-and-effect relationships by directed edges and non-zero dependencies between the noise variables of the processes by bidirected edges, and they have the following three special properties:
\begin{enumerate}
\item Since every vertex $(i, s) \in \Vertices$ corresponds to a particular time step $s$ of a particular component time series $X^i$, the vertex set factorizes as $\Vertices = \VarIndices \times \TimeIndices$ where $\VarIndices$ is the \emph{variable index set} and $\TimeIndices \subseteq \mathbb{Z}$ the \emph{time index set}. Thus, using the terminology of \citet{gerhardus2021characterization}, the causal graphs have \emph{time series structure}. The \emph{lag} of an edge $(i, s) \astast (j, u)$ is the non-negative integer $|u - s|$. An edge is \emph{lagged} if its lag is greater or equal than one, else it is \emph{contemporaneous}. We sometimes write $X^i_s$ instead of $(i, s)$ for clarity. 
\item Since causation cannot go back in time, the directed edges do not point back in time. That is, $(i, s) \tailhead (j, u)$ only if $s \leq u$. Thus, using the terminology of \citet{gerhardus2021characterization}, the causal graphs are \emph{time ordered}.\footnote{Contemporaneous edges are \emph{not} in conflict with the intuitive notion of time order: A cause in reality still precedes its effect, but the true time difference can be smaller than the observed time resolution.}
\item Due to causal stationarity, the edges are repetitive in time. Thus, using the terminology of \citet{gerhardus2021characterization}, the causal graphs have \emph{repeating edges}.
\end{enumerate}

\noindent Consequently, the causal graphs of interest are of the following type.

\begin{definition}[Time series ADMG, generalizing Definition 3.4 of \citet{gerhardus2021characterization}]\label{def:tsADMG-infinite}
A \emph{time series ADMG (ts-ADMG)} is an ADMG with time series structure (that is, $\Vertices = \VarIndices \times \TimeIndices$), that has time index set $\TimeIndices = \mathbb{Z}$, that is time ordered, and that has repeating edges.
\end{definition}

\begin{figure}[tb]
    \centering
    \includegraphics[scale=0.25]{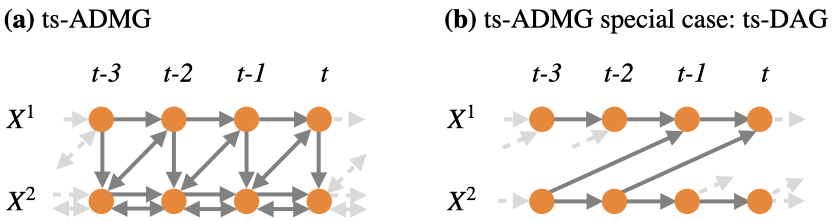}
    \caption{\textbf{(a)} Example of a ts-ADMG with variable index set $\VarIndices = \{1, 2\}$. \textbf{(b)} Example of a ts-ADMG that is also a ts-DAG, defined by the absence of bidirected edges.}
    \label{fig:tsADMG}
\end{figure}

\noindent Figure~\ref{fig:tsADMG} shows two ts-ADMGs for illustration. \emph{Time series DAGs (ts-DAGs)} \citep{gerhardus2021characterization} are ts-ADMGs without bidirected edges, see part (b) of Figure~\ref{fig:tsADMG} for an example. The absence of bidirected edges corresponds to the assumption of independent noise variables. In the literature, ts-DAGs are also known as \emph{time series chain graphs} \citep{dahlhaus2003causality}, \emph{time series graphs} \citep{runge2012escaping} and \emph{full time graphs} \citep{peters2017elements}. Throughout this paper, unless explicitly stated otherwise, we only consider ts-ADMGs that do not have self edges.

To stress the repetitive edge structure of ts-ADMGs, we often specify the time index of a vertex in relation to an arbitrary reference time step $t$, that is, we often write $(i, t-\tau)$ instead of, say, $(i, s)$. Let $\DirectedEdges^t = \{(i, t-\tau) \tailhead (j, t) \, \in \, \DirectedEdges \} \subsetneq \DirectedEdges$ be the subset of directed edges pointing into a vertex at time $t$, and let $\BidirectedEdges^t = \{(i, t-\tau) \headhead (j, t) \, \in \, \BidirectedEdges~|~ \tau \geq 0 \} \subsetneq \BidirectedEdges$ be the subset of bidirected edges between a vertex at time $t$ and a vertex at or before $t$. Then, the triple $(\VarIndices, \DirectedEdges^t, \BidirectedEdges^t)$ uniquely specifies a ts-ADMG.

Throughout this paper, unless explicitly stated otherwise, we impose two mild conditions on all ts-ADMGs: First, we require the variable index set $\VarIndices$ to be finite. On the level of the modeled time series processes, this requirement restricts to processes with finitely many component time series. Second, we require both $\DirectedEdges^t$ and $\BidirectedEdges^t$ to be finite sets (equivalently, we require all vertices to have finite in-degree). Given a finite variable index set, this second requirement is equivalent to the \emph{maximal lag $\ptsDAG$} defined as $\ptsDAG = \sup \, \{ |\tau_i - \tau_j| ~|~ \text{$(i,t -\tau_i) \astast (j, t-\tau_j)$ in $\Graph$}\} = \sup \, \{ |\tau_i| ~|~ \text{$(i,t -\tau_i) \asthead (j, t)$ in $\Graph$}\}$ being finite. On the level of the modeled time series processes, the second requirement thus restricts to processes of finite order.

The \emph{weight} $\weight{\pi}$ of a walk $\pi = ((i_1, s_1) \astast (i_2, s_2) \astast \ldots \astast (i_n, s_n))$ is the non-negative integer $|s_1 - s_n|$. If $\pi$ is of the form $\pi = ((i_1, s_1) \headtail (i_2, s_2) \headtail \ldots \headtail (i_n, s_n))$, then its weight $\weight{\pi} = s_1 - s_n$ equals the sum $\weight{\pi} = \sum_{k = 1}^{n-1} s_{k} - s_{k+1}$ of the lags $s_k - s_{k+1}$ of its edges; similarly for walks that are directed from $(i_1, s_1)$ to $(i_n, s_n)$. If $\pi_1$ is a directed walk from $(i, s_i)$ to $(j, s_j)$ and $\pi_2$ is a directed walk from $(j, s_j)$ to $(k, s_k)$, then $\weight{\pi_1} + \weight{\pi_2} = \weight{\pi}$ where $\pi$ is the directed walk from $(i, s_i)$ to $(k, s_k)$ obtained by appending $\pi_2$ to $\pi_1$ at their common vertex $(j, s_j)$.

\section{Finite marginal time series graphs}\label{sec:problem-formulation}
In this section, we motivate, define and start to approach the projection of ts-ADMGs to finite marginal graphs. To begin, Section~\ref{sec:projection-motivation} explains why the finite marginal graphs are useful for answering $m$-separation queries in the ts-ADMGs as well as for causal discovery and causal effect estimation in time series. Section~\ref{sec:projection-definition} then follows up with a formal definition of the finite marginal graphs, and Section~\ref{sec:reduction-past-confounding-paths} reduces the involved projection to the search for common ancestors in ts-DAGs.

\subsection{Motivation}\label{sec:projection-motivation}
When interpreted as a causal graph, a ts-ADMG entails various claims about the associated multivariate structural time series process.

An important type of claims are \textbf{independencies corresponding to m-separations} \citep{richardson2002ancestral, richardson2003markov}. For a \emph{finite} ADMG $\ADMG$, the \emph{causal Markov condition} \citep{spirtes2000causation} says that an $m$-separation $\mathbf{X} \perp\!\!\!\perp_{\ADMG} \mathbf{Y} ~\vert~ \mathbf{Z}$ in the graph $\ADMG$ implies the corresponding independence $\mathbf{X} \perp\!\!\!\perp \mathbf{Y} ~\vert~ \mathbf{Z}$ in all associated probability distributions \citep{verma1990causal, geiger1990identifying, richardson2003markov}. As opposed to that, for a ts-ADMG $\ADMG$, which is an \emph{infinite} graph, the same implication does not immediately follow: As an additional complication of the time series setting, it is non-trivial to say whether a given structural vector autoregressive process (cf.~first paragraph of Section~\ref{sec:background-time-series-models}) specifies a well-defined probability distribution---in the terminology of \citet{bongers2018causal}, whether the process admits a \emph{solution}---and what the properties of such solutions are. Thus, many works \emph{assume} the existence of a solution with the desired properties, see for example \citet{entner2010causal} and \citet{malinsky2018causal} in the context of causal discovery. According to \citet[Theorem 3.3. combined with Definition 2.1 and Example 2.2]{dahlhaus2003causality}, the causal Markov condition provably holds for the special case of stationary linear vector autoregressive processes with Gaussian innovation terms, see also  \citet[Theorem 1]{thams2022identifying}. What is important for our work here, independent of how one argues for the causal Markov condition, one still deals with the task of asserting $m$-separations $\mathbf{X} \perp\!\!\!\perp_{\ADMG} \mathbf{Y} ~\vert~ \mathbf{Z}$ in an infinite ts-ADMG $\ADMG$. To make this assertion, one must assert that in $\ADMG$ there is no path between $\mathbf{X}$ and $\mathbf{Y}$ that is active given $\mathbf{Z}$. However, as the ts-ADMG is infinite, there might be an infinite number of paths that could potentially be active. It is thus a non-trivial task to decide whether a given $m$-separation holds in a given ts-ADMG, and the authors are not aware of an existing general solution to it (also not in the special case of ts-DAGs). This task was the authors' original motivation for the presented study.

In this paper, we solve this task as follows: We develop an algorithm that performs the ADMG projection of (infinite) ts-ADMGs to finite marginal ADMGs on a finite time window $\TimeIndicesObserved = \{t-\tau ~|~ 0 \leq \tau \leq \ptimewindow\} \subsetneq \TimeIndices = \mathbb{Z}$ where $0 \leq \ptimewindow < \infty$. This projection algorithm implicitly solves the $m$-separation task because, if all vertices in $\mathbf{X} \cup \mathbf{Y} \cup \mathbf{Z}$ are within the time window $\TimeIndicesObserved$, then the $m$-separation $\mathbf{X} \perp\!\!\!\perp_{\ADMG} \mathbf{Y} ~\vert~ \mathbf{Z}$ holds in the (infinite) ts-ADMG if and only if it holds in the finite marginal ADMG. Of course, this approach merely shifts the difficulty to an equally non-trivial task: constructing the finite marginal ADMGs.

The ability to decide about $m$-separations is also necessary for causal reasoning. For example, most methods for \textbf{causal effect identification} in ADMGs---such as the (generalized) backdoor criterion \citep{pearl1993bayesian, pearl2009causality, maathuis2015generalized}, the ID-algorithm \citep{tian2002general, shpitser2006identification, shpitser2006identification_2, huang2006pearls} or graphical criteria to choose optimal adjustment sets \citep{runge2021necessary}---require the evaluation of $m$-separation statements in specific subgraphs of the causal graph. In order to apply these methods to ts-ADMGs with the goal of identifying time-resolved causal effects in structural time series processes,\footnote{Recall that the alternative approach of causal effect identification based on time-collapsed graphs is conceptually less straightforward and typically has less identification power.} one first needs a solution for evaluating $m$-separations in certain (still infinite) subgraphs of ts-ADMGs. More generally speaking, most of the graphical-model based causal inference framework applies to finite graphs, whereas specific modifications might be necessary for application to the (infinite) ts-ADMGs. Our approach of projecting ts-ADMGs to finite marginal ADMGs solves the graphical part of this problem, since one can equivalently check the relevant graphical criteria in the finite marginal ADMGs instead of the ts-ADMG. Therefore, our results constitute one step towards making large parts of the causal-graphical-model literature directly applicable to time series. The second step towards this goal, which is independent of our contribution and in general still open (cf.~the second paragraphs in the current subsection), is to more generally understand and prove in which cases the causal Markov condition holds for structural time series process.

\citet{thams2022identifying} is an example of a work that already uses finite marginals of an infinite time series graph for the purpose of causal effect identification. Specifically, that work uses finite ADMG latent projections of a ts-DAG (not a ts-ADMG) in the context of instrumental variable regression \citep{bowden1990instrumental} for time series. However, that work i) gives the finite marginals of only \emph{one} specific ts-DAG (which, in addition, is of the restricted special type discussed in Section~\ref{sec:special-case-all-lag-$1$-auto}) and ii) presents these finite marginal in an ad-hoc way. Contrary to that, in our paper, we i) consider general ts-ADMGs and ii) present a provably correct algorithm that constructs the finite marginal graphs.

Moreover, when using \textbf{causal discovery} approaches to learn (Markov equivalence classes of) ts-ADMGs from data, in practice, one is always restricted to a finite number of time steps. The natural targets of time-resolved time series causal discovery thus are (equivalence classes of) finite marginals of ts-ADMGs or ts-DAGs on finite time windows. In these finite marginal graphs, the time steps before the considered time window inevitably act as confounders.\footnote{To avoid this confounding by past time steps, one needs stronger assumptions. For example, as the causal discovery algorithms PCMCI \citep{runge2019detecting} and PCMCI$^+$ \citep{runge2020discovering} use, the assumption that no component time series are unobserved altogether (such that, effectively, one deals with a ts-DAG) in combination with a known or assumed upper bound on the maximal lag $\ptsDAG$ of that ts-DAG.} Indeed, the causal discovery algorithm tsFCI \citep{entner2010causal} aims to infer equivalence classes of finite marginal graphs that arise as DMAG projections of ts-DAGs, see \citet{gerhardus2021characterization} for a detailed explanation. Similarly, the causal discovery algorithms SVAR-FCI \citep{malinsky2018causal} and LPCMCI \citep{gerhardus2020high} aim to infer equivalence classes of subgraphs of these marginal DMAGs. A method for constructing finite marginal DMAGs from a given ts-ADMG or ts-DAG is, thus, needed to formally understand the target graphs of such time series causal discovery algorithms, which is the basis for an analysis of these algorithms. Further, with the ability to construct finite marginal DMAGs of ts-ADMGs one can even improve on the identification power of these state-of-the-art causal discovery algorithms, see Algorithm 1 in \citet{gerhardus2021characterization}. While here we only consider finite marginal ADMGs, in Section~\ref{sec:generalize-to-DMAGs} we extend our results to finite marginal DMAGs.

\subsection{Definition of marginal ADMGs}\label{sec:projection-definition}
The considerations in Section~\ref{sec:projection-motivation} motivate us to consider the ADMG projections of (infinite) ts-ADMGs to finite time windows $\TimeIndicesObserved = \{t-\tau ~|~ 0 \leq \tau \leq \ptimewindow\}$. In these projections, all vertices outside the observed time window $\TimeIndicesObserved$ are treated as unobserved. Hence, adopting the terminology of \citet{gerhardus2021characterization}, we call a vertex $(i, t-\tau)$ \emph{temporally observed} if $t-\tau \in \TimeIndicesObserved$ and else we call it \emph{temporally unobserved}.

In many applications where a ts-ADMG serves as a causal graphical model for a multivariate structural time series process, some of the component time series of the process might be unobserved altogether. We formalize this situation by introducing a partition $\VarIndices = \VarIndicesObserved \,\dot{\cup}\, \VarIndicesUnobserved$ of the ts-ADMG's variable index set $\VarIndices$ into the indices $\VarIndicesObserved$ of observed component time series and the indices $\VarIndicesUnobserved$ of unobserved component time series. Following \citet{gerhardus2021characterization}, we say that the component time series $X^i$ with $i \in \VarIndicesObserved$ and all corresponding vertices $(i, t-\tau)$ are \emph{observable}, whereas the component time series $X^i$ with $i \in \VarIndicesUnobserved$ and all corresponding vertices $(i, t-\tau)$ are \emph{unobservable}. For intuition, we often write $O^i$ (respl~$L^i$) instead of $X^i$ if $X^i$ is observable (resp.~unobservable). We require the set $\VarIndicesObserved$ to be non-empty, because else there would be no observed vertices, whereas $\VarIndicesUnobserved$ can but does not need to be empty.

Putting together these notions, the sets of observed and unobserved vertices respectively are $\ObservedVertices = \VarIndicesObserved \times \TimeIndicesObserved$ and $\LatentVertices = \Vertices \setminus \ObservedVertices = \left(\VarIndices \times \TimeIndices\right) \setminus \left(\VarIndicesObserved \times \TimeIndicesObserved\right)$. That is, a vertex is observed if and only if it is observable and temporally observed. We thus arrive at the following definition.

\begin{definition}[Marginal time series ADMG, adapting Definition 3.6 of \citet{gerhardus2021characterization}]\label{def:tsADMG-marginal}
Let $\ADMG$ be a ts-ADMG with variable index set $\VarIndices$, let $\VarIndicesObserved \subseteq \VarIndices$ be non-empty, let $\TimeIndicesObserved$ be $\TimeIndicesObserved = \{t-\tau ~|~ 0 \leq \tau \leq \ptimewindow\}$ and let $\ObservedVertices = \VarIndicesObserved \times \TimeIndicesObserved$. Then, its \emph{marginal time series ADMG (marginal ts-ADMG) $\tsADMG{\ObservedVertices}{\ADMG}$ on $\ObservedVertices$} is the ADMG $\ADMGprojection{\ObservedVertices}{\ADMG}$. We call the non-negative integer $\ptimewindow$ the \emph{observed time window length}.
\end{definition}

\begin{remark}
To avoid confusion between ts-ADMGs (which are infinite graphs, see Definition~\ref{def:tsADMG-infinite}) and marginal ts-ADMGs (which are finite graphs, see Definition~\ref{def:tsADMG-marginal}), we often attach the attribute ``infinite'' to the former (``infinite ts-ADMG'') and the attribute ``finite'' to the latter (``finite marginal ts-ADMG''). To avoid confusion, we stress that the \emph{number of temporally observed time steps} is $p+1$; for example, the observed time window $[t, t]$ has the observed time window length $p=0$ and $1 = p+1$ temporally observed time steps.
\end{remark}

\begin{figure}[tb]
    \centering
    \includegraphics[scale=0.25]{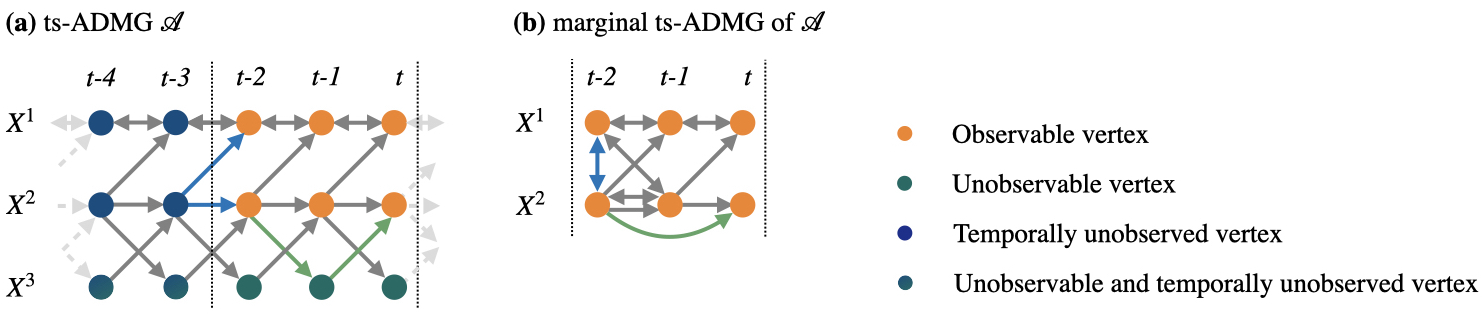}
    \caption{\textbf{(a)} An infinite ts-ADMG $\ADMG$. \textbf{(b)} The finite marginal ts-ADMG $\tsADMG{\ObservedVertices}{\ADMG}$ for $\VarIndicesObserved = \{1, 2\}$ and $\ptimewindow = 2$. The blue-colored path $X^1_{t-2} \headhead X^2_{t-3} \tailhead X^2_{t-2}$ (resp.~the green-colored path $X^2_{t-2} \tailhead X^3_{t-1} \tailhead X^2_{t}$) in $\ADMG$ induces the blue-colored edge $X^1_{t-2} \headhead X^2_{t-2}$ (resp.~the green-colored edge $X^2_{t-2} \tailhead X^2_t$) in $\tsADMG{\ObservedVertices}{\ADMG}$.}
    \label{fig:example-marginal-tsADMG}
\end{figure}

\noindent Figure~\ref{fig:example-marginal-tsADMG} shows an example of a ts-ADMG and a corresponding finite marginal ts-ADMG for illustration. We are now ready to formally state the goal and contribution of this paper.

\begin{problem}\label{problem:eventual-goal}
Develop an algorithm that, for a given triple of
\begin{itemize}
\item an arbitrary infinite ts-ADMG $\ADMG$,
\item a given non-empty set $\VarIndicesObserved$ of the observable component time series' variable indices and
\item a given length $\ptimewindow$ of the observed time window,
\end{itemize}
provably determines the finite marginal ts-ADMG $\tsADMG{\ObservedVertices}{\ADMG}$ in finite time. Here, $\ADMG$ is only subject to the following requirements: absence of self edges, finiteness of its variable index set $|\VarIndices| < \infty$ and finiteness of its maximal lag $\ptsDAG < \infty$.
\end{problem}

\begin{remark}
For a given triple $(\ADMG, \VarIndicesObserved, \ptimewindow)$ it is sometimes possible to manually find $\tsADMG{\ObservedVertices}{\ADMG}$ by ``looking at'' the paths in $\ADMG$. However, Problem~\ref{problem:eventual-goal} asks for a general algorithmic procedure that works for every choice of $(\ADMG, \VarIndicesObserved, \ptimewindow)$. We stress that we do not impose any restrictions other than those stated. In particular, we do not restrict $\ptimewindow$ in relation to the maximal lag $\ptsDAG$ of the ts-ADMG. Moreover, we do not impose connectivity assumptions on the ts-ADMG $\ADMG$. In particular, every component time series (including the unobserved component time series) is allowed but not required to be auto-dependent at any time lag. To the author's knowledge, Problem~\ref{problem:eventual-goal} remained unsolved prior to our work.
\end{remark}

Problem~\ref{problem:eventual-goal} is non-trivial because the set $\LatentVertices = \left(\VarIndices \times \TimeIndices\right) \setminus \left(\VarIndicesObserved \times \TimeIndicesObserved\right)$ of latent vertices includes the set $\LatentVerticesTemporally = \VarIndices \times \left\{(-\infty, t-\ptimewindow-1] \cup [t+1, \infty)\right\}$ of temporally unobserved vertices and, hence, is infinite. Therefore, for any given pair of observed vertices $(i, t-\tau_i)$ and $(j, t-\tau_j)$ there might be infinitely many paths in the infinite ts-ADMG $\ADMG$ that could potentially induce an edge $(i, t-\tau_i) \astast (j, t-\tau_j)$ in the finite marginal ts-ADMG $\tsADMG{\ObservedVertices}{\ADMG}$. As we will see below, the way around this complication is the repeating edges property of ts-ADMGs. This property allows to effectively restrict to a finite search space when combined with an evaluation of number-theoretic solvability problems.

\subsection{Reduction to common-ancestor search}\label{sec:reduction-past-confounding-paths}
In this section, we identify the missing ingredient for solving Problem~\ref{problem:eventual-goal} to be an algorithm for the exhaustive search of common ancestors in infinite ts-ADMGs. To this end, we reduce the marginal ts-ADMG projection in three steps. First, in Section~\ref{subsec:tsADMG-to-tsDAG}, we reduce the marginal ts-ADMG projection task for infinite ts-ADMGs to the projection task for the simpler infinite ts-DAGs. Second, in Section~\ref{subsec:ADMGtosimpleADMG}, we reduce the marginal ts-ADMG projection task for infinite ts-DAGs with arbitrary subsets of unobservable component time series to the projection task for infinite ts-DAGs without unobservable components. We refer to the map from an infinite ts-DAGs without unobservable component time series to its marginal ts-ADMG projection as the \emph{simple marginal ts-ADMG} projection. Third, in Section~\ref{sec:reduction-to-common-ancestor-search}, we reduce the simple marginal ts-ADMG projection task to the exhaustive search for common ancestors in infinite ts-DAGs. Figure~\ref{fig:reduction-steps} illustrates these reduction steps. We conclude with brief remarks on extensions in Section~\ref{subsec:projection-generalization}.

\begin{figure}[tb]
    \centering
    \includegraphics[scale=0.25]{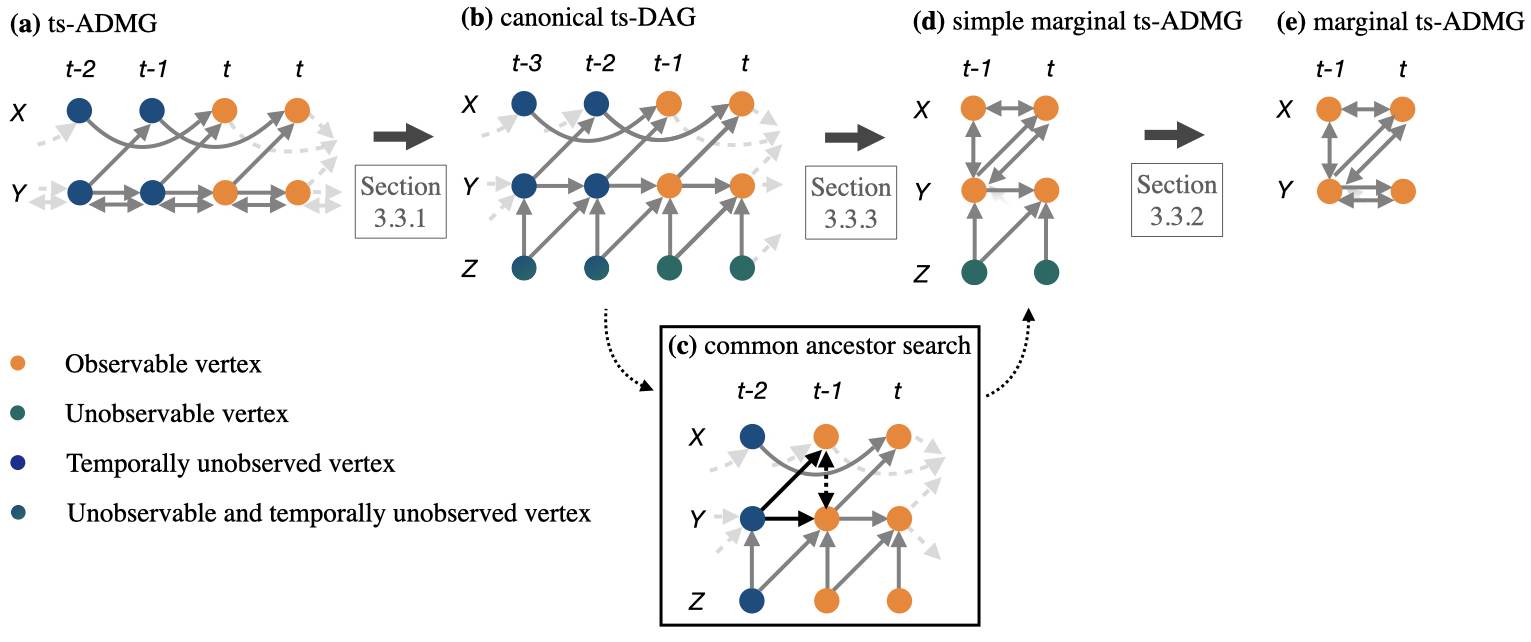}
    \caption{Illustration of the stepwise approach to go from an infinite ts-ADMG $\ADMG$ in part \textbf{(a)} to the finite marginal ts-ADMG $\tsADMG{\ObservedVertices}{\ADMG}$ in part \textbf{(e)}. First, see part \textbf{(b)} and Section~\ref{subsec:tsADMG-to-tsDAG}, we determine the infinite canonical ts-DAG $\CanonicaltsDAG{\ADMG}$ of the infinite ts-ADMG $\ADMG$. Second, see part \textbf{(d)} and Section~\ref{sec:reduction-to-common-ancestor-search}, we determine the simple finite marginal ts-ADMG $\tsADMG{\ObservedVertices^\prime}{\CanonicaltsDAG{\ADMG}}$ of the canonical ts-DAG $\CanonicaltsDAG{\ADMG}$. This steps involves a common-ancestor search in the canonical ts-DAG $\CanonicaltsDAG{\ADMG}$, see part \textbf{(c)}, which is the topic of Section~\ref{sec:refined-common-ancestor}. Third, see Section~\ref{subsec:ADMGtosimpleADMG}, we determine the finite marginal ts-ADMG $\tsADMG{\ObservedVertices}{\ADMG}$ from the simple marginal ts-ADMG $\tsADMG{\ObservedVertices^\prime}{\CanonicaltsDAG{\ADMG}}$ according to the equality of graphs $\tsADMG{\ObservedVertices}{\ADMG} = \ADMGprojection{\ObservedVertices}{\tsADMG{\ObservedVertices^\prime}{\CanonicaltsDAG{\ADMG}}}$.
    }
    \label{fig:reduction-steps}
\end{figure}

\subsubsection{Reduction to the marginal ts-ADMG projection of infinite ts-DAGs}\label{subsec:tsADMG-to-tsDAG}
To replace infinite ts-ADMGs with infinite ts-DAGs, we employ the following definition.

\begin{definition}[Canonical time series DAG, adapted from Definition 4.13 in \citet{gerhardus2021characterization}]\label{def:canonical-tsDAG}
Let $\ADMG = (\VarIndices \times \mathbb Z, \DirectedEdges, \BidirectedEdges)$ be a ts-ADMG. The \emph{canonical time series DAG (canonical ts-DAG) $\CanonicaltsDAG{\ADMG}$ of $\ADMG$} is the ts-DAG $\CanonicaltsDAG{\ADMG} = (\VarIndices^{ca}\times \mathbb Z, \DirectedEdges^{ca})$ where
\begin{itemize}
\item $\VarIndices^{ca} = \VarIndices \cup \mathbf{J}$ with $\mathbf{J} = \{(i, j, \tau)~\vert~ \text{$X^{i}_{t-\tau} \headhead X^j_t$ in $\ADMG$ with $\tau > 0$ or ($i < j$ and $\tau = 0$)}\}$ and
\item $\DirectedEdges^{ca} = \DirectedEdges \cup \{X^{(i, j, \tau)}_{s-\tau} \tailhead X^j_s ~\vert~ s \in \mathbb Z, (i, j, \tau)\in \mathbf J\} \cup \{X^{(i, j, \tau)}_{s} \tailhead X^i_s ~\vert~ s \in \mathbb Z, (i, j, \tau)\in \mathbf J\}$.
\end{itemize}
\end{definition}

\begin{figure}[tb]
    \centering
    \includegraphics[scale=0.25]{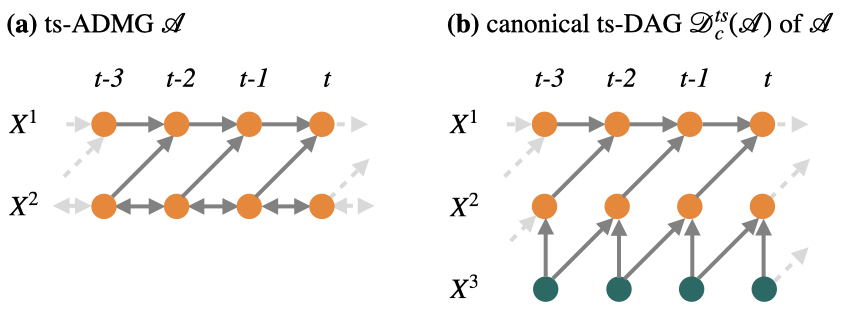}
    \caption{\textbf{(a)} A ts-ADMG. \textbf{(b)} The canonical ts-DAG of the ts-ADMG in (a).}
    \label{fig:canonical_tsDAG}
\end{figure}

\noindent Figure~\ref{fig:canonical_tsDAG} shows a canonical ts-DAG and the corresponding ts-ADMG for illustration. Intuitively, we obtain the canonical ts-DAG $\CanonicaltsDAG{\ADMG}$ by replacing all bidirected edges $X^j_{s-\tau} \headhead X^i_s$ of the ts-ADMG $\ADMG$ with paths $X^j_{s-\tau} \headtail X^{(i,j,\tau)}_{s-\tau} \tailhead X^i_s$ where the $X^{(i,j,\tau)}$ are auxiliary unobservable component time series. It follows that $\CanonicaltsDAG{\ADMG} = \ADMG$ if and only if $\ADMG$ is a ts-DAG. Moreover, acyclicity, time order and the property of repeating edges carry over from $\ADMG$ to $\CanonicaltsDAG{\ADMG}$, so the canonical ts-DAG $\CanonicaltsDAG{\ADMG}$ is indeed a ts-DAG. The definition also implies that, first, $\ADMG$ and $\CanonicaltsDAG{\ADMG}$ have the same directed paths and, second, there is a one-to-one correspondence between confounding paths through unobserved vertices in $\ADMG$ and the same type of paths in $\CanonicaltsDAG{\ADMG}$. Lastly, $\ADMG$ and $\CanonicaltsDAG{\ADMG}$ have the same $m$-separations among their shared vertices. By combining these observations, we arrive at the following result.

\begin{proposition}\label{prop:reduction-to-projection-of-tsDAG}
Let $\tsADMG{\ObservedVertices}{\ADMG}$ be the finite marginal ts-ADMG of the infinite ts-ADMG $\ADMG$ on the set $\ObservedVertices$ of observed vertices. Then, $\tsADMG{\ObservedVertices}{\ADMG} = \tsADMG{\ObservedVertices}{\CanonicaltsDAG{\ADMG}}$ where $\tsADMG{\ObservedVertices}{\CanonicaltsDAG{\ADMG}}$ is the finite marginal ts-ADMG of the infinite canonical ts-DAG $\CanonicaltsDAG{\ADMG}$ of $\ADMG$.
\end{proposition}

\subsubsection{Reduction to the simple marginal ts-ADMG projection}\label{subsec:ADMGtosimpleADMG}
The projection of an infinite ts-DAG $\DAG$ to its finite marginal ts-ADMG $\tsADMG{\ObservedVertices}{\DAG}$ on the set $\ObservedVertices = \VarIndicesObserved \times \TimeIndicesObserved$ of observed vertices marginalizes out all unobserved vertices $\LatentVertices = \Vertices \setminus \ObservedVertices$. This set of unobserved vertices consists of, first, all vertices that are either strictly before time $t-\ptimewindow$ or strictly after time $t$ (temporally unobserved) and, second, all unobservable vertices within the observed time window $[t-\ptimewindow, t]$ (temporally observed but unobservable). Accordingly, $\LatentVertices = \LatentVerticesTemporally \,\dot{\cup}\, \LatentVerticesUnobservable$ where $\LatentVerticesTemporally = \VarIndices \times \left\{(-\infty, t-\ptimewindow-1] \cup [t+1, \infty)\right\}$ and $\LatentVerticesUnobservable = \VarIndicesUnobserved \times [t-\ptimewindow, t]$.

This partition of $\LatentVertices$ is useful for our purpose because the ADMG latent projection commutes with partitioning the set of unobserved vertices. Thus, to determine $\tsADMG{\ObservedVertices}{\DAG}$, we can first marginalize $\DAG$ over $\LatentVerticesTemporally$ and then marginalize the resulting graph over $\LatentVerticesUnobservable$.

\begin{proposition}\label{prop:IgnoreFutureVertices}
Let $\DAG = (\VarIndices \times \mathbb{Z}, \DirectedEdges)$ be an infinite ts-DAG, let $\ObservedVertices = \VarIndicesObserved \times \TimeIndicesObserved$ with $\VarIndicesObserved \subseteq \VarIndices$ non-empty and $\TimeIndicesObserved = \{t-\tau ~|~ 0 \leq \tau \leq \ptimewindow\}$ where $\ptimewindow < \infty$, and let $\ObservedVertices^\prime = \VarIndices \times \TimeIndicesObserved$. Then, $\tsADMG{\ObservedVertices}{\DAG} = \ADMGprojection{\ObservedVertices}{\tsADMG{\ObservedVertices^\prime}{\DAG}}$ where $\ADMGprojection{\ObservedVertices}{\tsADMG{\ObservedVertices^\prime}{\DAG}}$ is the ADMG latent projection to $\ObservedVertices$ of the finite marginal ts-ADMG $\tsADMG{\ObservedVertices^\prime}{\DAG}$ of $\DAG$ on $\ObservedVertices^\prime$.
\end{proposition}

\noindent Crucially, the marginal ts-ADMG $\tsADMG{\ObservedVertices^\prime}{\DAG}$ is a finite graph. Hence, finding the ADMG latent projection $\ADMGprojection{\ObservedVertices}{\tsADMG{\ObservedVertices^\prime}{\DAG}}$ of $\tsADMG{\ObservedVertices^\prime}{\DAG}$ is a solved problem. To solve Problem~\ref{problem:eventual-goal}, it is thus sufficient to find an algorithm for projecting infinite ts-DAGs $\DAG$ to finite marginal ts-ADMGs $\tsADMG{\ObservedVertices^\prime}{\DAG}$ with $\ObservedVertices^\prime = \VarIndices \times \TimeIndicesObserved$, that is, for the special case of no unobservable component time series. We refer to this simplified projection as the \emph{simple marginal ts-ADMG projection}.

\subsubsection{Reduction to common-ancestor search in ts-DAGs}\label{sec:reduction-to-common-ancestor-search}
In the following, we first consider the directed edges and then the bidirected edges of  the simple marginal ts-ADMG $\tsADMG{\ObservedVertices^\prime}{\DAG}$. To recall, the projection of an infinite ts-DAG $\DAG$ to its simple finite marginal ts-ADMG $\tsADMG{\ObservedVertices^\prime}{\DAG}$ marginalizes over the set $\LatentVerticesTemporally = \VarIndices \times \left\{(-\infty, t-\ptimewindow-1] \cup [t+1, \infty)\right\}$ of temporally unobserved vertices only.

By point 1 in Definition~\ref{def:ADMG-latent-projection}, there is a directed edge $(i, t-\tau_i) \tailhead (j, t-\tau_j)$ in $\tsADMG{\ObservedVertices^\prime}{\DAG}$ if and only if in $\DAG$ there is at least one directed path $\pi = ((i, t-\tau_i) \tailhead \ldots \tailhead (j, t-\tau_j))$ such that all middle vertices on $\pi$, if any, are in $\LatentVerticesTemporally$. Since time order of $\DAG$ restricts such $\pi$ to the time window $[t-\tau_i, t-\tau_j] \subseteq \TimeIndicesObserved$ and since $\LatentVerticesTemporally$ contains only vertices outside of $\TimeIndicesObserved$, we see that $\pi = ((i, t-\tau_i) \tailhead (j, t-\tau_j))$ and thus get the following result.

\begin{proposition}\label{prop:directed-edges-tsADMG-simple}
Let $(i, t-\tau_i)$ and $(j, t-\tau_j)$ be vertices in $\tsADMG{\ObservedVertices^\prime}{\DAG}$ with $\ObservedVertices^\prime = \VarIndices \times \TimeIndicesObserved$. Then, $(i, t-\tau_i) \tailhead (j, t-\tau_j)$ in $\tsADMG{\ObservedVertices^\prime}{\DAG}$ if and only if $(i, t-\tau_i) \tailhead (j, t-\tau_j)$ in $\DAG$.
\end{proposition}

\noindent Thus, we can directly read off the directed edges of $\tsADMG{\ObservedVertices^\prime}{\DAG}$ from $\DAG$. The remaining task is to find the bidirected edges of $\tsADMG{\ObservedVertices^\prime}{\DAG}$.

By point 2 in Definition~\ref{def:ADMG-latent-projection}, there is a bidirected edge $(i, t-\tau_i) \headhead (j, t-\tau_j)$ in $\tsADMG{\ObservedVertices^\prime}{\DAG}$ if and only if in $\DAG$ there is at least one confounding path $\pi = ((i, t-\tau_i) \headast \ldots \asthead (j, t-\tau_j))$ such that all middle vertices on $\pi$, if any, are in $\LatentVerticesTemporally$. Recall that $\LatentVerticesTemporally$ consists of, first, all vertices strictly before time $t-\ptimewindow$ and, second, all vertices strictly after time $t$. Due to time order of $\DAG$ and the definitional requirement that confounding paths do not have colliders, such $\pi$ cannot contain vertices strictly after time $\max(t-\tau_i, t-\tau_j) \leq t$. We thus find that all middle vertices of $\pi$, if any, are strictly before $t-\ptimewindow$ and arrive at the following result.

\begin{proposition}\label{prop:bidirected-edges-tsADMG-simple}
Let $(i, t-\tau_i)$ and $(j, t-\tau_j)$ be vertices in $\tsADMG{\ObservedVertices^\prime}{\DAG}$ with $\ObservedVertices^\prime = \VarIndices \times \TimeIndicesObserved$. Then, $(i, t-\tau_i) \headhead (j, t-\tau_j)$ in $\tsADMG{\ObservedVertices^\prime}{\DAG}$ if and only if there are (not necessarily distinct) vertices $(k, t-\tau_k)$ with $\tau_k > \ptimewindow$ and $(l, t-\tau_l)$ with $\tau_l > \ptimewindow$ that have a common ancestor in $\DAG$.
\end{proposition}

\begin{remark}\label{remark:bidirected-edges-tsADMG-simple}
Recall that every vertex is its own ancestor. Consequently, $(k, t-\tau_k)$ and $(l, t-\tau_l)$ have a common ancestor if at least one of the following conditions is true:
\begin{itemize}
    \item The vertices $(k, t-\tau_k)$ and $(l, t-\tau_l)$ are equal.
    \item There is a directed path from $(k, t-\tau_k)$ to $(l, t-\tau_l)$.
    \item There is a directed path from $(l, t-\tau_l)$ to $(k, t-\tau_k)$.
    \item There is a confounding path between $(k, t-\tau_k)$ and $(l, t-\tau_l)$.
\end{itemize}
\end{remark}

Due to the repeating edges property of ts-DAGs, the vertices $(k, t-\tau_k)$ and $(l, t-\tau_l)$ have a common ancestor if and only if the time-shifted vertices $(k, t-\tau_k+\Delta t)$ and $(l, t-\tau_l+ \Delta t)$ with $\Delta t \in \mathbb Z$ have a common ancestor. Choosing $\Delta t = \min(\tau_k, \tau_l)$, we can place at least one of the time-shifted vertices at time $t$. We thus reduced the simple ts-ADMG projection and, by extension, Problem~\ref{problem:eventual-goal} to the following problem.

\begin{problem}\label{problem:common-ancestor}
Let $(i, t-\tau_i)$ and $(j, t)$, where $\tau_i$ is a non-negative integer, be two vertices in an arbitrary infinite ts-DAG $\DAG$ (subject only to $\ptsDAG < \infty$ and $\vert \VarIndices \vert < \infty$). Decide in finite time whether $(i, t-\tau_i)$ and $(j, t)$ have a common ancestor in $\DAG$.
\end{problem}

Problem~\ref{problem:common-ancestor} is still non-trivial because ts-DAGs extend to the infinite past and, hence, there might be infinitely many paths that can potentially give rise to a common ancestor of a given pair of vertices. In Section~\ref{sec:refined-common-ancestor}, we present a method that, despite this complication, answers common-ancestor queries in ts-DAGs in finite time.

To aid computational efficiency, we note the following: Proposition~\ref{prop:bidirected-edges-tsADMG-simple} together with the repeating edges property of ts-DAGs implies that $(i,t-\tau_i-\Delta t) \headhead (j, t-\tau_j-\Delta t)$ with $0 < \Delta t \leq \ptimewindow - \max(\tau_i, \tau_j)$ is in $\tsADMG{\ObservedVertices^\prime}{\DAG}$ if $(i,t-\tau_i) \headhead (j, t-\tau_j)$ is in $\tsADMG{\ObservedVertices^\prime}{\DAG}$. Hence, it is advantageous to first apply Proposition~\ref{prop:bidirected-edges-tsADMG-simple} to pairs of vertices $(i,t-\tau_i)$ and $(j, t-\tau_j)$ with small $\max(\tau_i, \tau_j)$: If we find that these vertices are connected by a bidirected edge, then we can automatically infer the existence of $\ptimewindow - \max(\tau_i, \tau_j)$ further bidirected edges.

\subsubsection{Extensions}\label{subsec:projection-generalization}
We can extend the above findings in two ways: First, in Section~\ref{sec:generalize-to-DMAGs} we show how to determine the finite marginal DMAG projection of an infinite ts-ADMG $\ADMG$ from the finite marginal ts-ADMG $\tsADMG{\ObservedVertices}{\ADMG}$. Second, given the ability to determine arbitrary finite marginal ts-ADMGs, we can also determine the ADMG latent projection of ts-ADMGs to \emph{arbitrary finite} sets $\ObservedVertices^{\prime\prime}$ of observed vertices. Indeed, let $\ObservedVertices$ be any set of the form $\ObservedVertices = \VarIndicesObserved \times \TimeIndicesObserved$ with $\ObservedVertices^{\prime\prime} \subseteq \ObservedVertices$. Then, because the ADMG latent projection commutes with partitioning the set of latent vertices, we can determine $\ADMGprojection{\ObservedVertices^{\prime\prime}}{\ADMG}$ as $\ADMGprojection{\ObservedVertices^{\prime\prime}}{\DAG} = \ADMGprojection{\ObservedVertices^{\prime\prime}}{\tsADMG{\ObservedVertices}{\ADMG}}$. Since $\tsADMG{\ObservedVertices}{\ADMG}$ is a finite graph, finding its projection $\ADMGprojection{\ObservedVertices^{\prime\prime}}{\tsADMG{\ObservedVertices}{\DAG}}$ is a solved problem.

\section{Common-ancestor search in ts-DAGs}\label{sec:refined-common-ancestor}
In this section, we present the three main results of this paper. These results solve Problem~\ref{problem:common-ancestor} (that is, the common-ancestor search in infinite ts-DAGs) and, hence, by extension also Problem~\ref{problem:eventual-goal} (that is, the construction of finite marginal ts-ADMGs of infinite ts-ADMGs). To begin, in Section~\ref{sec:special-case-all-lag-$1$-auto}, we restrict ourselves to ts-DAGs that have lag-$1$ auto-dependencies in every component and present a simple solution to Problem~\ref{problem:common-ancestor} in this special case. In Section~\ref{subsec:approaches-do-not-work}, we explain why the same approach does not work for general ts-DAGs and discuss two simple heuristics that do not work either. In Section~\ref{subsec:general-solution}, we present our solution to Problem~\ref{problem:common-ancestor}. To this end, we first reformulate Problem~\ref{problem:common-ancestor} in terms of \emph{multi-weighted directed graphs} (see Definition~\ref{Definitionm-w-graph} and Problem~\ref{prob1}). Theorem~\ref{cor.prob1} then reduces Problem~\ref{problem:common-ancestor} to the problem of deciding whether any of a finite number of \emph{linear Diophantine equations} has a non-negative integer solution. Moreover, utilizing standard results from number theory, we show how to answer these solvability questions in finite time by Theorem~\ref{thm:our-linear-diophantine}. The combination of Theorems~\ref{cor.prob1} and \ref{thm:our-linear-diophantine} thus solves Problem~\ref{problem:common-ancestor}. As a corollary, in Theorem~\ref{thm.upper-bound-main-theorem} we present an easily computable bound $\pcutoff(\DAG)$ such that the common-ancestor searches of Problem~\ref{problem:common-ancestor} can be restricted to the segment of the ts-DAG $\DAG$ inside the finite time window $[t-\pcutoff(\DAG),t]$. Thus, Theorem~\ref{thm.upper-bound-main-theorem} constitutes an alternative solution to Problem~\ref{problem:common-ancestor}.

\subsection{Special case of all lag-1 auto-dependencies}\label{sec:special-case-all-lag-$1$-auto}
As an important special case, we first restrict to ts-DAGs $\DAG$ that have lag-$1$ auto-dependencies everywhere, that is, to ts-DAGs $\DAG$ which for all $i \in \VarIndices$ have the edge $(i, t-1) \tailhead (i,t)$. To solve Problem~\ref{problem:common-ancestor} in this special case, we make use of the following well-known concept.

\begin{definition}[Summary graph, cf.~for example \citet{peters2017elements}]\label{def:summary-graph}
Let $\DAG = (\VarIndices \times \mathbb{Z}, \DirectedEdges)$ be a ts-DAG. The \emph{summary graph $\SummaryGraph{\DAG}$ of $\DAG$} is the directed graph with vertex set $\VarIndices$ such that $i \tailhead j$ in $\SummaryGraph{\DAG}$ if and only if there is at least one $\tau$ such that $(i, t-\tau) \tailhead (j, t)$ in $\DAG$.\footnote{Sometimes, see for example the definition in \citet{peters2017elements}, summary graphs specifically exclude self edges $i \tailhead i$. Here, we allow for self edges. Note that self-edges $i \tailhead i$ in the summary graph correspond to lagged auto-dependencies $(i,t-\tau) \tailhead (i, t)$ in the time-resolved graph (here, the ts-DAG).}
\end{definition}

\noindent Figure~\ref{fig:heuristic_summary_graph} illustrates the concept of summary graphs, which above we also referred to as time-collapsed graphs. Intuitively, we obtain the summary graph $\SummaryGraph{\DAG}$ by collapsing the ts-DAG $\DAG$ along its temporal dimension. Thus, although the ts-DAG $\DAG$ is acyclic by definition, its summary graph $\SummaryGraph{\DAG}$ can be cyclic. Due to our standing assumption that $|\VarIndices|<\infty$, the summary graph $\SummaryGraph{\DAG}$ is a finite graph and all path searches in it will terminate in finite time. This fact makes summary graphs a promising tool for answering Problem~\ref{problem:common-ancestor}. Indeed, whenever there is a confounding path (resp.~a directed path) between $(i, t-\tau_i)$ and $(j, t)$ in the ts-DAG $\DAG$, then the projection of that path to $\SummaryGraph{\DAG}$---obtained by forgetting the time indices of the vertices on the path---is a confounding walk (resp.~a directed walk) from $i$ to $j$ in $\SummaryGraph{\DAG}$. In the special case of all lag-$1$ auto-dependencies, also the converse is true and we obtain the following result.

\begin{proposition}\label{prop:common-ancestors-all-lag-$1$-are-there}
Let $\DAG$ be a ts-DAG that for all $i \in \VarIndices$ has the edge $(i,t-1) \tailhead (i,t)$, and let $(i, t-\tau_i)$ with $\tau_i \geq 0$ and $(j, t)$ be two vertices in $\DAG$. Then, $(i, t-\tau_i)$ and $(j, t)$ have a common ancestor in $\DAG$ if and only if the vertices $i$ and $j$ have a common ancestor in $\SummaryGraph{\DAG}$.
\end{proposition}

We can understand the if-part of Proposition~\ref{prop:common-ancestors-all-lag-$1$-are-there} as follows: Suppose there is a confounding path between $i$ and $j$ in $\SummaryGraph{\DAG}$. By splitting this confounding path at its (unique) root vertex $k$, we obtain a directed path $\pi_{ki}$ from $k$ to $i$ and a directed path $\pi_{kj}$ from $k$ to $j$ in $\SummaryGraph{\DAG}$. By definition of the summary graph and the repeating edges property of ts-DAGs, in the ts-DAG $\DAG$ there thus are a directed path $\rho_{ki}$ from $(k, t-\tau_{ki})$ to $(i,t-\tau_i)$ for some $\tau_{ki} \geq \tau_i$ and a directed path $\rho_{kj}$ from $(k, t-\tau_{kj})$ to $(j, t)$ for some $\tau_{kj} \geq 0$. If $\tau_{kj} = \tau_{ki}$, then we can directly join $\rho_{ki}$ and $\rho_{kj}$ at their common vertex $(k, t-\tau_{ki}) = (k, t-\tau_{kj})$ to obtain a confounding walk $\rho$. If $\tau_{kj} \neq \tau_{ki}$, then an additional step is needed: Due to the assumption that $\DAG$ has all lag-$1$ auto-dependencies, there is the path $\rho_{kk} = ((k, t-\max(\tau_{ki},\tau_{ki})) \tailhead (k, t-\max(\tau_{ki},\tau_{ki})+1) \tailhead \ldots \tailhead (k, t-\min(\tau_{ki},\tau_{ki}))$ in the ts-DAG $\DAG$. By appropriately concatenating the paths $\rho_{ki}$, $\rho_{kj}$ and $\rho_{kk}$, we obtain a confounding walk $\rho$. If $\rho$ is a path, we thus showed that $(i, t-\tau_i)$ and $(j, t)$ have a common ancestor in $\DAG$. If $\rho$ is not a path, then we can remove a sufficiently large subwalk from $\rho$ to obtain a path $\tilde{\rho}$ which is either a confounding path or directed from $(i, t-\tau_i)$ to $(j, t)$ or directed from $(j, t)$ to $(i, t-\tau_i)$ or trivial. Thus, $(i, t-\tau_i)$ to $(j, t)$ have a common ancestor in $\DAG$ (see Remark~\ref{remark:bidirected-edges-tsADMG-simple} for clarification of what a common ancestor is). The cases in which $i=j$ or in which $i$ and $j$ are connected by a directed path (cf.~Remark~\ref{remark:bidirected-edges-tsADMG-simple}) follow similarly.

While ts-DAGs with lag-$1$ auto-dependencies in all component time series are an important special case, restricting to that case is \emph{not} standard in the literature. For example, none of the causal discovery works \citet{entner2010causal}, \citet{malinsky2018causal} or \citet{gerhardus2020high} makes this assumption. Conversely, \citet{mastakouri2020necessary} even specifically assumes a subset of the components to not have lag-$1$ auto-dependencies. Thus, we are not satisfied with Proposition~\ref{prop:common-ancestors-all-lag-$1$-are-there} and move to the general case.

\subsection{Approaches that do not work in general} \label{subsec:approaches-do-not-work}
We now turn to general ts-DAGs, that is, we no longer require that $(i,t-1) \tailhead (i, t)$ for all $i \in \VarIndices$. In Sections~\ref{subsec:summary-graph-insufficient} and \ref{subsect:simple-heuristics-insufficient}, we consider two approaches that might appear promising at first but are in general not sufficient to decide about common ancestorship in ts-DAGs.

\subsubsection{Common-ancestor search in the summary graph}\label{subsec:summary-graph-insufficient}
The following example shows that, for general ts-DAGs, common ancestorship of $i$ and $j$ in the summary graph $\SummaryGraph{\DAG}$ does not necessarily imply common ancestorship of $(i,t-\tau_i)$ and $(j, t)$ in the ts-DAG $\DAG$.

\begin{example}\label{example:with-summary-graph}
In the summary graph in part (b) of Figure~\ref{fig:heuristic_summary_graph}, the vertex $Y$ is a common ancestor of $X$ and $Z$. However, in the corresponding ts-DAG in part (a) of Figure~\ref{fig:heuristic_summary_graph}, the vertices $X_t$ and $Z_{t-1}$ do not have a common ancestor.
\end{example} 

\begin{figure}[tb]
    \centering
    \includegraphics[scale=0.25]{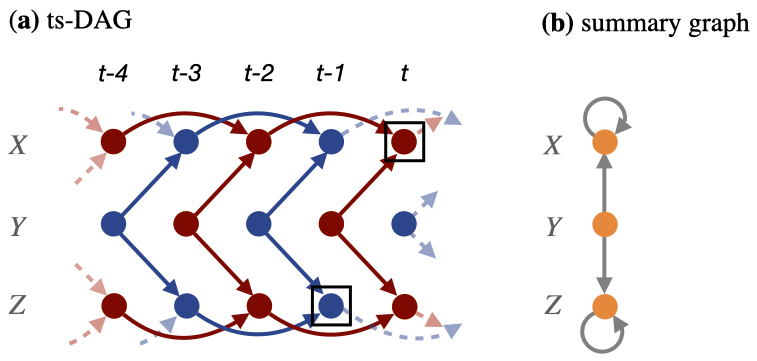}
    \caption{A ts-DAG and its summary graph, respectively see part \textbf{(a)} and part \textbf{(b)}. See also Example~\ref{example:with-summary-graph}.}
    \label{fig:heuristic_summary_graph}
\end{figure}

To understand why Proposition~\ref{prop:common-ancestors-all-lag-$1$-are-there} does not generalize to general ts-DAGs, we recall its justification from the last paragraph of Section~\ref{sec:special-case-all-lag-$1$-auto}: The argument there involved the path $\rho_{kk} = ((k, t-\max(\tau_{ki},\tau_{ki})) \tailhead (k, t-\max(\tau_{ki},\tau_{ki})+1) \tailhead \ldots \tailhead (k, t-\min(\tau_{ki},\tau_{ki}))$. This path only exists if there is the edge $(k,t-1) \tailhead (k,t)$, which need not be the case for general ts-DAGs. Put differently, the existence of a confounding path between $i$ and $j$ in $\SummaryGraph{\DAG}$ does imply the existence of a path $\rho_{ki}$ from $(k, t-\tau_{ki})$ to $(i,t-\tau_i)$ for some $\tau_{ki} \geq \tau_i$ and a directed path $\rho_{kj}$ from $(k, t-\tau_{kj})$ to $(j, t)$ for some $\tau_{kj} \geq 0$. But, in general, there is no pair of such paths with $\tau_{ki} = \tau_{kj}$, that is, a pair of such paths which start at the same time step.

The opposite direction, however, still holds: Common ancestorship of $(i,t-\tau_i)$ and $(j, t)$ in $\DAG$ does imply common ancestorship of $i$ and $j$ in $\SummaryGraph{\DAG}$. Thus, common ancestorship in $\SummaryGraph{\DAG}$ is a necessary but not sufficient condition for common ancestorship in $\DAG$.

\subsubsection{Simple heuristics for a finite time window}\label{subsect:simple-heuristics-insufficient}
A different approach would proceed as follows: Given a pair of vertices $(i,t-\tau_i)$ and $(j, t)$ with $0 \leq \tau_i \leq \ptimewindow$, search for common ancestors of that pair within a sufficiently large but finite time window $[t-\pcutoff(\DAG),t]$ of the ts-DAG $\DAG$. While for every fixed $\DAG$ there indeed is some non-negative integer $\pcutoff(\DAG)$ such that $(i,t-\tau_i)$ and $(j, t)$ have a common ancestor in $\DAG$ if and only if they have a common ancestor within the time window $[t-\pcutoff(\DAG),t]$, this fact itself is of little practical use: To solve Problem~\ref{problem:common-ancestor}, one would need to know such an integer $\pcutoff(\DAG)$ \emph{a priori}. Therefore it might be tempting to come up with heuristics for choosing $\pcutoff(\DAG)$ for a given $\DAG$, such as the following:

\begin{itemize}
    \item Let $\pcutoff(\DAG) = \ptimewindow + \pcutoff_{sum}(\DAG)$, where $\ptimewindow$ is the observed time window length and $\pcutoff_{sum}(\DAG)$ is the sum of all lags of edges in $\DirectedEdges^t \cup \BidirectedEdges^t$.
    \item Let $\pcutoff(\DAG) = \ptimewindow + \pcutoff_{prod}(\DAG)$, where $\ptimewindow$ is the observed time window length and $\pcutoff_{prod}(\DAG)$ is the product of all non-zero lags of edges in $\DirectedEdges^t \cup \BidirectedEdges^t$.
\end{itemize}

\noindent However, Examples~\ref{ref:counterexample-1} and \ref{ref:counterexample-2} in Section~\ref{sec:counterexamples} show that neither of these simple heuristics work in general. What is more, even if one were to come up with a heuristic for which one would not find a counterexample, then one would still have to prove this heuristic.

In Section~\ref{sec:formula-cutoff}, we will eventually prove a formula for a sufficiently large $\pcutoff(\DAG)$ by considering bounds for minimal non-negative integer solutions of linear Diophantine equations. However, this formula is far from obvious and we expect an explicit common-ancestor search in the corresponding time window to be computationally more expensive than a direct application of the intermediate number-theoretic result from which the formula is derived.

\subsection{General solution to the common-ancestor search in ts-DAGs} \label{subsec:general-solution}
In this section, we present the main results of Section~\ref{sec:refined-common-ancestor}. These results provide a general solution to Problem~\ref{problem:common-ancestor}.

\subsubsection{Time series DAGs as multi-weighted directed graphs} \label{subsubsec:multi-weighted}
In order to express our results on the common-ancestor search in ts-DAGs efficiently, we reformulate the information contained in ts-DAGs and, correspondingly, Problem~\ref{problem:common-ancestor} in the language of \emph{multi-weighted directed graphs}. This reformulation gives us access to standard techniques of discrete graph combinatorics. We begin with the definition of multi-weighted directed graphs.

\begin{definition}[Multi-weighted directed graph] \label{Definitionm-w-graph}
A \emph{multi-weighted directed graph (MWDG)} is a tuple $(\cG,\w)$ of a directed graph $\cG = (\VarIndices,\DirectedEdges)$ and a collection $\w = (\weightset{e})_{e \in \DirectedEdges}$ of multi-weights $\weightset{e}$. The multi-weights $\emptyset \neq \weightset{e} \subsetneq \mathbb N_0$ are \emph{finite} non-empty sets of non-negative integers. We call $w(e) \in \weightset{e}$ a \emph{weight} of $e$.
\end{definition}

\noindent A multi-weighted directed graph $(\cG,\w)$ is \emph{weakly acyclic} if it satisfies both of the following conditions:
\begin{enumerate}
\item If $e = i \tailhead i$, then $0 \notin \weightset{e}$.
\item The directed graph $\Graph^0 = (\VarIndices, \DirectedEdges^0)$ with $\DirectedEdges^0 = \{e \in \DirectedEdges ~\vert~ 0 \in \weightset{e} \}$ is acyclic.
\end{enumerate}
For $\pi$ a trivial or directed walk in a MWDG $\cG$, we define its multi-weight $\weightset{\pi}$ as
\begin{align*}
\weightset{\pi} = \begin{cases}
\left\{ \sum_{e \in \mathbf{E}_{\pi}} w(e) ~\middle\vert~ w(e) \in \weightset{e}\right\} \equiv \sum_{e \in \mathbf{E}_{\pi}} \weightset{e} \quad &\text{if $\pi$ is directed} \\
\{0\} \quad &\text{if $\pi$ is trivial}
\end{cases}
\end{align*}
where $\mathbf{E}_{\pi}$ denotes the sequence of edges on $\pi$. We call $w(\pi) \in \weightset{\pi}$ a \emph{weight} of $\pi$. To abbreviate notation below, given two sets $A,B \subseteq \mathbb N_0$ and a non-negative integer $n$ we define
\begin{align*}
A + B = \left\{ a +b ~\vert~ \in A, \ b \in B \right\} \quad \text{and} \quad n \cdot A = \left\{n \cdot a ~\vert~ a \in A \right\} \, .
\end{align*}

In order to understand the connection between ts-DAGs and weakly acyclic MWDGs, we employ the following definition.
\begin{definition} \label{Definitionm-w-summary-graph}
Let $\DAG = (\VarIndices \times \mathbb{Z}, \DirectedEdges)$ be a ts-DAG. Then, the \emph{multi-weighted summary graph $\mathcal{S}_{\w}(\DAG)$ of $\DAG$} is the multi-weighted directed graph $\mathcal{S}_{\w}(\DAG) = (\mathcal{S}(\DAG), \w)$ where
\begin{itemize}
\item $\mathcal{S}(\DAG)$ is the summary graph of $\DAG$ in the sense of Definition~\ref{def:summary-graph} and 
\item $\w = (\weightset{i \tailhead j})_{i \tailhead j \in \DirectedEdges}$ with $\weightset{i \tailhead j} = \{ \tau ~\vert~ (i, t-\tau) \tailhead (j, t) \in \DAG\}$.
\end{itemize}
\end{definition}

\begin{figure}[tb]
    \centering
    \includegraphics[scale=0.25]{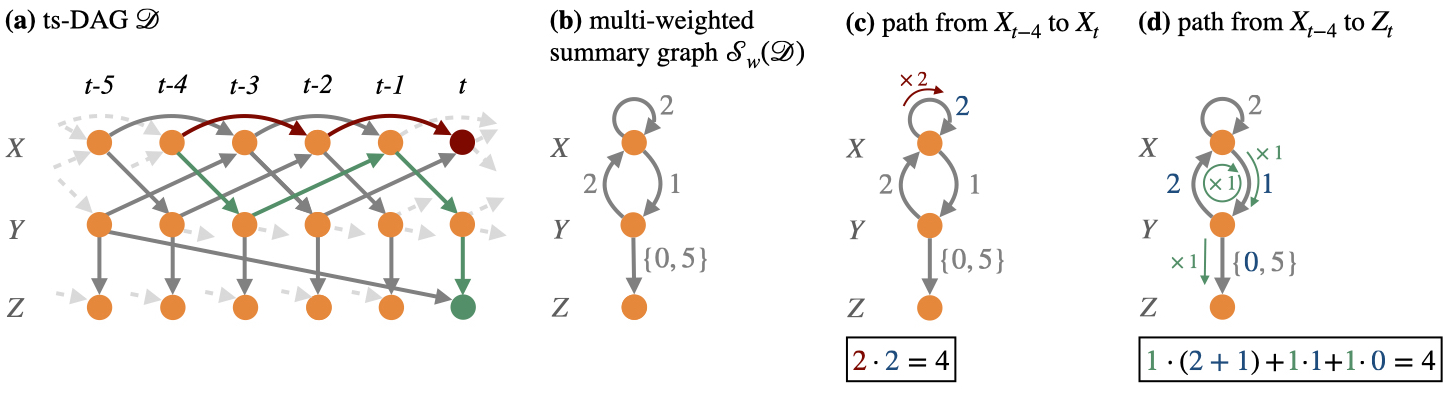}
    \caption{\textbf{(a)} A ts-DAG in which we wish to find common ancestor of $X_t$ and $Z_t$. \textbf{(b)} The multi-weighted summary graph of this ts-DAG. We display the multi-weights next to the corresponding edges; for example, $\w(Y \tailhead Z) = \{0, 5\}$ because in the ts-DAG there are the edges $Y_t \tailhead Z_t$ and $Y_{t-5} \tailhead Z_t$. For multi-weights with only a single element, for example $\w(X \tailhead X) = \{2\}$, we omit the set brackets for simplicity. \textbf{(c)} Illustration of how the red-colored path $\pi_1 = (X_{t-4} \tailhead X_{t-2} \tailhead X_t)$ in the ts-DAG maps to the multi-weighted summary graph, including the information on the weight $\weight{\pi_1} = 4 \in \w(\pi_1)$ of this path. \textbf{(d)} Illustration of how the green-colored path $\pi_2 = (X_{t-4} \tailhead Y_{t-3} \tailhead X_{t-2} \tailhead Y_t \tailhead Z_t)$ in the ts-DAG maps to the multi-weighted summary graph, including the information on the weight $\weight{\pi_2} = 4 \in \w(\pi_2)$ of this path. See also Examples~\ref{example:running-example-1}, \ref{example:running-example-2}, \ref{example:running-example-3} and \ref{example:running-example-4}.}
    \label{fig:multi_weighted_summary_graph}
\end{figure}

\noindent Parts (a) and (b) of Figure~\ref{fig:multi_weighted_summary_graph} illustrate the concept of multi-weighted summary graphs. Due to acyclicity and the absence of self edges, the multi-weighted summary graph $\mathcal{S}_{\w}(\DAG)$ of a ts-DAG $\DAG$ is a weakly acyclic MWDG. Conversely, a weakly acyclic MWDG $(\Graph, \w)$ where $\Graph = (\VarIndices, \DirectedEdges^{\Graph})$ induces the ts-DAG $\DAG = (\VarIndices \times \mathbb{Z}, \DirectedEdges)$ where $\DirectedEdges = \{(i, t-\tau) \tailhead (j, t) ~\vert~ (i, j) \in \DirectedEdges^{\Graph}, \tau \in \w(i\tailhead j), t\in \mathbb{Z}\}$. These mappings between ts-DAGs and weakly acyclic MWDGs are one-to-one and inverses of each other. Consequently, we can re-express every statement about ts-DAGs on the level of weakly acyclic MWDGs. The following result gives the corresponding reformulation of common ancestorship in ts-DAGs.

\begin{proposition}\label{prop:common-ancestorship-reexpression}
Let $\DAG$ be a ts-DAG, and let $(i,t-\tau)$ and $(j,t)$ with $\tau \geq 0$ be two vertices in $\DAG$. Then, $(i,t-\tau)$ and $(j,t)$ have a common ancestor in $\DAG$ if and only if in the multi-weighted summary graph $\mathcal{S}_{\w}(\DAG)$ of $\DAG$ there are trivial or directed walks $\pi$ and $\pi^\prime$ that satisfy all of the following conditions:
\begin{enumerate}
\item Their first vertex is equal, that is, $\pi(1) = \pi'(1)$. 
\item They respectively end at the vertices $i$ and $j$, that is, $\pi(\length{\pi}) = i$ and $\pi'(\length{\pi'}) = j $.
\item They have weights $w(\pi) \in \weightset{\pi}$ and $w(\pi') \in \weightset{\pi'}$ with $w(\pi)+\tau = w(\pi')$.
\end{enumerate}
\end{proposition}

Using Proposition~\ref{prop:common-ancestorship-reexpression}, we can reformulate Problem~\ref{problem:common-ancestor} on the level of MWDGs as follows.

\begin{problem} \label{prob1}
Let $(\cG,\w)$ be a weakly acyclic MWDG, let $\tau \in \mathbb{N}_0$, and let $i$, $j$ and $k$ be (not necessarily distinct) vertices in $\cG$. Decide in finite time whether
\begin{align*}
\underbrace{\left(\bigcup_{\pi \in \cW(k,i)} \weightset{\pi} + \{ \tau \} \right) }_{\equiv \mathbf{W}_{\tau}(k,i)}\cap \underbrace{\left(\bigcup_{\pi' \in \cW(k,j)} \weightset{\pi'}\right)}_{\equiv \mathbf{W}_{0}(k,j)} \neq \emptyset.
\end{align*}
Here, $\cW(k,i\neq k)$ is the set of directed walks in $\cG$ from $k$ to $i$ and $\cW(k,i=k)$ is the set of directed walks in $\cG$ from $k$ to $i=k$ plus the trivial walk $(i)$; similarly for $\cW(k,j)$.
\end{problem}

\noindent The following example illustrates Proposition~\ref{prop:common-ancestorship-reexpression} and Problem~\ref{prob1}.

  \begin{example}\label{example:running-example-1}
Consider the ts-DAG in part (a) of Figure~\ref{fig:multi_weighted_summary_graph}. In this graph, $X_{t-4}$ is an ancestor of $X_t$ through the red-colored directed path $\rho = (X_{t-4} \tailhead X_{t-2} \tailhead X_t)$ and an ancestor of $Z_t$ through the green-colored directed path $\rho^\prime = (X_{t-4} \tailhead Y_{t-3} \tailhead X_{t-1} \tailhead Y_t \tailhead Z_t)$. Consequently, $X_{t-4}$ is a common ancestor of $X_t$ and $Z_t$ in the ts-DAG. Now turn to the corresponding multi-weighted summary graph in part (b) of the same figure. In this graph, as parts (c) and (d) of the same figure illustrate, there are the corresponding directed walks $\pi = (X \tailhead X \tailhead X)$ and $\pi^\prime = (X \tailhead Y \tailhead X \tailhead Y \tailhead Z)$. Note that we can respectively map $\rho$ and $\rho^\prime$ to $\pi$ and $\pi^\prime$ by removing the time indices from the vertices on $\rho$ and $\rho^\prime$. For the multi-weights of $\pi$ and $\pi^\prime$ we find $\w(\pi) = 2\cdot \w(X \tailhead X) = 2 \cdot \{2\} = \{4\}$ and $\w(\pi^\prime) = 2 \cdot \w(X\tailhead Y) + 2 \cdot \w(Y\tailhead X) + 2 \cdot \w(Y \tailhead Z) = 2 \cdot \{2\} + 2 \cdot \{1\} + 1 \cdot \{0, 5\} = \{4, 9\}$. Thus, setting $i = X$, $j = Z$ and $\tau = 0$, we see that $\pi$ and $\pi^\prime$ satisfy all three properties listed in Proposition~\ref{prop:common-ancestorship-reexpression} (with $\weight{\pi} = \weight{\pi} +\tau  = \weight{\pi^\prime} = 4$). Moreover, setting $k = X$, we see that the set intersection in Problem~\ref{prob1} is indeed non-empty.
\end{example}

\subsubsection{Solution of Problem~\ref{prob1}} \label{subsubsec.solution-to-problem}
To express our solution of Problem \ref{prob1}, given a tuple $(a_0,a_1,\dots,a_{\mu})$ where $0 \leq \mu < \infty$ and $a_0 \in \mathbb{N}_0$ and $a_\alpha \in \mathbb{N}$ for all $1 \leq \alpha \leq \mu$, we define the affine convex cone over non-negative integers of that tuple as
\begin{align}\label{eq:cones}
\mathrm{con}(a_0;a_1,\dots,a_{\mu}) = \left\{ a_0 + \sum_{\alpha=1}^{\mu} n_{\alpha} \cdot a_{\alpha} \ \bigg| \ n_{\alpha} \in \mathbb{N}_0 \text{ for all } 1 \leq \alpha \leq \mu \right\}.
\end{align}
Our solution of Problem \ref{prob1} rests on the following lemma that specifies a finite decomposition of the sets $\mathbf{W}_{\tau}(k,i)$ into such cones.

\begin{lemma} \label{thm.main1}
Let $(\cG,\w)$ be a weakly acyclic MWDG, let $\tau \in \mathbb{N}_0$, and let $i$ and $k$ be (not necessarily distinct) vertices in $\cG$. Then, there exist \emph{finite} sets
\begin{itemize}
\item $\cW_0(k,i)$,
\item $\mathcal{M}_{\pi}$ for every $\pi \in \cW_0(k,i)$,
\item and $D_{\tau}(\pi,\mathcal S)$ for every $\pi \in \cW_0(k,i)$ and every $\mathcal{S} \in \mathcal{M}_{\pi}$
\end{itemize}
such that
\begin{align}
\mathbf{W}_{\tau}(k,i) = \bigcup_{\pi\in \cW_0(k,i)}\, \bigcup_{\mathcal{S} \in \mathcal{M}_{\pi}} \,\bigcup_{(a_0,\dots,a_{\mu}) \in D_{\tau}(\pi,\mathcal S)}\mathrm{con}(a_0;a_1,\dots,a_{\mu}) \, . \label{eq:main-decomposition}
\end{align}
\end{lemma}

\noindent Here, $\cW_0(k, i \neq k )$ is the finite set of cycle-free directed walks from $k$ to $i$ in the finite graph $\cG$ and $\cW_0(k,i=k) = \{(k)\}$ is the single-element set consisting of the trivial walk $(k)$. As their construction is more involved, we postpone the definitions of the sets $\mathcal{M}_{\pi}$ and $D_{\tau}(\pi,\mathcal S)$ to the proof of Lemma~\ref{thm.main1} in Section~\ref{subsec.proof-of-decomposition-theorem}. Instead, we here give a intuitive justification by means of Remark~\ref{remark:theorem-1} (readers can skip this remark without loosing the conceptual flow of this section).

\begin{remark}\label{remark:theorem-1}
By definition of the set $\mathbf{W}_{\tau}(k,i)$, we see that $w \in \mathbf{W}_{\tau}(k,i)$ if and only if there is a directed walk $\pi_{cyc}$ from $k$ to $i$ (or, if $k=i$, a trivial walk $\pi_{cyc}$ consisting of $k=i$ only) such that $w = \tau + w(\pi_{cyc})$ with $w(\pi_{cyc}) \in \w(\pi_{cyc})$. We now face the complication that, since the graph $\Graph$ can be cyclic (recall that the summary graph $\mathcal{S}(\DAG)$ of a ts-DAG $\DAG$ can be cyclic), there potentially is an infinite number of directed walks. This potential infiniteness is the manifestation of the potentially infinite number of paths that might give rise to common ancestorship in ts-DAGs as faced when approaching the problem in the formulation of Problem~\ref{problem:common-ancestor}. The idea to work around this complication is as follows: Given a in general cyclic directed walk $\pi_{cyc}$ from $k$ to $i$, we can map $\pi_{cyc}$ to a cycle-free path $\pi \in \cW_0(k,i)$ from $k$ to $i$ by ``collapsing'' the cycles on $\pi_{cyc}$.\footnote{In general, $\pi$ depends on the order in which we ``collapse'' the cycles on $\pi_{cyc}$. This non-uniqueness is, however, not relevant to the argument.} Conversely, we can obtain $\pi_{cyc}$ from $\pi$ by ``inserting'' certain cycles a certain number of times in a certain order. Consequently, there are non-negative integers $n_{1}, \ldots, n_{\mu}$ where $0 \leq \mu < \infty$ such that $\w(\pi_{cyc}) \subseteq \w(\pi) + n_1 \cdot \w(c_1) + \dots + n_\mu \cdot \w(c_\mu)$ where $c_1, \ldots, c_\mu$ are cycles in $\Graph$. Since we can obtain every cycle by appropriately combining irreducible cycles, we can without loss of generality assume that the $c_1, \ldots, c_\mu$ are irreducible. We now wish to set $a_0 = \tau + w(\pi)$ where $w(\pi) \in \w(\pi)$ and $a_\alpha = w(c_{\alpha})$ where $w(c_{\alpha}) \in \w(c_{\alpha})$ for all $1 \leq \alpha \leq \mu$ and take the union over $\mathrm{con}(a_0;a_1,\dots,a_{\mu})$ for all such tuples $(a_0;a_1,\dots,a_{\mu})$. There are at most finitely many of such tuples because every walk has a finite set of multi-weights, and $a_\alpha \neq 0$ due to weak acyclicity. This choice of tuples is correct if all of the irreducible cycles $c_1, \ldots, c_\mu$ intersect $\pi$. In general, however, some of the $c_1, \ldots, c_\mu$ might not intersect $\pi$. For example, for $\mu = 2$ it could be that $c_2$ intersects $c_1$ but not $\pi$, thus translating to the constraint that $n_2 > 0$ only if $n_1 > 0$. In order to avoid such constraints, we instead extend $\pi$ to the walk $\pi^{ext}$ by ``inserting'' the cycle $c_1$ once into $\pi$. Then, the above argument applies with $\pi^{ext}$ replacing $\pi$, and we let both $\pi$ and $\pi^{ext}$ be elements of $\mathcal{M}_{\pi}$.\footnote{Formally, $\mathcal{M}_{\pi}$ is not a set of walks. Thus, to be formally correct, we should rather say ``we let both $\pi$ and $\pi^{ext}$ correspond to an element of $\mathcal{M}_{\pi}$'', but we use the above informal formulation for simplicity.} We might need to add other such extensions of $\pi$ to $\mathcal{M}_{\pi}$ and might even need to build extensions of $\pi^{ext}$ (for example, if $\mu = 3$ and $c_3$ intersects $c_2$ but neither of $c_1$ and $\pi$, and $c_2$ intersects $c_1$ but not $\pi$), which too are added to $\mathcal{M}_{\pi}$. However, since there are only finitely many irreducible cycles, this process terminates and the set $\mathcal{M}_{\pi}$ remains finite.
\end{remark}

To illustrate Lemma~\ref{thm.main1}, the following examples continue Example~\ref{example:running-example-1}.

\begin{example}\label{example:running-example-2}
For the multi-weighted summary graph in part (b) of Figure~\ref{fig:multi_weighted_summary_graph}, the set $\cW_0(k = X, j = Z)$ contains the element $\tilde{\pi}^\prime = (X \tailhead Y \tailhead Z)$. This walk intersects the irreducible cycles $c_1 = (X \tailhead X)$ and $c_2 = (X \tailhead Y \tailhead X)$. We can ``insert'' these cycles an arbitrary number of times into $\tilde{\pi}^\prime$, say $n^\prime_1$ and $n^\prime_2$ times, to obtain a directed walk $\tilde{\pi}_{n^\prime_1, n^\prime_2}^\prime \in \cW(k = X, j = Z)$ from $X$ to $Z$ with $\w(\tilde{\pi}^\prime_{n^\prime_1, n^\prime_2}) = \w(\tilde{\pi}^\prime) + n^\prime_1 \cdot \w(c_1) + n^\prime_2 \cdot \w(c_2)$. We thus see that $\w(\tilde{\pi}_{n^\prime_1, n^\prime_2}^\prime) \in \mathbf{W}_{\tau = 0}(k = X, j = Z)$ for all $(n^\prime_1, n^\prime_2) \in \mathbb{N}_0 \times \mathbb{N}_0$ and, hence, find the inclusion $\mathrm{con}(a_0; a_1, a_2) \subseteq \mathbf{W}_{\tau = 0}(k = X, j = Z)$ for all $(a_0; a_1, a_2) \in (\{\tau = 0\} + \w(\tilde{\pi}^\prime)) \times \w(c_1) \times \w(c_2) = \{1, 6\} \times \{2\} \times \{3\} =  \{(1; 2, 3), (6; 2, 3)\}$. Also note that $\tilde{\pi}^\prime_{0, 1} = \pi^\prime$ and $\weight{\rho^\prime} = 4 \in \w(\pi^\prime) = \w(\tilde{\pi}^\prime_{0, 1})$ with $\rho^\prime = (X_{t-4} \tailhead Y_{t-3} \tailhead X_{t-1} \tailhead Y_t \tailhead Z_t)$ and $\pi^\prime = (X \tailhead Y \tailhead X \tailhead Y \tailhead Z)$ as in Example~\ref{example:running-example-1}. Conversely, $\pi^\prime$ maps to $\tilde{\pi}^\prime$ by ``collapsing'' the irreducible cycle $c_2 = (X \tailhead Y \tailhead X)$ on $\pi^\prime$ to the single vertex $(X)$.

The set $\cW_0(k = X, i = X)$ contains the trivial walk $\tilde{\pi} = (X)$. Using arguments along the same lines as before, we see that $\mathrm{con}(a_0; a_1, a_2) \subseteq \mathbf{W}_{\tau = 0}(k = X, i = X)$ for all $(a_0; a_1, a_2) \in (\{\tau = 0\} + \w(\tilde{\pi})) \times \w(c_1) \times \w(c_2) = \{0\} \times \{2\} \times \{3\} = \{(0; 2, 3)\}$.
\end{example}

The importance of Lemma~\ref{thm.main1} for our purpose lies in two facts: First, the union on the right-hand-side of eq.~\eqref{eq:main-decomposition} is over a \emph{finite} number of affine cones $\mathrm{con}(a_0;a_1,\dots,a_{\mu})$. Second, while the cones $\mathrm{con}(a_0;a_1,\dots,a_{\mu})$ themselves can be infinite (the cone is infinite if and only if $\mu > 0$), they take the particular form as specified by the right-hand-side of eq.~\eqref{eq:cones}. Intuitively speaking, we have thus found regularity in the potentially infinite set $\mathbf{W}_{\tau}(k,i)$. Utilizing this regularity, we can solve Problem~\ref{prob1} as follows.

\begin{theorem}[Solution of Problem \ref{prob1}] \label{cor.prob1}
Let $(\cG,\w)$ be a weakly acylic MWDG, let $\tau \in \mathbb{N}_0$ be a non-negative integer, and let $i, j, k$ be (not necessarily distinct) vertices in $\cG$. Then, 
\begin{align*}
\mathbf{W}_{\tau}(k,i) \cap \mathbf{W}_{0}(k,j) \neq \emptyset
\end{align*}
if and only if there exist
\begin{itemize}
\item $ (a_0,\dots,a_{\mu}) \in D_{\tau}(\pi,\mathcal S)$ where $\mathcal{S} \in \mathcal{M}_{\pi}$ and $\pi \in  \cW_0(k,i)$, as well as
\item $ (a_0',\dots,a_{\nu}') \in D_{0}(\pi',\mathcal S')$ where $\mathcal{S}' \in \mathcal{M}_{\pi'}$ and $ \pi' \in  \cW_0(k,j)$ 
\end{itemize}
such that the linear Diophantine equation
\begin{align} \label{eq.new_Diophantine}
 a_0 + \sum_{\alpha=1}^{\mu} n_{\alpha} \cdot a_{\alpha} =  a'_0 + \sum_{\beta=1}^{\nu} n'_{\beta} \cdot a'_{\beta},
\end{align}
has a non-negative integer solution $(n_1,\dots,n_{\mu};n'_1,\dots,n'_{\nu})\in \mathbb{N}_0^{\mu}\times \mathbb{N}_0^{\nu}$. Moreover, we can decide about the existence versus non-existence of such a solution in finite time.
\end{theorem}

Theorem~\ref{cor.prob1} reduces Problem~\ref{prob1}, and by extension Problem~\ref{problem:common-ancestor} (that is, the common-ancestor search in infinite ts-DAGs) and Problem~\ref{problem:eventual-goal} (that is, the construction of finite marginal ts-ADMGs of infinite ts-ADMGs), to a number theoretic problem: deciding whether one of the finitely many linear Diophantine equations in Theorem~\ref{cor.prob1} admits a solution that consists entirely of \emph{non-negative} integers. Theorem~\ref{thm:our-linear-diophantine}, in Section~\ref{sec:solvability-check} below, shows how to make these decisions in finite time. Combining Theorem~\ref{cor.prob1} and Theorem~\ref{thm:our-linear-diophantine} thus solves Problem~\ref{prob1} and, by extension, also Problem~\ref{problem:common-ancestor} and Problem~\ref{problem:eventual-goal}. In Section~\ref{sec:pseudocode}, we also provide pseudocode for this solution.

\begin{proof}[Proof of Theorem~\ref{cor.prob1}]
Using the decomposition of $\mathbf{W}_{\tau}(k,i)$ and $\mathbf{W}_{0}(k,j)$ as stated by Lemma~\ref{thm.main1}, we see that 
\begin{align*}
\mathbf{W}_{\tau}(k,i) \cap \mathbf{W}_{0}(k,j) \neq \emptyset
\end{align*}
if and only if there are $(a_0,\dots,a_{\mu}) \in D_{\tau}(\pi,\mathcal S)$ where $\mathcal{S} \in \mathcal{M}_{\pi}$ and $ \pi \in  \cW_0(k,i)$ as well as $ (a_0',\dots,a_{\nu}') \in D_{0}(\pi',\mathcal S')$ where $\mathcal{S}' \in \mathcal{M}_{\pi'}$ and $ \pi^\prime \in  \cW_0(k,j)$ such that
\begin{align*}
\mathrm{con}(a_0;a_1,\dots,a_{\mu}) \cap \mathrm{con}(a'_0;a'_1,\dots,a'_{\nu}) \neq \emptyset \, .
\end{align*} 
By definition of the affine cones $\mathrm{con}(a_0;a_1,\dots,a_{\mu})$ and $\mathrm{con}(a'_0;a'_1,\dots,a'_{\nu})$, their intersection is non-empty if and only if eq.~\eqref{eq.new_Diophantine} has a solution in $\mathbb{N}_0^{\mu}\times \mathbb{N}_0^{\nu}$.

Theorem~\ref{thm:our-linear-diophantine} in Section~\ref{sec:solvability-check} below shows the second part of Theorem~\ref{cor.prob1}, namely that we can answer these solvability queries in finite time.
\end{proof}

To illustrate Theorem~\ref{cor.prob1}, the following examples continues Examples~\ref{example:running-example-1} and \ref{example:running-example-2}.

\begin{example}\label{example:running-example-3}
In Example~\ref{example:running-example-2} above, we found that $\mathrm{con}(a_0; a_1, a_2) \subseteq \mathbf{W}_{\tau = 0}(k = X, i = X)$ with $(a_0; a_1, a_2) = (0; 2, 3)$ and $\mathrm{con}(a^\prime_0; a^\prime_1, a^\prime_2) \subseteq \mathbf{W}_{\tau = 0}(k = X, j = Z)$ with $(a^\prime_0; a^\prime_1, a^\prime_2) = (1; 2, 3)$. This combination of tuples gives rise to the linear Diophantine equation
\begin{equation*}
0 + n_1 \cdot 2 + n_2 \cdot 3 = 1 + n_1^\prime \cdot 2 + n_2^\prime \cdot 3 \, ,
\end{equation*}
which has the non-negative integer solution $(n_1, n_2; n_1^\prime, n_2^\prime) = (2, 0; 0, 1)$ that corresponds to the common ancestor $X_{t-4}$ of $X_t$ and $Z_t$ in the ts-DAG in part (a) of Figure~\ref{fig:multi_weighted_summary_graph}, cf.~Example~\ref{example:running-example-1}. Further solutions are $(n_1, n_2; n_1^\prime, n_2^\prime) = (2 + n, 0; n, 1)$ and $(n_1, n_2; n_1^\prime, n_2^\prime) = (n, 1; 1+n, 0)$ for all $n \in \mathbb{N}_0$, from which we conclude that $X_{t-\tau^\prime}$ is a common ancestor of $X_t$ and $Z_t$ for all $\tau^\prime \in \{3, 4, 5, \ldots, \} \subseteq \mathbf{W}_{\tau=0}(k=X,i=X) \cap \mathbf{W}_{0}(k=X,j=Z)$.

We remark that we discuss these explicit solutions solely for illustration, whereas for Theorem~\ref{cor.prob1} we only need to know whether at least one non-negative integer solution exists.
\end{example}

\subsubsection{Existence of non-negative integers solutions of linear Diophantine equations}\label{sec:solvability-check}
Our solution of Problem~\ref{prob1} through Theorem~\ref{cor.prob1} depends on a method for deciding whether a linear Diophantine equation admits a solution that consists of non-negative integers only. Importantly, we only need to decide whether a solution \emph{exists or not}. If a solution exists, then we \emph{do not necessarily need to explicitly find a solution}, although in some subcases we will do so. Not having to find explicit solutions reduces the required computational effort  significantly.

To derive conditions for deciding whether or not eq.~\eqref{eq.new_Diophantine} admits a non-negative integer solution, we make use of the following two standard results from the number theory literature.

\begin{lemma}[See for example \citet{andreescu_introduction_2010}]\label{lemma:linear-diophantine-general}
Let $a_0$ be an integer and let $a_1, \, \ldots, \, a_{\rho}$ with $\rho \geq 1$ be non-zero integers. Then, the linear Diophantine equation $a_0 = \sum_{\alpha \, = \, 1}^\rho n_{\alpha} \cdot a_\alpha$ has an integer solution $(n_1, \, \ldots, \, n_\rho) \in \mathbb Z^{\rho}$ if and only if $a_0 \!\!\mod \!\gcd(a_1, \, \ldots, \, a_\rho) = 0$ where $\gcd(\cdot)$ denotes the greatest common divisor. \hfill $\qedsymbol$
\end{lemma}

\begin{lemma}[See for example \citet{ramirez_alfonsin_diophantine_2005}]\label{lemma:frobenius-number}
Let $\rho \geq 1$ and let $a_1, \, \ldots, \, a_{\rho}$ be positive integers with $\gcd(a_1, \, \ldots, \, a_\rho) = 1$. Then, there is a unique largest integer $f(a_1, \, \ldots, \, a_{\rho})$, known as the \emph{Frobenius number} of $a_1, \, \ldots, \, a_{\rho}$, such that there are no non-negative integers $n_1, \, \ldots, \, n_\rho$ with $\sum_{\alpha \, = \, 1}^\rho n_{\alpha} \cdot a_\alpha= f(a_1, \, \ldots, \, a_{\rho})$. In particular, for all integers $a_0 > f(a_1, \, \ldots, \, a_{\rho})$ \emph{there are} non-negative integers $n_1, \, \ldots, \, n_\rho$ with $a_0 = \sum_{\alpha \, = \, 1}^\rho n_{\alpha} \cdot a_\alpha$. \hfill $\qedsymbol$
\end{lemma}

By itself, Lemma~\ref{lemma:linear-diophantine-general} is not sufficient to decide whether eq.~\eqref{eq.new_Diophantine} has a non-negative integer solution since Lemma~\ref{lemma:linear-diophantine-general} deals with all integer solutions rather than only the non-negative ones. However, in combination with Lemma~\ref{lemma:frobenius-number} we arrive at the following result.

\begin{theorem}
\label{thm:our-linear-diophantine}
Consider the linear Diophantine equation~\eqref{eq.new_Diophantine} and write $c = a_0-a_0^\prime$, $g_a = \gcd(a_1, \, \ldots, \, a_\mu)$, $g_{a^\prime} = \gcd(a^\prime_1, \, \ldots, \, a^\prime_\nu)$ and $g_{aa^\prime} = \gcd(g_a,g_{a^\prime})$. Then, the following mutually exclusive and collectively exhaustive cases answer the question whether this equation has at least one non-negative integer solution for the unknowns $n_1, \, \ldots, \, n_\mu, \, n^\prime_1, \, \ldots, \, n^\prime_\nu$:
\begin{enumerate}
\item \underline{$\mu = 0$ and $\nu = 0$.}
There is a non-negative integer solution if and only if $c = 0$.
\item \underline{$\mu = 0$ and $\nu \neq 0$ and $c > 0$.}
If $c \!\! \mod \! g_{a^\prime} \neq 0$, then there is no non-negative integer solution. If $c \!\! \mod \! g_{a^\prime} = 0$, then there is a non-negative integer solution if and only if there is a solution within the finite search space $\{0, \, \ldots, \lfloor{\tfrac{c}{a^\prime_1}}\rfloor\} \times \dots \times \{0, \, \ldots, \lfloor{\tfrac{c}{a^\prime_\nu}}\rfloor\}$.
\item \underline{($\mu = 0$ and $\nu \neq 0$ and $c \leq 0$) or ($\mu \neq 0$ and $\nu = 0$ and $c \geq 0$).} There is a non-negative integers solution if and only if $c = 0$.
\item \underline{$\mu \neq 0$ and $\nu = 0$ and $c < 0$.}
If  $-c \!\! \mod \! g_a \neq 0$, then there is no non-negative integer solution. If  $-c \!\! \mod \! g_a = 0$, then there is a non-negative integer solution if and only if there is a solution within the finite search space $\{0, \, \ldots, \lfloor{-\tfrac{c}{a_1}}\rfloor\} \times \dots \times \{0, \, \ldots, \lfloor{-\tfrac{c}{a_\mu}}\rfloor\}$.
\item \underline{$\mu \neq 0$ and $\nu \neq 0$.}
There is a non-negative integer solution if and only $c \! \! \mod g_{aa^\prime} = 0$.
\end{enumerate}
\end{theorem}

\begin{remark}\label{rem.search-avoiding-condition}
In the if-and-only-if subcases of cases 2 and 4 of Theorem~\ref{thm:our-linear-diophantine}, it is not always necessary to run an explicit search for a solution within the described finite search spaces. Focusing on case 2  because case 4 is similar, if $c \!\! \mod \! g_{a^\prime} = 0$ and moreover $c \geq g_{a^\prime} \cdot f^\prime(\tfrac{a^\prime_1}{g_{a^\prime}} , \, \ldots, \, \tfrac{a^\prime_\nu}{g_{a^\prime}} )$ where $f^\prime(d_1, \, \ldots, \, d_\rho) = \left( \min_\alpha d_{\alpha} -1 \right) \cdot \left( \max_\alpha d_{\alpha} -1 \right)$, then there always exists a non-negative integer solution and an explicit search can thus be avoided. This claim holds because of the upper bound $f^\prime(d_1, \, \ldots, \, d_\rho) \geq f(d_1, \, \ldots, \, d_\rho) + 1$ on the Frobenius number \citep{brauer_problem_1942} that is attributed to Schur and Brauer.
\end{remark}

Theorem~\ref{thm:our-linear-diophantine} immediately translates into a finite-time algorithm for deciding whether or not the linear Diophantine equation~\eqref{eq.new_Diophantine} has a non-negative integer solution. Of the five cases in Theorem~\ref{thm:our-linear-diophantine}, cases 2 and 4 are the computationally most expensive ones as they potentially require to run an explicit solution search within the specified finite search spaces; for example, in case 2, whether the reduced equation $c = \sum_{\beta \,=\, 1}^{\nu} n^\prime_\beta \cdot a^\prime_\beta$ has a solution within the finite search space $\{0, \, \ldots, \lfloor{\tfrac{c}{a^\prime_1}}\rfloor\} \times \dots \times \{0, \, \ldots, \lfloor{\tfrac{c}{a^\prime_\nu}}\rfloor\}$. One can, in principle, perform this explicit search in a brute-force way. However, using that the search for a non-integer solution to the reduced equation is a special case of the \emph{subset sum problem} \citep{bringmann_near-linear_2017}, there are also more refined search algorithms. The subset sum problem is a well-studied NP-complete combinatorial problem that dynamic programming algorithms can solve in pseudo-polynomial time \citep{pisinger_linear_1999}. For example, the \textsf{R}-package \textsf{nilde} \citep{arnqvist2019nilde} provides an implementation of such a more refined algorithm. Moreover, in order to altogether avoid the explicit searches if possible, one can use Remark~\ref{rem.search-avoiding-condition} as follows: Focusing on case 2, if the necessary condition $c \!\! \mod \! g_{a^\prime} = 0$ is met, then one can subsequently check the sufficient condition $c \geq g_{a^\prime} \cdot f^\prime(\tfrac{a^\prime_1}{g_a^\prime}, \, \ldots, \, \tfrac{a^\prime_\nu}{g_{a^\prime}})$ before moving to the explicit search. If this sufficient condition is met, then there is a non-negative integer solution and one does not need to run the explicit search at all.\footnote{To further reduce the number of explicit searches, one can use the following approach: Let $\mathcal{L}_1, \ldots, \mathcal{L}_{A}$ be the finite collection of linear Diophantine equations specified by Theorem~\ref{cor.prob1}. The most straightforward approach is to consider these equations in the sequential order in which they are given. However, if case 2 or 4 of Theorem~\ref{cor.prob1} applies to $\mathcal{L}_1$, then this sequential approach entails to potenially run the respective explicit searches before even looking at $\mathcal{L}_2$; and similar for the other equations. Instead, we might in a first phase only consider those equations to which case 1, 3 or 5 of Theorem~\ref{cor.prob1} applies and then, only if we do not find a solution in the first phase, move to considering the other equations in a second phase. Various further computational improvements seem possible, but here we content ourselves with the explained conceptual solution.}

To illustrate Theorem~\ref{thm:our-linear-diophantine} and continue with our running example, the following example continues Examples~\ref{example:running-example-1}, \ref{example:running-example-2} and \ref{example:running-example-3}.

\begin{example}\label{example:running-example-4}
In Example~\ref{example:running-example-3}, we considered the linear Diophantine equation
\begin{equation*}
0 + n_1 \cdot 2 + n_2 \cdot 3 = 1 + n_1^\prime \cdot 2 + n_2^\prime \cdot 3 \, .
\end{equation*}
Comparing with the general form of eq.~\eqref{eq.new_Diophantine}, we see that $\mu = 2$ and $\nu = 2$, so case 5 of Theorem~\ref{thm:our-linear-diophantine} applies. Since $c = a_0 - a_0^\prime = 0 - 1 = -1$ and $g_{aa^\prime} = \gcd(\gcd(a_1, a_2), \gcd(a_1^\prime, a_2^\prime)) = \gcd(\gcd(2, 3), \gcd(2, 3)) = 1$, the condition $c \! \! \mod g_{aa^\prime} = 0$ is fulfilled, from which we re-discover that the equation has a non-negative integer solution. Note that, to draw this conclusion, there is no need to explicitly find a solution. 
\end{example}

\subsubsection{A formula for the cutoff point}\label{sec:formula-cutoff}
As a corollary to Proposition~\ref{prop:common-ancestorship-reexpression}, Theorem~\ref{cor.prob1}, Theorem~\ref{thm:our-linear-diophantine} and Remark~\ref{rem.search-avoiding-condition}, we can now derive a formula for a finite cutoff point $\pcutoff(\DAG) < \infty$ with the property that one can restrict the common-ancestor search in ts-DAGs $\DAG$ to the finite interval $[t-\pcutoff(\DAG),t]$ (cf.~the discussion in Section~\ref{subsec:summary-graph-insufficient}).

\begin{theorem} \label{thm.upper-bound-main-theorem}
Let $\DAG = (\VarIndices \times \mathbb{Z}, \DirectedEdges)$ be a ts-DAG and let $\mathcal{S}_{\w}(\DAG) = (\mathcal{S}(\DAG), \w)$ be its multi-weighted summary graph. Denote the (finite) set of equivalence classes of irreducible cycles in $\mathcal{S}(\DAG)$ as $\mathcal{C}$ and define the quantities
\begin{align*}
&K = \max_{\mathbf{c} \in \mathcal{C}} \max \w(\mathbf{c}) \,\, &&\text{(maximal weight of any irreducible cycle in $\SummaryGraph{\DAG}$)\, ,} \\
&L = \max_{k,i \in \VarIndices} \max_{\pi\in \cW_0(k,i)} \max \w(\pi) \,\,  &&\text{(maximal weight of any directed or trivial path in $\SummaryGraph{\DAG}$)\, ,} \\
&M = \sum_{\mathbf{c} \in \mathcal{C}} \max \w(\mathbf{c}) \,\,  &&\text{(sum over the maximal weights of all irreducible cycles)\, .} 
\end{align*}
Note that the multi-weight of an equivalence class of cycles is well-defined since any representative of the class has the same multi-weight.

Let $(i,t-\tau)$ and $(j,t)$ with $0 \leq \tau \leq \ptimewindow$ be two vertices in $\DAG$. Then, $(i,t-\tau)$ and $(j,t)$ have a common ancestor in the infinite ts-DAG $\DAG$ if and only if $(i,t-\tau)$ and $(j,t)$ have a common ancestor in the finite segment of $\DAG$ on the time window $[t-\pcutoff(\DAG), t]$ where
\begin{align}\label{eq:cutoff}
\pcutoff(\DAG) = \left(K^2+1\right) \cdot \left(\ptimewindow + L+M \right) + K\cdot\left[\left(K-1\right)^2 +1\right] \, .
\end{align}
\end{theorem}

Theorem~\ref{thm.upper-bound-main-theorem} provides a direct solution to Problem~\ref{problem:common-ancestor} and, by extension, to Problem~\ref{problem:eventual-goal}. Specifically, given a ts-ADMG $\ADMG$, we can determine its marginal ts-ADMG $\tsADMG{\ObservedVertices}{\ADMG}$ by the equality $\tsADMG{\ObservedVertices}{\ADMG} = \ADMGprojection{\ObservedVertices}{\ADMG_{[t-\pcutoff(\DAG)-\ptimewindow, t]}}$ where $\ADMG_{[t-\pcutoff(\DAG)-\ptimewindow, t]}$ is the finite segment of $\ADMG$ on the finite time window $[t-\pcutoff(\DAG)-\ptimewindow, t]$ with $\pcutoff(\DAG)$ as in eq.~\eqref{eq:cutoff}. Since $\ADMG_{[t-\pcutoff(\DAG), t]}$ is a finite graph, its projection to $\ADMGprojection{\ObservedVertices}{\ADMG_{[t-\pcutoff(\DAG), t]}}$ is a solved problem. While conceptually simple, we expect this approach of using Theorem~\ref{thm.upper-bound-main-theorem} to be computationally more expensive than the approach of using Theorem~\ref{cor.prob1} and Theorem~\ref{thm:our-linear-diophantine}. The rational behind this expectation is that the bound $\pcutoff(\DAG)$ is rather rough, such that $\ADMG_{[t-\pcutoff(\DAG)-\ptimewindow, t]}$ can potentially be large.

\begin{example}\label{example:running-example-5}
Consider once more the ts-DAG $\DAG$ and its multi-weighted summary graph in parts (a) and (b) of Figure~\ref{fig:multi_weighted_summary_graph}. There are exactly two equivalence classes of irreducible cycles, namely $\mathbf{c}_1 = [X \tailhead X]$ and $\mathbf{c}_2 = [X \tailhead Y \tailhead X]$. The multi-weights of these equivalence classes are $\w(\mathbf{c}_1) = \{2\}$ and $\w(\mathbf{c}_2) = \{3\}$, such that $K = 3$ and $M = 5$. Moreover, $L = 6$ corresponding to the directed path $\pi = (X \tailhead Y \tailhead Z)$ with $6 \in \w(\pi) = \{1, 6\}$. Inserting these numbers into eq.~\eqref{eq:cutoff}, we get $\pcutoff(\DAG) = 10\cdot p + 125$.
\end{example}

\section{Conclusions}\label{sec:summary}
\paragraph*{Summary}
In this paper, we considered the projection of infinite time series graphs with latent confounders (infinite ts-ADMGs, see Definition~\ref{def:tsADMG-infinite}) to marginal graphs on finite time windows by means of the ADMG latent projection (finite marginal ts-ADMGs, see Definition~\ref{def:tsADMG-marginal}). While the projection procedure itself is not new, its practical execution on infinite graphs is non-trivial and had previously not been approached in generality. To close this conceptual gap, we first reduced the considered projection task (see Problem~\ref{problem:eventual-goal}) to the search for common ancestors in infinite ts-DAGs (see Problem~\ref{problem:common-ancestor}). We then further reduced this common-ancestor search to the task of deciding whether any of a finite number of linear Diophantine equations has a non-negative integer solution (see Problem~\ref{prob1} and Theorem~\ref{cor.prob1}). Thus, we established an intriguing connection between the theory of infinite graphs with repetitive edges and number theory. Building on standard results from number theory, we then derived criteria with which one can answer the corresponding solvability queries in finite time (see Theorem~\ref{thm:our-linear-diophantine}). Thus, by the combination of Theorems~\ref{cor.prob1} and \ref{thm:our-linear-diophantine}, we provided a solution to both the common-ancestor search in infinite ts-DAGs (Problem~\ref{problem:common-ancestor}) and the task of constructing finite marginal ts-ADMGs (Problem~\ref{problem:eventual-goal}). In Section~\ref{sec:pseudocode}, we also provide pseudocode that implements this solution. As a corollary to this solution, we derived a finite upper bound on a time window to which one can restrict the common ancestor searches relevant for the projection task (see Theorem~\ref{thm.upper-bound-main-theorem}). This result constitutes an alternative and conceptually simple solution to Problems~\ref{problem:eventual-goal} and \ref{problem:common-ancestor}, but we expect it to be computationally disadvantageous as compared to the number theoretic solution by means of Theorems~\ref{cor.prob1} and \ref{thm:our-linear-diophantine}. In Section~\ref{sec:generalize-to-DMAGs}, we further show how to execute the DMAG latent projection of infinite ts-ADMGs by utilizing the finite marginal ts-ADMGs.

\paragraph*{Significance}
The finite marginal graphs are useful tools for answering $m$-separation queries in infinite time series graphs as well as for causal effect identification and causal discovery in time series, see Section~\ref{sec:projection-motivation}. In particular, provided the causal Markov condition holds with respect to the infinite time series graph, the entirety of causal effect identification results for finite graphs directly applies to finite marginal graphs, whereas specific modifications might be necessary for applying these methods to infinite time series graphs. Therefore, we envision our results to be widely applicable in future research on causal inference in time series.

\paragraph*{Limitations}
As we stated in the previous paragraph, the applicability of causal effect estimation methods to the finite marginal graphs is contingent on the causal Markov condition holding with respect to the infinite time series graphs. Moreover, the projection to the finite marginal graphs inherently comes with the choice of an observed time window length, and any derived statement about (non-)identifiability with respect to the finite marginal graphs will be contingent on that choice. Lastly, our projection methods make the assumption of causal stationarity (note, however, that we really need this assumption only for the ``half-infinite'' graph that extends from the infinite past to some arbitrary finite time step).

\section*{Acknowledgments}
The authors thank Christoph Käding for helpful discussions at early stages of this project.

J.W., U.N., and J.R. received funding from the European Research Council (ERC) Starting Grant CausalEarth under the European Union’s Horizon 2020 research and innovation program (Grant Agreement No. 948112).

S.F. was enrolled at Technische Universität Berlin while working on this project.

\begin{appendix}

\section{Generalization to the DMAG latent projection}\label{sec:generalize-to-DMAGs}
In the main paper, we considered the projection of infinite ts-ADMGs (see Definition 2.2) to finite marginal ts-ADMGs (see Definition 3.1) by means of the ADMG latent projection (see Definition 2.1, due to \citet{pearl1995theory}, see also for example \citet{richardson2023nested}). Here, we extend our results to the DMAG latent projection \citep{richardson2002ancestral, zhang2008causal}, which is another widely-used projection procedure for representing causal knowledge in the presence of unobserved confounders.

\begin{definition}[DMAG latent projection \citep{richardson2002ancestral, zhang2008causal}]\label{def:DMAG-latent-projection}
Let $\ADMG$ be an ancestral ADMG with vertex set $\Vertices = \ObservedVertices \,\dot{\cup}\, \LatentVertices$. Then, its \emph{marginal DMAG $\DMAGprojection{\ObservedVertices}{\ADMG}$ on $\ObservedVertices$} is the bidirected graph with vertex set $\ObservedVertices$ such that
\begin{enumerate}
    \item there is an edge $i \astast j$ in $\DMAGprojection{\ObservedVertices}{\ADMG}$ if and only if $i \neq j$ and there is no set $\Zbold \subseteq \ObservedVertices \setminus \{i, j\}$ that $m$-separates $i$ and $j$ in $\ADMG$;
    \item an edge $i \astast j$ in $\DMAGprojection{\ObservedVertices}{\ADMG}$ is of the form $i \tailhead j$ if and only if $i \in \ancestors{j}{\ADMG}$ and, thus, $i \headhead j$ if and only if $i \notin \ancestors{j}{\ADMG}$ and $j \notin \ancestors{i}{\ADMG}$.
\end{enumerate}
\end{definition}

\noindent It follows that $i \in \ancestors{j}{\DMAGprojection{\ObservedVertices}{\ADMG}}$ if and only if $i, j \in \ObservedVertices$ and $i \in \ancestors{j}{\ADMG}$, so $\DMAGprojection{\ObservedVertices}{\ADMG}$ is an ADMG. The definition also readily implies that $\DMAGprojection{\ObservedVertices}{\ADMG}$ is ancestral, and \citet{richardson2002ancestral} shows that in $\DMAGprojection{\ObservedVertices}{\ADMG}$ there is no inducing path $\pi$ between non-adjacent vertices. Since a \emph{directed maximal ancestral graph (DMAG)} by definition is an ancestral ADMG that does not have inducing paths between non-adjacent vertices \citep{richardson2002ancestral, mooij2020constraint}, we thus see that $\DMAGprojection{\ObservedVertices}{\ADMG}$ is a DMAG indeed. Moreover, two observed vertices $i$ and $j$ are $m$-separated given $\Zbold$ in $\ADMG$ if and only if $i$ and $j$ are $m$-separated given $\Zbold$ in $\DMAGprojection{\ObservedVertices}{\ADMG}$, see Theorem 4.18 in \citet{richardson2002ancestral}. For an explanation of the difference between the ADMG and DMAG latent projections see, for example, Section 3.3 of \citet{triantafillou2015constraint}.

In complete analogy to Definition 3.1, we now define \emph{marginal time series DMAGs} \citep{gerhardus2021characterization} as DMAG latent projections of infinite ts-ADMGs.

\begin{definition}[Marginal time series DMAG, generalizing Definition 3.6 of \cite{gerhardus2021characterization}]\label{def:tsDMAG}
Let $\ADMG$ be a ts-ADMG with variable index set $\VarIndices$, let $\VarIndicesObserved \subseteq \VarIndices$ be non-empty, let $\TimeIndicesObserved$ be $\TimeIndicesObserved = \{t-\tau ~\vert~ 0 \leq \tau \leq \ptimewindow\}$ and let $\ObservedVertices = \VarIndicesObserved \times \TimeIndicesObserved$. Then, its \emph{marginal time series DMAG (marginal ts-DMAG) $\tsDMAG{\ObservedVertices}{\ADMG}$ on $\ObservedVertices$} is the DMAG $\DMAGprojection{\ObservedVertices}{\CanonicaltsDAG{\ADMG}}$, where $\CanonicaltsDAG{\ADMG}$ is the (infinite) canonical ts-DAG of $\ADMG$.
\end{definition}

\begin{remark}
\cite{gerhardus2021characterization} calls marginal ts-DMAGs simply ``ts-DMAGs'' (that is, does not use the attribute ``marginal''). However, since in the main paper we use the term ``marginal ts-ADMG'' to distinguish these finite graphs from the infinite ts-ADMGs, we here use the terminolgy ``marginal ts-DMAGs'' for consistency. Moreover, Definition 3.6 of \cite{gerhardus2021characterization} only applies to the special case of infinite ts-DAGs instead of infinite ts-ADMGs. Noting that $\CanonicaltsDAG{\ADMG} = \ADMG$ if $\ADMG$ is a ts-DAG, we see that Definition~\ref{def:tsDMAG} is indeed a proper generalization of Definition 3.6 of \cite{gerhardus2021characterization}. Lastly, we need to define $\tsDMAG{\ObservedVertices}{\ADMG}$ as $\DMAGprojection{\ObservedVertices}{\CanonicaltsDAG{\ADMG}}$ rather than as $\DMAGprojection{\ObservedVertices}{\ADMG}$ because Definition~\ref{def:DMAG-latent-projection} requires the input graph to be \emph{ancestral}, and hence $\DMAGprojection{\ObservedVertices}{\ADMG}$ is in general formally undefined.
\end{remark}

In analogy to marginal ts-ADMGs, the construction of marginal ts-DMAGs is non-trivial because there might be infinitely many paths in the infinite ts-ADMG $\ADMG$ that could potentially induce an edge $(i, t-\tau_i) \astast (j, t-\tau_j)$ in the finite marginal ts-DMAG $\tsDMAG{\ObservedVertices}{\ADMG}$. To the authors' knowledge, the construction of finite marginal ts-DMAGs has not yet been solved in the literature. Here, as a corollary to the results of the main paper, we solve this non-trival task by means of the following result.

\begin{proposition}\label{prop:get-ts-DMAG}
Let $\ADMG$ be an infinite ts-ADMG with variable index set $\VarIndices$, and let $\ObservedVertices = \VarIndicesObserved \times \TimeIndicesObserved$ with $\VarIndicesObserved \subseteq \VarIndices$ non-empty and $\TimeIndicesObserved = \{t-\tau ~\vert~ 0 \leq \tau \leq \ptimewindow\}$ where $\ptimewindow < \infty$. Then, $\tsDMAG{\ObservedVertices}{\ADMG} = \DMAGprojection{\ObservedVertices}{\CanonicalDAG{\tsADMG{\ObservedVertices}{\ADMG}}}$ where $\CanonicalDAG{\tsADMG{\ObservedVertices}{\ADMG}}$ is the (finite) canonical DAG of the finite marginal ts-ADMG $\tsADMG{\ObservedVertices}{\ADMG}$ of $\ADMG$.
\end{proposition}

\begin{remark}
Proposition~\ref{prop:get-ts-DMAG} is similar to a known relation between the ADMG and DMAG latent projections as specified, for example, by Theorem 13 and Algorithm 1 of \cite{triantafillou2015constraint}. Note that $\DMAGprojection{\ObservedVertices}{\tsADMG{\ObservedVertices}{\ADMG}}$ is in general ill-defined because $\tsADMG{\ObservedVertices}{\ADMG}$ need not be ancestral, hence we use $\DMAGprojection{\ObservedVertices}{\CanonicalDAG{\tsADMG{\ObservedVertices}{\ADMG}}}$ instead. 
\end{remark}

\noindent Importantly, the results of the main paper enable us to algorithmically construct the finite marginal ts-ADMG $\tsADMG{\ObservedVertices}{\ADMG}$, and the construction of its finite canonical DAG $\CanonicalDAG{\tsADMG{\ObservedVertices}{\ADMG}}$ as well as the projection of this finite graph to $\DMAGprojection{\ObservedVertices}{\CanonicalDAG{\tsADMG{\ObservedVertices}{\ADMG}}}$ by means of the DMAG latent projection are solved problems. Thus, we solved the task of algorithmically constructing finite marginal ts-DMAGs.

Lastly, similar to the discussion in Section 3.3.4, we can further project $\tsDMAG{\ObservedVertices}{\ADMG}$ to $\DMAGprojection{\ObservedVertices^{\prime\prime}}{\tsDMAG{\ObservedVertices}{\ADMG}} = \DMAGprojection{\ObservedVertices^{\prime\prime}}{\ADMG}$ with $\ObservedVertices^{\prime\prime} \subseteq \ObservedVertices$ arbitrary. Hence, we also solved the construction arbitrary finite DMAG latent projections $\DMAGprojection{\ObservedVertices^{\prime\prime}}{\ADMG}$ of infinite ts-ADMGs.

\section{Counterexamples}\label{sec:counterexamples}
Here, we provide the counterexamples to the simple heuristics considered in Section 4.2.2.

\begin{example}\label{ref:counterexample-1}
Consider the infinite ts-DAG $\DAG$ in part (a) of Figure~\ref{fig:heuristic_sums}, which for $\VarIndicesObserved = \VarIndices = \{X, Y\}$ and $\ptimewindow = 1$ projects to the finite marginal ts-ADMG $\tsADMG{\ObservedVertices}{\DAG}$ in part (b) of the same figure. We show that this case is a counterexample to the first of the two heuristics considered in Section 4.2.2. To this end, we first note that this heuristic prescribes restricting to the time window $[t-10, t]$, where $10$ is the observed time window length $\ptimewindow = 1$ plus the sum of lags $\pcutoff_{sum}(\DAG) = 9 = 5 + 3 + 1$ of all edges in $\DirectedEdges^t \cup \BidirectedEdges^t = \{X_{t-5} \tailhead X_t, \, Y_{t-3} \tailhead Y_t, \, Y_{t-1} \tailhead X_t\}$. Now consider the blue-colored vertices $X_{t-1}$ and $Y_t$. These vertices have the common ancestor $Y_{t-12}$ by means of the path $X_{t-1} \headtail X_{t-6} \headtail X_{t-11} \headtail Y_{t-12} \tailhead Y_{t-9}\tailhead Y_{t-6} \tailhead Y_{t-3} \tailhead Y_t$, but no common ancestor within the time window $[t-10, t]$. Thus, while there is the bidirected edge $X_{t-1} \headhead Y_t$ in the marginal ts-ADMG $\tsADMG{\ObservedVertices}{\DAG}$, this edge is not in the ADMG latent projection $\ADMGprojection{\ObservedVertices}{\DAG_{[t-10, t]}}$ of the segment $\DAG_{[t-10, t]}$ of $\DAG$ on $[t-10, t]$ as shown in part (c) of Figure~\ref{fig:heuristic_sums}. Also the edges $X_{t-1} \headhead X_t$ and $Y_{t-1} \headhead X_t$ are in $\tsADMG{\ObservedVertices}{\DAG}$ but not in $\ADMGprojection{\ObservedVertices}{\DAG_{[t-10, t]}}$.
\end{example}

\renewcommand{\thefigure}{A}
\begin{figure}[tbp]
    \centering
    \includegraphics[scale=0.25]{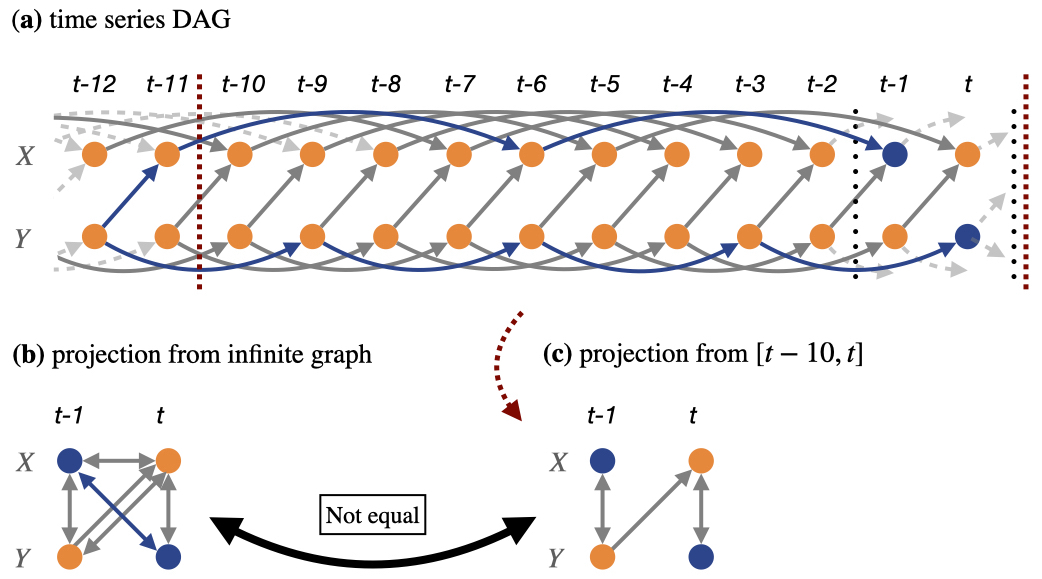}
    \caption{\textbf{(a)} A ts-DAG $\DAG$. \textbf{(b)} The marginal ts-ADMG $\tsADMG{\ObservedVertices}{\DAG}$ for $\VarIndicesObserved = \VarIndices$ and $\ptimewindow = 1$. \textbf{(c)} The ADMG latent projection $\ADMGprojection{\ObservedVertices}{\DAG_{[t-10, t]}}$ of the finite segment $\DAG_{[t-10, t]}$ of $\DAG$ on $[t-10, t]$. See also Example~\ref{ref:counterexample-1}.}
    \label{fig:heuristic_sums}
\end{figure}

\begin{example}\label{ref:counterexample-2}
Consider the infinite ts-DAG $\DAG$ in part (a) of Figure~\ref{fig:heuristic_product_lags}, which for \newline $\VarIndicesObserved = \VarIndices = \{X^1, X^2, X^3, X^4, X^5\}$ and $\ptimewindow = 1$ projects to the finite marginal ts-ADMG $\tsADMG{\ObservedVertices}{\DAG}$ in part (c) of the same figure. We show that this case is a counterexample to the second of the two heuristics considered in Section 4.2.2. To this end, we first note that this heuristic prescribes restricting to the time window $[t-2, t]$, where $2$ is the observed time window length $\ptimewindow = 1$ plus the product of all non-zero lags $\pcutoff_{prod}(\DAG) = 1 = 1^9$ of edges in $\DirectedEdges^t \cup \BidirectedEdges^t = \{X^1_{t-1} \tailhead X^1_t, \, \ldots, X^5_{t-1} \tailhead X^5_t, \, X^2_{t-1}\tailhead X^1_t, \, X^3_{t-1} \tailhead X^2_t, \, X^3_{t-1} \tailhead X^4_t, \, X^4_{t-1} \tailhead X^5_t\}$. Now consider the blue-colored vertices $X^1_{t-1}$ and $X^5_{t-1}$. These vertices have the common ancestor $X^3_{t-3}$ by means of the path $X^1_{t-1} \headtail X^2_{t-2} \headtail X^3_{t-3} \tailhead X^4_{t-2} \tailhead X^5_{t-1}$, but no common ancestor within the time window $[t-2, t]$. Consequently, while there is the bidirected edge $X^1_{t-1} \headhead X^5_{t-1}$ in the marginal ts-ADMG $\tsADMG{\ObservedVertices}{\DAG}$, this edge is not in the ADMG latent projection $\ADMGprojection{\ObservedVertices}{\DAG_{[t-2, t]}}$ of the segment $\DAG_{[t-2, t]}$ of $\DAG$ on $[t-2, t]$ as shown in part (b) of Figure~\ref{fig:heuristic_product_lags}. Moreover, there are several other bidirected edges that are in $\tsADMG{\ObservedVertices}{\DAG}$ but not in $\ADMGprojection{\ObservedVertices}{\DAG_{[t-2, t]}}$.
\end{example}

\renewcommand{\thefigure}{B}
\begin{figure}[tbp]
    \centering
    \includegraphics[scale=0.24]{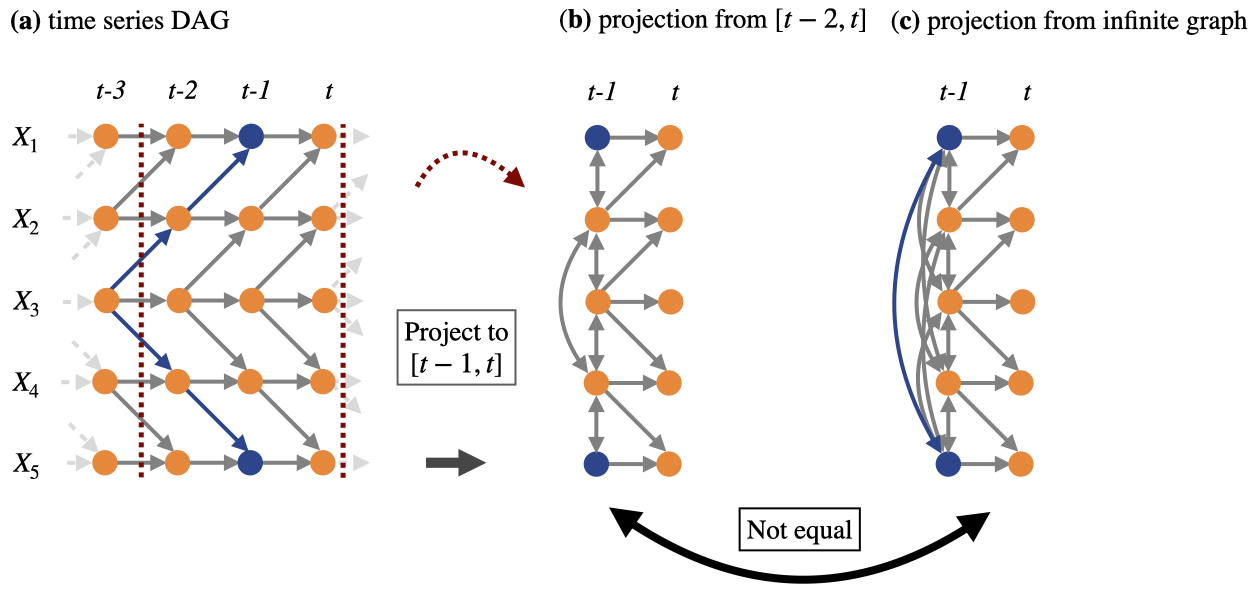}
    \caption{\textbf{(a)} A ts-DAG $\DAG$. \textbf{(b)} The ADMG latent projection $\ADMGprojection{\ObservedVertices}{\DAG_{[t-2, t]}}$ of the finite segment $\DAG_{[t-2, t]}$ of $\DAG$ on $[t-2, t]$ for $\VarIndicesObserved = \VarIndices$ and $\ptimewindow = 1$. \textbf{(c)} The marginal ts-ADMG $\tsADMG{\ObservedVertices}{\DAG}$. See also Example~\ref{ref:counterexample-2}.}
    \label{fig:heuristic_product_lags}
\end{figure}

\section{Proofs}\label{sec:proofs}
Here, we provide proofs for all theoretical claims of the main paper and of Section~\ref{sec:generalize-to-DMAGs}. We also state and prove several auxiliary results needed for this task. We split this section in three parts: First, the proofs of all propositions. Second, the proof of Lemma 4.8. Third, the proof of all theorems.

\subsection{Proofs of all propositions}

\begin{proof}[Proof of Proposition 3.5]
See the explanations in the paragraph between Definition 3.4 and Proposition 3.5.
\end{proof}

\begin{lemma}\label{lemma:AMDG-projection-partitioning-of-latents}[See also Remark~\ref{remark:not-novel}]
Let $\DAG$ be a DAG with vertex set $\Vertices = \ObservedVertices \,\dot{\cup}\, \LatentVertices_1 \,\dot{\cup}\, \LatentVertices_2$. Then, the ADMG latent projection commutes with partitioning the latent vertices, that is, the equality of graphs \newline $\ADMGprojection{\ObservedVertices}{\DAG} = \ADMGprojection{\ObservedVertices}{\ADMGprojection{\ObservedVertices\cup \LatentVertices_1}{\DAG}}$ holds.
\end{lemma}

\begin{remark}\label{remark:not-novel}
The statement in Lemma~\ref{lemma:AMDG-projection-partitioning-of-latents} seems to be well known in the field, so we do not claim novelty in this regard. However, since we did not find a formal proof in the literature, we here include such a formal proof. We stress that $\DAG$ is allowed to be an infinite graph and that $\LatentVertices_1$, $\LatentVertices_2$ are allowed to be infinite sets. 
\end{remark}

\begin{proof}[Proof of Lemma~\ref{lemma:AMDG-projection-partitioning-of-latents}]
Let $i \tailhead j$ be in $\ADMGprojection{\ObservedVertices}{\DAG}$. Then, $i, j \in \ObservedVertices$ and there is a directed path $\pi$ from $i$ to $j$ in $\DAG$ such that no middle vertex on $\pi$, if any, is in $\ObservedVertices$. Among all vertices on $\pi$, let $v_1, v_2, \ldots v_n$, where $n \geq 2$ and $v_1 = i$ and $v_n = j$, be the ordered sequence of vertices on $\pi$ that are in $\ObservedVertices \cup \LatentVertices_1$. Then, for all $1 \leq k \leq n-1$ the subpath $\pi(k, k+1)$ is a directed path from $v_k$ to $v_{k+1}$ in $\DAG$ such that all middle vertices on $\pi(k, k+1)$, if any, are in $\LatentVertices_2$. Consequently, for all $1 \leq k \leq n-1$ there is the edge $v_k \tailhead v_{k+1}$ in $\ADMGprojection{\ObservedVertices \cup \LatentVertices_1}{\DAG}$. By appending these edges, we obtain the concatenation $\pi^\prime = (v_1 \tailhead \ldots \tailhead v_n)$, which is a path in $\ADMGprojection{\ObservedVertices \cup \LatentVertices_1}{\DAG}$. That the concatenation $\pi^\prime$ is a path rather than a walk follows because all vertices on $\pi^\prime$ are also on $\pi$, which is a path. The path $\pi^\prime$ is a directed path from $v_1 = i$ to $v_n = j$ in $\ADMGprojection{\ObservedVertices \cup \LatentVertices_1}{\DAG}$ and no middle vertex on $\pi^\prime$ is in $\ObservedVertices$. Therefore, there is the edge $i \tailhead j$ in $\ADMGprojection{\ObservedVertices}{\ADMGprojection{\ObservedVertices \cup \LatentVertices_1}{\DAG}}$.

Let $i \tailhead j$ be in $\ADMGprojection{\ObservedVertices}{\ADMGprojection{\ObservedVertices \cup \LatentVertices_1}{\DAG}}$. Then, $i, j \in \ObservedVertices$ and there is a directed path $\pi$ from $i$ to $j$ in $\ADMGprojection{\ObservedVertices \cup \LatentVertices_1}{\DAG}$ such that all middle vertices on $\pi$, if any, are in $\LatentVertices_1$. Let $v_1, v_2, \ldots v_n$, where $n \geq 2$ and $v_1 = i$ and $v_n = j$, be the ordered sequence of all vertices on $\pi$. Then, for all $1 \leq k \leq n-1$ there is the edge $v_k \tailhead v_{k+1}$ in $\ADMGprojection{\ObservedVertices \cup \LatentVertices_1}{\DAG}$. Hence, for all $1 \leq k \leq n-1$ there is a directed path $\pi^\prime_k$ from $v_k$ to $v_{k+1}$ in $\DAG$ such that all middle vertices on $\pi^\prime_k$, if any, are in $\LatentVertices_2$. Let $\pi^\prime$ be the walk in $\DAG$ obtained by appending the sequence of paths $\pi^\prime_1, \ldots, \pi^\prime_{n-1}$ at the respective common vertices $v_2, \ldots, v_{n-1}$. This walk $\pi^\prime$ is directed from $v_1 = i$ to $v_j = j$. Thus, due to acyclicity of $\DAG$, the walk $\pi^\prime$ is a path. Moreover, all middle vertices on $\pi^\prime$, if any, are in $\LatentVertices_1 \,\dot{\cup}\, \LatentVertices_2$. Therefore, there is the edge $i \tailhead j$ in $\ADMGprojection{\ObservedVertices}{\DAG}$.

Let $i \headhead j$ be in $\ADMGprojection{\ObservedVertices}{\DAG}$. Then, $i, j \in \ObservedVertices$ and there is a confounding path $\pi$ between $i$ to $j$ in $\DAG$ such that no middle vertex on $\pi$, if any, is in $\ObservedVertices$. Among all vertices on $\pi$, let $v_1, v_2, \ldots v_n$, where $n \geq 2$ and $v_1 = i$ and $v_n = j$, be the ordered sequence of vertices on $\pi$ that are in $\ObservedVertices \cup \LatentVertices_1$. We now distinguish two mutually exclusive and collectively exhaustive cases:
\begin{itemize}
\item \textit{Case 1: The unique root vertex on $\pi$ is in $\ObservedVertices \cup \LatentVertices_1$.} Then, there is $l$ with $2 \leq l \leq n-1$ such that $v_l$ is the unique root vertex on $\pi$. Moreover, for all $1 \leq k \leq l-1$ the subpath $\pi(k, k+1)$ is a directed path from $v_{k+1}$ to $v_k$ such that all middle vertices on $\pi(k, k+1)$, if any, are in $\LatentVertices_2$, and for all $l \leq k \leq n-1$ the subpath $\pi(k, k+1)$ is a directed path from $v_{k}$ to $v_{k+1}$ such that all middle vertices on $\pi(k, k+1)$, if any, are in $\LatentVertices_2$. Consequently, for all $1 \leq k \leq l-1$ there is the edge $v_k \headtail v_{k+1}$ in $\ADMGprojection{\ObservedVertices \cup \LatentVertices_1}{\DAG}$ and for all $l \leq k \leq n-1$ there is the edge $v_k \tailhead v_{k+1}$ in $\ADMGprojection{\ObservedVertices \cup \LatentVertices_1}{\DAG}$.
\item \textit{Case 2: The unique root vertex on $\pi$ is not in $\ObservedVertices \cup \LatentVertices_1$}. Then, there is a $l$ with $1 \leq l \leq n-1$ such that the subpath $\pi(l, l+1)$ is a confounding path between $v_l$ and $v_{l+1}$ such that all middle vertices on $\pi(l, l+1)$, if any, are in $\LatentVertices_2$. Consequently, there is the edge $v_l \headhead v_{l+1}$ in $\ADMGprojection{\ObservedVertices \cup \LatentVertices_1}{\DAG}$. Moreover, for all $1 \leq k \leq l-1$ the subpath $\pi(k, k+1)$ is a directed path from $v_{k+1}$ to $v_k$ such that all middle vertices on $\pi(k, k+1)$, if any, are in $\LatentVertices_2$, and for all $l+1 \leq k \leq n-1$ the subpath $\pi(k, k+1)$ is a directed path from $v_{k}$ to $v_{k+1}$ such that all middle vertices on $\pi(k, k+1)$, if any, are in $\LatentVertices_2$. Consequently, for all $1 \leq k \leq l-1$ there is the edge $v_k \headtail v_{k+1}$ in $\ADMGprojection{\ObservedVertices \cup \LatentVertices_1}{\DAG}$ and for all $l+1 \leq k \leq n-1$ there is the edge $v_k \tailhead v_{k+1}$ in $\ADMGprojection{\ObservedVertices \cup \LatentVertices_1}{\DAG}$.
\end{itemize}
Let $\pi^\prime$ be the path in $\ADMGprojection{\ObservedVertices \cup \LatentVertices_1}{\DAG}$ obtained by appending these edges at the vertices $v_2, \ldots v_{n-1}$, which is of the form $\pi^\prime = (v_1 \headtail \ldots \headtail v_l \tailhead \ldots \tailhead v_n)$ or of the form $\pi^\prime = (v_1 \headtail \ldots \headtail v_l \headhead v_{l+1} \tailhead \ldots \tailhead v_n)$ and where $v_{l} = v_1$ and $v_{l+1} = v_n$ are allowed. That the concatenation $\pi^\prime$ is a path rather than a walk follows because all vertices on $\pi^\prime$ are also on $\pi$, which is a path. The path $\pi^\prime$ is a confounding path between $v_1 = i$ and $v_n = j$ in $\ADMGprojection{\ObservedVertices \cup \LatentVertices_1}{\DAG}$ such that no middle vertex on $\pi^\prime$, if any, is in $\ObservedVertices$. Therefore, there is the edge $i \headhead j$ in $\ADMGprojection{\ObservedVertices}{\ADMGprojection{\ObservedVertices \cup \LatentVertices_1}{\DAG}}$.

Let $i \headhead j$ be in $\ADMGprojection{\ObservedVertices}{\ADMGprojection{\ObservedVertices \cup \LatentVertices_1}{\DAG}}$. Then, $i, j \in \ObservedVertices$ and there is a confounding path $\pi$ between $i$ to $j$ in $\ADMGprojection{\ObservedVertices \cup \LatentVertices_1}{\DAG}$ such that all middle vertices on $\pi$, if any, are in $\LatentVertices_1$. Let $v_1, v_2, \ldots v_n$, where $n \geq 2$ and $v_1 = i$ and $v_n = j$, be the ordered sequence of all vertices on $\pi$. We now distinguish two mutually exclusive and collectively exhaustive cases:
\begin{itemize}
\item \textit{Case 1: There is no bidirected edge on $\pi$.} Then, there is $l$ with $2 \leq l \leq n-1$ such that for all $1 \leq k \leq l-1$ the subpath $\pi(l, l+1)$ is $v_k \headtail v_{k+1}$ and for all $l \leq k \leq n-1$ the subpath $\pi(l, l+1)$ is $v_k \tailhead v_{k+1}$. Hence, for all $1 \leq k \leq l-1$ there is a directed path $\pi^\prime_k$ from $v_{k+1}$ to $v_k$ in $\DAG$ such that all middle vertices on $\pi^\prime_k$, if any, are in $\LatentVertices_2$, and for all $l \leq k \leq n-1$ there is a directed path $\pi^\prime_k$ from $v_{k}$ to $v_{k+1}$ in $\DAG$ such that all middle vertices on $\pi^\prime_k$, if any, are in $\LatentVertices_2$.
\item \textit{Case 2: There is a bidirected edge on $\pi$.} Then, there is $l$ with $1 \leq l \leq n-1$ such that the subpath $\pi(l, l+1)$ is $v_l \headhead v_{l+1}$. Consequently, there is a confounding path $\pi^\prime_l$ between $v_l$ and $v_{l+1}$ in $\DAG$ such that all middle vertices on $\pi^\prime_l$, if any, are in $\LatentVertices_2$. Moreover, for all $1 \leq k \leq l-1$ the subpath $\pi(k, k+1)$ is $v_k \headtail v_{k+1}$ and for all $l+1 \leq k \leq n-1$ the subpath $\pi(k, k+1)$ is $v_k \tailhead v_{k+1}$. Hence, for all $1 \leq k \leq l-1$ there is a directed path $\pi^\prime_k$ from $v_{k+1}$ to $v_k$ in $\DAG$ such that all middle vertices on $\pi^\prime_k$, if any, are in $\LatentVertices_2$, and for all $l+1 \leq k \leq n-1$ there is a directed path $\pi^\prime_k$ from $v_{k}$ to $v_{k+1}$ in $\DAG$ such that all middle vertices on $\pi^\prime_k$, if any, are in $\LatentVertices_2$.
\end{itemize}
Let $\pi^\prime$ be the walk in $\DAG$ obtained by appending the sequence of paths $\pi^\prime_1, \ldots, \pi^\prime_{n-1}$ at the respective common vertices $v_2, \ldots, v_{n-1}$. Not all vertices on the concatenation $\pi^\prime$, which is a confounding walk, are necessarily also on $\pi$. Thus, there can be vertices that appear more than once on $\pi^\prime$. However, all vertices on $\pi^\prime$ that are not also on $\pi$ are necessarily in $\LatentVertices_2$. Therefore, neither $v_1 = i$ nor $v_n = j$ appears more than once on $\pi^\prime$. We now distinguish two mutually exclusive and collectively exhaustive cases:
\begin{itemize}
\item \textit{Case 1: No vertex appears more than once in $\pi^\prime$.} Then, $\pi^\prime$ is a path. Let $\tilde{\pi}$ be $\pi^\prime$.
\item \textit{Case 2: At least one vertex appears more than once in $\pi^\prime$.} Let $w_1, \ldots, w_m$, where $m \geq n$ and $w_1 = i$ and $w_m = j$, be the ordered sequence of all vertices on $\pi^\prime$. Let $a_1$ be the minimum over all $b$ with $1 \leq b \leq m$ such that $w_b$ appears more than once on $\pi^\prime$. Then, $a_1 \neq 1$ and $a_1 \neq m$ because neither $w_1 = i$ nor $w_m = j$ appears more than once on $\pi^\prime$, and $a_1 \neq m-1$ because else $w_a$ could, given the definition of $a_1$, only appear once on $\pi^\prime$. Thus, in summary, $2 \leq a_1 \leq m-2$. Moreover, $\pi^\prime(1, a_1)$ is either a directed path from $w_{a_1}$ to $w_1$ (namely if $w_{a_1}$ is on $\pi^\prime$ between $w_1$ and (including) the unique root vertex on $\pi$) or a confounding path between $w_{1}$ and $w_{a_1}$ (namely if $w_{a_1}$ is on $\pi^\prime$ between $w_m$ and (excluding) the unique root vertex on $\pi$). Let $a_2$ be the maximum over all $b$ with $a_1 < b \leq m-1$ such that $w_{b} = w_{a_1}$. Then $a_1 < a_2 \leq m-1$. Moreover, $\pi^\prime(a_2, m)$ is either a directed path from $w_{a_2}$ to $w_m$ (namely if $w_{a_2}$ is on $\pi^\prime$ between $w_m$ and (including) the unique root vertex on $\pi$) or a confounding path between $w_{m}$ and $w_{a_2}$ (namely if $w_{a_2}$ is on $\pi^\prime$ between $w_1$ and (excluding) the unique root vertex on $\pi$). Note that $\pi^\prime(1, a_1)$ and $\pi^\prime(a_2, m)$ cannot both be a confounding walk at the same time. Let $\tilde{\pi}$ be the walk obtained by appending $\pi^\prime(1, a_1)$ and $\pi^\prime(a_2, m)$ at their common vertex $w_{a_1} = w_{a_2}$. By definition of $a_1$ and $a_2$, this walk $\tilde{\pi}$ is a path. 
\end{itemize}
The path $\tilde{\pi}$ is a confounding path between $i$ and $j$ in $\DAG$ such that no middle vertex on $\tilde{\pi}$, if any, is in $\ObservedVertices$. Therefore, there is the edge $i \headhead j$ in $\ADMGprojection{\ObservedVertices}{\DAG}$.
\end{proof}

\begin{proof}[Proof of Proposition 3.6]
Use Lemma~\ref{lemma:AMDG-projection-partitioning-of-latents} with $\LatentVertices_1 = \LatentVerticesUnobservable$ and $\LatentVertices_2 = \LatentVerticesTemporally$.
\end{proof}

\begin{proof}[Proof of Proposition 3.7]
See the explanations in the paragraph directly above Proposition 3.7.
\end{proof}

\begin{proof}[Proof of Proposition 3.8]
See the explanations in the paragraph directly above Proposition 3.8.
\end{proof}

\begin{proof}[Proof of Proposition 4.2]
\textbf{Only if}. The premise is that the vertices $(i, t-\tau_i)$ and $(j,t)$ with $\tau_i \geq 0$ have a common ancestor in the ts-DAG $\DAG$. Thus, there is a path $\rho$ between in $(i,t-\tau_i)$ and $(j, t)$ in $\DAG$ that is i) a confounding walk or ii) directed from $(i,t-\tau_i)$ to $(j,t)$ or iii) directed from $(j, t)$ to $(i,t-\tau_i)$ or iv) trivial. Let $\pi$ be the projection of $\rho$ to the summary graph $\mathcal{S}(\DAG)$, which is obtained by ``removing'' the time indices of the vertices on $\rho$. Then, $\pi$ is a walk between $i$ and $j$ in $\mathcal{S}(\DAG)$ that is i) a confounding path or ii) directed from $i$ to $j$ or iii) directed from $j$ to $i$ or iv) trivial. If $\pi$ is a path, we have thus shown that $i$ and $j$ have a common ancestor in $\mathcal{S}(\DAG)$. If $\pi$ is not a path, then let $k$ be the vertex closest to $i$ on $\pi$ (including $i$ itself) that appears more than once on $\pi$. Let $\tilde{\pi}$ be the subwalk of $\pi$ obtained by collapsing the subwalk of $\pi$ between the first and last appearance of $k$ to the single vertex $k$. Then, $\tilde{\pi}$ is a path between $i$ and $j$ in $\mathcal{S}(\DAG)$ that is i) a confounding path or ii) directed from $i$ to $j$ or iii) directed from $j$ to $i$ or iv) trivial. Hence, $i$ and $j$ have a common ancestor in $\mathcal{S}(\DAG)$.

\textbf{If.} The premise is that the vertices $i$ and $j$ have a common ancestor in the summary graph $\mathcal{S}(\DAG)$. Thus, there is a path $\pi$ between in $i$ and $j$ in $\DAG$ that is i) a confounding path or ii) directed from $i$ to $j$ or iii) directed from $j$ to $i$  or iv) trivial. We now distinguish these four cases:
\begin{itemize}
\item \textit{Case i).} Then, the arguments in the paragraph directly below Proposition 4.2 show that $(i,t-\tau_i)$ and $(j, t)$ have a common ancestor in $\DAG$.
\item \textit{Case ii).} Then, in the ts-DAG $\DAG$ there is a directed path $\rho_{ij}$ from $(i, t-\tau_{ij})$ to $(j, t)$ for some $\tau_{ij} \geq 0$. Moreover, since $\DAG$ has all lag-$1$ autocorrelations, there is the (potentially trivial) path $\rho_{ii} = ((i, t-\max(\tau_i, \tau_{ij})) \tailhead (i, t-\max(\tau_i, \tau_{ij})+1) \tailhead \ldots \tailhead (i, t-\min(\tau_i, \tau_{ij})))$. By concatenating the paths $\rho_{ij}$ and $\rho_{ii}$ at their common endpoint vertex, which is $(i, t-\max(\tau_i, \tau_{ij}))$ if $\tau_{ij} > \tau_i$ and $(i, t-\min(\tau_i, \tau_{ij}))$ if $\tau_{ij} \leq \tau_i$, we obtain a walk $\rho$ between $(i, t-\tau_i)$ and $(j, t)$ that is a confounding walk (if $\tau_{ij} > \tau_i$) or a directed walk from $(i, t-\tau_i)$ to $(j, t)$ (if $\tau_{ij} \leq \tau_i$). If $\rho$ is a path, we have thus shown that $(i, t-\tau_i)$ and $(j, t)$ have a common ancestor in $\DAG$. If $\rho$ is not a path, then let $(i, t-\tilde{\tau}_i)$ be the vertex closest to $(i, t-\tau_i)$ on $\rho$ (including $(i, t-\tau_i)$ itself) that appears more than once on $\rho$. Let $\tilde{\rho}$ be the walk obtained by collapsing the subwalk of $\rho$ between the first and last appearance of $(i,t-\tilde{\tau}_i)$ to the single vertex $(i,t-\tilde{\tau}_i)$. Then, $\tilde{\rho}$ is a path between $(i,t-\tau_i)$ and $(j, t)$ in $\DAG$ that is a confounding path or directed from $(i,t-\tau_i)$ to $(j, t)$ or trivial. Hence, we have shown that $(i,t-\tau_i)$ and $(j,t)$ have a common ancestor in $\DAG$.
\item \textit{Case iii).} Then, in the ts-DAG $\DAG$ there is a directed path $\rho_{ji}$ from $(j, t-\tau_{ji})$ to $(i, t-\tau_i)$ for some $\tau_{ji} \geq \tau_i \geq 0$. Moreover, since $\DAG$ has all lag-$1$ autocorrelations, there is the path $\rho_{jj} = ((j, t-\tau_{ji}) \tailhead (j, t-\tau_{ji} +1) \tailhead \ldots \tailhead (j, t))$. By concatenating the paths $\rho_{ji}$ and $\rho_{jj}$ at their common endpoint vertex $(j, t-\tau_{ji})$, we obtain a walk $\rho$ between $(i, t-\tau_i)$ and $(j, t)$ that is a confounding walk (if $\tau_{ji} > 0$) or a directed walk from $(j, t)$ to $(i, t-\tau_i)$ (if $\tau_{ij} = 0$, which also implies $\tau_i = 0$). If $\rho$ is a path, we have thus shown that $(i, t-\tau_i)$ and $(j, t)$ have a common ancestor in $\DAG$. If $\rho$ is not a path, then let $(j, t-\tilde{\tau}_j)$ be the vertex closest to $(j, t)$ on $\rho$ (including $(j,t)$ itself) that appears more than once on $\rho$. Let $\tilde{\rho}$ be walk obtained by collapsing the subwalk of $\rho$ between the first and last appearance of $(j, t)$ to the single vertex $(j,t)$. Then, $\tilde{\rho}$ is a path between $(i,t-\tau_i)$ and $(j, t)$ in $\DAG$ that is a confounding path or directed from $(j,t)$ to $(i, t-\tau_i)$ or trivial. Hence, we have shown that $(i,t-\tau_i)$ and $(j,t)$ have a common ancestor in $\DAG$.
\item \textit{Case iv).} Then, $i = j$ and, using the fact that the ts-DAG $\DAG$ has all lag-$1$ autocorrelations, in $\DAG$ there is the directed path $\rho = ((i,t-\tau_i) \tailhead (i,t-\tau_i+1) \tailhead \ldots \tailhead (i, t))$ from $(i, t-\tau_i)$ to $(i, t) = (j, t)$ if $\tau_i > 0$. Hence, $(i,t-\tau_i)$ and $(j,t)$ have a common ancestor in $\DAG$.
\end{itemize}
We have thus proven the claim.
\end{proof}

\begin{proof}[Proof of Proposition 4.6]
\textbf{Only if}. The premise is that $(i,t-\tau)$ and $(j, t)$ with $\tau \geq 0$ have a common ancestor in $\DAG$. Thus, in $\DAG$ there is a path $\rho$ between $(i,t-\tau)$ and $(j, t)$ that is i) a confounding path or ii) directed from $(i, t-\tau)$ to $(j, t)$ or iii) directed from $(j, t)$ to $(i, t-\tau)$ or iv) trivial. By splitting $\rho$ at its (unique) root vertex $(k,t-\tau_k)$,\footnote{We define the root vertex of a trivial path to be the single vertex on that trivial path.} we obtain a path $\rho_{ki}$ from $(k,t-\tau_k)$ to $(i,t-\tau)$ and a path $\rho_{kj}$ from $(k, t-\tau_k)$ to $(j, t)$. Then, $\rho_{ki}$ is the trivial path consisting of $(k,t-\tau_k) = (i, t-\tau)$ or directed from $(k, t-\tau_k)$ to $(i, t-\tau)$. Similarly, $\rho_{kj}$ is the trivial path consisting of $(k,t-\tau_k) = (j, t)$ or directed from $(k, t-\tau_k)$ to $(j, t)$.

Let $\pi$ and $\pi^\prime$ respectively be the projections of $\rho_{ki}$ and $\rho_{kj}$ to the summary graph $\mathcal{S}(\DAG)$, which are obtained by ``removing'' the time indices of the vertices. Then, $\pi$ is the trivial path consisting of $k=i$ or a directed walk from $k$ to $i$, and $\pi^\prime$ is the trivial path consisting of $k=j$ or a directed walk from $k$ to $j$. Consequently, $\pi$ and $\pi^\prime$ satisfy the first two of the three conditions in Proposition 4.6.

To show that $\pi$ and $\pi^\prime$ also satisfy the third condition, let $e^{1}, \ldots ,e^{n_{ki}}$ be the (possibly empty) ordered sequence of edges on $\rho_{ki}$, and let $f^{1}, \ldots ,f^{n_{kj}}$ be the (possibly empty) ordered sequence of edges on $\rho_{kj}$. Since $\rho_{ki}$ is directed if it is non-trivial, we then get $w(\rho_{ki}) = \sum_{a=1}^{n_{ki}} w(e^{a})$ in terms of the lags $w(e^{a})$ of the edges $e^a$ if $n_{ki} \geq 1$ and $w(\rho_{ki}) = 0$ if $n_{ki} = 0$ (note that $\rho_{ki}$ is trivial if and only if $n_{ki} = 0$). Similarly, $w(\rho_{kj}) = \sum_{b=1}^{n_{kj}} w(f^{b})$ in terms of the lags $w(f^{b})$ of the edges $f^b$ if $n_{kj} \geq 1$ and $w(\rho_{kj}) = 0$ if $n_{kj} = 0$. Moreover, $w(\rho_{ki}) + \tau = w(\rho_{kj})$ because $\rho_{ki}$ and $\rho_{kj}$ have $(k,t-\tau_k)$ as their common root vertex.

Let $e_{\mathcal{S}}^{1}, \ldots ,e_{\mathcal{S}}^{n_{ki}}$ be the (possibly empty) ordered sequence of edges on $\pi$, and let $f_{\mathcal{S}}^{1}, \ldots ,f_{\mathcal{S}}^{n_{kj}}$ be the (possibly empty) ordered sequence of edges on $\pi^\prime$. Then, $e_{\mathcal{S}}^{a}$ is for all $1 \leq a \leq n_{ki}$ the projection of $e^a$ to $\mathcal{S}(\DAG)$ and $f_{\mathcal{S}}^{b}$ is for all $1 \leq b \leq n_{kj}$ the projection of $f^b$ to $\mathcal{S}(\DAG)$. Thus, $w(e^a) \in \w(e_{\mathcal{S}}^{a})$ for all $1 \leq a \leq n_{ki}$ and $w(f^b) \in \w(f_{\mathcal{S}}^{b})$ for all $1 \leq b \leq n_{kj}$, where $\w(\cdot)$ denotes a multi-weight. Then, by definition of the multi-weight of directed walks, $w(\pi) = \sum_{a=1}^{n_{ki}} w(e^{a}) = w(\rho_{ki}) \in \w(\pi)$ if $n_{ki} \geq 1$. If $n_{ki} = 0$, then $0 \in \w(\pi) = \{0\}$ by definition of the multi-weight of a trivial walk. Similarly, $w(\pi^\prime) = \sum_{b=1}^{n_{kj}} w(f^{b}) = w(\rho_{kj}) \in \w(\pi^\prime)$ if $n_{kj} \geq 1$, and $0 \in \w(\pi^\prime) = \{0\}$ if $n_{kj} = 0$. Hence, $\pi$ and $\pi^\prime$ also satisfy the third condition in Proposition 4.6.

\textbf{If.} The premise is that there are paths $\pi$ and $\pi^\prime$ in the summary graph $\mathcal{S}(\DAG)$ that satisfy all of the three conditions in Proposition 4.6. Let $w(\pi) \in \w(\pi)$ and $w(\pi^\prime) \in \w(\pi^\prime)$ with $w(\pi) + \tau = w(\pi^\prime)$, which exist due to the third condition in Proposition 4.6. Further, let $e_{\mathcal{S}}^{1}, \ldots ,e_{\mathcal{S}}^{n_{ki}}$ be the (possibly empty) ordered sequence of edges on $\pi$, and let $f_{\mathcal{S}}^{1}, \ldots ,f_{\mathcal{S}}^{n_{kj}}$ be the (possibly empty) ordered sequence of edges on $\pi^\prime$. Then, by definition of the multi-weight of directed walks, for all $1 \leq a \leq n_{ki}$ there is $w(e_{\mathcal{S}}^a) \in \w(e_{\mathcal{S}}^a)$ such that $w(\pi) = \sum_{a=1}^{n_{ki}} w(e_{\mathcal{S}}^a)$ if $n_{ki} \geq 1$. If $n_{ki} = 0$, then $w(\pi) = 0$ by definition of the multi-weight of a trivial walk. Similarly, for all $1 \leq b \leq n_{kj}$ there is $w(f_{\mathcal{S}}^b) \in \w(f_{\mathcal{S}}^b)$ such that $w(\pi^\prime) = \sum_{b=1}^{n_{kj}} w(f_{\mathcal{S}}^b)$ if $n_{kj} \geq 1$, and $w(\pi^\prime) = 0$ if $n_{kj} = 0$.

Suppose, for the moment, that $n_{ki} \geq 1$. Since $w(e_{\mathcal{S}}^a) \in \w(e_{\mathcal{S}}^a)$ for all $1 \leq a \leq n_{ki}$, according to the definition of the multi-weighted summary graph for all $1 \leq a \leq n_{ki}$ there is the edge $(k_a, t-w(e_{\mathcal{S}}^a)) \tailhead (k_{a+1}, t)$ in $\DAG$. Here, $k_1, \ldots, k_{n_{ki}}$ is the ordered sequence of vertices on $\pi$. We denote $k_1 =k$ and infer from the second condition in Proposition 4.6 that $k_{n_{ki}} = i$. Using the repeating edges property of ts-DAGs to appropriately shift these edges backwards in time and recalling that the weight of a directed walk is the sum the lags of its edges, we see that in $\DAG$ there is a directed path $\rho_{ki}$ from $(k_1, t-w(\pi)-\tau) = (k, t-w(\pi)-\tau)$ to $(k_{n_{ki}}, t-\tau) = (i,t-\tau)$. If $n_{ki} = 0$, then we let $\rho_{ki}$ be the trivial path consisting of $(i,t-\tau)$ only. Similarly, if $n_{kj} \geq 1$, then in $\DAG$ there is a directed path $\rho_{kj}$ from $(l_1, t-w(\pi^\prime))$ to $(j, t)$. Here, $l_1$ is the fist vertex on $\pi^\prime$, for which $l_1 = k$ according to first condition of Proposition 4.6. If $n_{kj} = 0$, then we let $\rho_{kj}$ be the trivial path consisting of $(j, t)$ only. Since $w(\pi) + \tau = w(\pi^\prime)$, the paths $\rho_{ki}$ and $\rho_{kj}$ have $(k,t-w(\pi)-\tau) = (k, t-w(\pi^\prime)) = (l_1, t-w(\pi^\prime))$ as a common endpoint vertex.

Let $\rho$ be the walk obtained by concatenating $\rho_{ki}$ and $\rho_{kj}$ at this common endpoint vertex. This walk $\rho$ is i) a confounding walk between $(i,t-\tau)$ and $(j, t)$ or ii) directed from $(i,t-\tau)$ to $(j, t)$ or iii) directed from $(j, t)$ to $(i,t-\tau)$ or iv) the trivial path consisting of $(i,t-\tau) = (j, t)$ only. If $\rho$ is a walk, we have thus shown that $(i,t-\tau)$ and $(j, t)$ have a common ancestor in $\DAG$. If $\rho$ is not a path, then we can remove a sufficiently large subwalk from $\rho$ to obtain a path $\tilde{\rho}$ which is i) a confounding path between $(i, t-\tau)$ and $(j, t)$ or ii) directed from $(i, t-\tau)$ to $(j, t)$ or iii) directed from $(j, t)$ to $(i, t-\tau_i)$ or iv) trivial. Hence, $(i,t-\tau)$ and $(j, t)$ have a common ancestor in $\DAG$.
\end{proof}

\begin{proof}[Proof of Proposition~\ref{prop:get-ts-DMAG}]
According to Proposition 1 in \citet{richardson2023nested}, the $m$-separations in the marginal ADMG $\ADMGprojection{\ObservedVertices}{\ADMG} = \tsADMG{\ObservedVertices}{\ADMG}$ are in one-to-one correspondence with the $d$-separations between vertices in $\ObservedVertices$ in the ADMG $\ADMG$.\footnote{Strictly speaking, Proposition 1 in \citet{richardson2023nested} only applies to ADMG latent projections of DAGs rather than of ADMGs without self edges. However, the proof of Proposition 1 in \citet{richardson2023nested} equally works for this more general case.} Moreover, as straightforwardly follows from the definition of the ADMG latent projection (see Definition 2.1, due to \citet{pearl1995theory}, see also for example \citet{richardson2023nested}), for all $i, j \in \ObservedVertices$ we have that $i$ is an ancestor of $j$ in $\ADMGprojection{\ObservedVertices}{\ADMG}$ if and only if $i$ is an ancestor of $j$ in $\ADMG$.

The definition of canonical DAGs (see the last paragraph of Section 2.2, cf.~also Section 6.1 of \citet{richardson2002ancestral}) straightforwardly implies that $m$-separations in the ADMG $\ADMGprojection{\ObservedVertices}{\ADMG}$ are in one-to-one correspondence with $d$-separations between vertices in $\ObservedVertices$ in its canonical DAG $\CanonicalDAG{\ADMGprojection{\ObservedVertices}{\ADMG}} = \CanonicalDAG{\tsADMG{\ObservedVertices}{\ADMG}}$. Moreover, as another implication of the definition of canonical DAGs, for all $i, j \in \ObservedVertices$ we have that $i$ is an ancestor of $j$ in $\CanonicalDAG{\ADMGprojection{\ObservedVertices}{\DAG}}$ if and only if $i$ is an ancestor of $j$ in $\ADMGprojection{\ObservedVertices}{\DAG}$.

Putting together the previous observations, we see that $m$-separations $\CanonicalDAG{\tsADMG{\ObservedVertices}{\ADMG}}$ between vertices in $\ObservedVertices$ are in one-to-one correspondence with $m$-separations in $\ADMG$ between vertices in $\ObservedVertices$. Moreover, for all $i, j \in \ObservedVertices$ we see that $i$ is an ancestor of $j$ in $\CanonicalDAG{\tsADMG{\ObservedVertices}{\ADMG}}$ if and only if $i$ is an ancestor of $j$ in $\ADMG$. The equality $\tsDMAG{\ObservedVertices}{\ADMG} = \DMAGprojection{\ObservedVertices}{\CanonicalDAG{\tsADMG{\ObservedVertices}{\ADMG}}}$ now follows because the DMAG latent projection (see Definition~\ref{def:DMAG-latent-projection}, due to \citet{richardson2002ancestral} and \citet{zhang2008causal}) only requires knowledge of $m$-separations and ancestral relationships between observed vertices.
\end{proof}

\subsection{Proof of Lemma 4.8} \label{subsec.proof-of-decomposition-theorem}
In order to prove Lemma 4.8, we first prove the following slightly adjusted version of it.

\begin{lemma} \label{thm.support-thm-1}
There exist \emph{finite} sets
\begin{itemize}
\item $\cW_0(k,i)$,
\item $\mathcal{N}_{\pi^\prime}$ for every $\pi^\prime \in \cW_0(k,i)$,
\item and $E_{\tau}(\pi^\prime,\mathcal S)$ for every $\pi^\prime \in \cW_0(k,i)$ and every $\mathcal{S} \in \mathcal{N}_{\pi^\prime}$
\end{itemize}
such that
\begin{align*}
\mathbf{W}_{\tau}(k,i) = \bigcup_{\pi' \in \cW_0(k,i)}\, \bigcup_{\mathcal{S} \in \mathcal{N}_{\pi'}}\, \bigcup_{(a_0,\dots,a_{\mu}) \in E_{\tau}(\pi',\mathcal S)}\mathrm{con}_+(a_0;\dots,a_{\mu})\, . 
\end{align*} 
\end{lemma}

\noindent As opposed to Lemma 4.8, Lemma~\ref{thm.support-thm-1} uses cones $\mathrm{con}_+(a_0;\dots,a_{\mu})$ over positive integers, instead of cones over non-negative integers, defined as
\begin{align*}
\mathrm{con}_+(a_0;a_1,\dots,a_{\mu}) = \left\{ a_0 + \sum_{\alpha=1}^{\mu} n_{\alpha} \cdot a_{\alpha} \ \bigg| \ n_{\alpha} \in \mathbb{N} \text{ for all } 1 \leq \alpha \leq \mu \right\} \, ,
\end{align*}
that is, with positive integers $n_{\alpha}$. From this adjusted version of the statement, we will then derive Lemma 4.8 by sorting the involved cones over positive integers more economically into cones over non-negative integers.

Before we define the sets in Lemmas 4.8 and~\ref{thm.support-thm-1} in detail, we need to introduce additional notation on walks in directed graphs and discuss a few results on cycles.

\subsubsection{Operations on walks in directed graphs} \label{subsubsec.operations-walks}
Given a finite directed graph $\cG$, let $\cW(\cG)$ be the union of the set of walks $\cW(i, j)$ for all pairs of (not necessarily distinct) vertices $i$ and $j$ in $\cG$, with $\cW(i, j)$ as defined in Problem 3. We define the following operations on elements in $\cW(\cG)$.
\begin{itemize}
\item \textbf{Concatenation}: Given walks $\pi, \pi' \in \cW(\cG)$ with $\pi(\length{\pi}) = \pi'(1)$, we define their \emph{concatenation} as $\pi \circ \pi' = (\pi(1),\dots,\pi(\length{\pi}),\pi'(2),\dots, \pi'(\length{\pi'}))$.
\item \textbf{Removal}: Given a walk $\pi \in \cW(\cG)$ with subcycle $c = \pi(a,b)$ (that is, $a < b$ and $\pi(a) = \pi(b)$), we denote by $\pi \backslash c = (\pi(1),\dots, \pi(a),\pi(b+1),\dots, \pi(\length{\pi}))$ the walk obtained by collapsing $c$ on $\pi$ to the single vertex $\pi(a) = \pi(b)$. We call the index $a$ the \emph{removal point} of $c$ from $\pi$. 
\item \textbf{Insertion}: Given a walk $\pi \in \cW(\cG)$, an integer $i \in \{ 1,\dots, \length{\pi} \}$ and a directed cycle \newline $c = (c(1),\dots, c(m)) \in \cW(\cG)$ with $c(1) = c(m) = \pi(i)$, we let $\pi \cup_i c$ denote the walk $(\pi(1),\dots,\pi(i-1),c(1),\dots, c(m), \pi(i+1),\dots, \pi(\length{\pi}))$ obtained by inserting $c$ into $\pi$ at point $i$. 
\end{itemize}

\subsubsection{Cycles in directed graphs} \label{subsubsec.cycles}
Recall that the finite group $\mathbb{Z}_n$ acts on the set of cycles of length $n$ by revolving the vertices, that is, through the group action 
\[ \alpha_n(\Bar{k})\left((c(1),\dots,c(n))\right) = (c(\overline{1+k}),\dots,c(\overline{n+k}))\, \]
where $c$ is a cycle, $\Bar{k} \in \mathbb{Z}_n$ and $\Bar{\ell} = \ell\!\! \mod n \in \mathbb{Z}_n$. These group actions $\alpha_n$ induce an equivalence relation on the set of irreducible cycles such that two irreducible cycles $c_1$ and $c_2$ are equivalent (denoted as $c_1 \sim c_2$) if and only if, first, $c_1$ and $c_2$ have the same length $n$ and, second, there is $\Bar{k} \in \mathbb{Z}_n$ such that $c_2 = \alpha_n(\Bar{k})(c_1)$. We denote the set of all equivalence classes of \emph{irreducible} cycles in $\cW(\cG)$ by $\cC$ and the corresponding equivalence classes themselves in bold font, for example $\mathbf{c} \in \cC$. We denote the set of nodes on an equivalence class $\mathbf{c}$ by $\mathrm{nodes}(\mathbf{c})$, which is well-defined because all elements in the equivalence class $\mathbf{c}$ share the same nodes.\footnote{In Section 2.1, we defined a slightly different equivalence relation on the set of cycles according to which two cycles are also equivalent if one of them is the other read in the reverse direction (for example, $i\tailhead j \tailhead i$ and $i \headtail j \headtail i$). However, here we only consider cycles in the set of walks $\cW(\cG)$, which by definition only contains walks of the form $\tailhead \ldots \tailhead$ (and not $\headtail \ldots \headtail$). Thus, given the restriction to cycles in $\cW(\cG)$, the two definitions are in fact equivalent.}

It is convenient to define another graphical object, the \emph{graph of cycles}, which encodes how ``far'' different equivalence classes in $\mathcal{C}$ are from each other.

\begin{definition} \label{def.graph-of-cycles}
Let $\cG$ be a directed graph and let $\cC$ be the set of its irreducible cycle classes as defined in the previous paragraph. Then, we define the \emph{graph of cycles} $\cG_{\cC} = (\cC,\mathbf{U})$ as the \emph{undirected} graph with node set $\cC$ and undirected edges such that
\begin{align*}
(\c,\c') \in \mathbf{U} \qquad \text{if and only if } \qquad \c \neq \c' \quad \text{and} \quad \mathrm{nodes}(\c) \cap \mathrm{nodes}(\c') \neq \emptyset \, .
\end{align*}
\end{definition}

\begin{example} \label{ex.graph-of-cycles}
    Let $\cG$ be the directed graph with vertex set $\Vertices = \{1, 2, 3, 4\}$ given by $1 \to 2 \to3 \to4 \to3 \to2$. Then, $\mathcal{C} = \{\mathbf{c}_1, \mathbf{c}_2\}$ with $\c_1 = [(3,4,3)]$ (and $(4,3,4) \sim (3,4,3)$) and $\c_2 = [(2,3,2)]$ (and $(3,2,3) \sim (2,3,2)$). The graph of cycles $\cG_{\cC}$ is $\c_1 - \c_2$. 
\end{example}

The geodesic distance on the graph of cycles $\cG_{\cC}$ equips the set $\mathcal{C}$ with a metric $d$, that is,
\begin{align*}
d(\c,\c') = \begin{cases} \infty \quad \text{if} \quad \c, \c' \text{ are not connected by a path in $\cG$}\, ; \\
\text{length of the shortest (possibly trivial) path on } \cG_{\cC} \text{ between } \c \text{ and } \c' \, .
 \end{cases}
\end{align*}  

\noindent Given a subset $\mathcal{C}^\prime \subseteq \cC$, we define $d(\c,\mathcal{D})= \min_{\mathbf{c}^\prime \in \mathcal{C}^\prime} d(\c,\c')$. Then, we extend $d$ to a notion of distance between a cycle class $\c$ and a walk $\pi \in \cW(\cG)$ on the original graph $\cG$ by defining
\begin{align*}
d(\pi,\c) = 1 + d(\c, \mathrm{touch}(\pi))\, , 
\end{align*}
where we define the subset $\mathrm{touch}(\pi) \subseteq \cC$ by
\begin{align*}
\c' \in \mathrm{touch}(\pi) \qquad \text{if and only if} \qquad \mathrm{nodes}(\c') \cap \mathrm{nodes}(\pi) \neq \emptyset\, .
\end{align*}
In other words, $\mathrm{touch}(\pi)$ is the set of cycle classes that ``touch'' the walk $\pi$ in the sense of sharing a node with $\pi$.

\begin{continuance}{ex.graph-of-cycles}
For $\pi  = (1 \tailhead 2 \tailhead 3)$ in $\cG$ we have $\mathrm{touch}(\pi) = \{\c_1, \c_2\}$ and $d(\pi, \mathbf{c_2}) = 1$. For $\pi^\prime = (1\tailhead 2)$ we have $\mathrm{touch}(\pi^\prime) = \{\c_1\}$ and $d(\pi^\prime, \mathbf{c_2}) = 2$.
\end{continuance}

Next, for given $\pi \in \cW(\cG)$, we define the set $\mathcal{P}_{\pi}$ to consist of
\begin{itemize}
\item all (possibly trivial) undirected paths $\xi$ on the graph of cycles $\cG_{\cC}$ that start at a point $\xi(1)\in \mathrm{touch}(\pi)$ and
\item the empty path,
\end{itemize}
and then consider the set projection
\begin{align*}
\Lambda_{\pi}: \,\, &\mathcal{P}_{\pi} \to 2^{\cC}\\
&(\xi(1),\dots, \xi(\length{\xi})) \mapsto \{ \xi(1),\dots, \xi(\length{\xi}) \} \subseteq \cC
\end{align*}
that maps an element of $\mathcal{P}_{\pi}$ to its (possibly empty) set of nodes (which is consequently a subset of $\cC$). The range of this map generates the \emph{set monoid} $\mathcal{N}_{\pi}$ defined as the unique smallest subset $\mathcal{N}_{\pi} \subseteq 2^{\cC}$ such that 
\begin{itemize}
\item $\Lambda_{\pi}(\mathcal{P}_{\pi})  \subseteq \mathcal{N}_{\pi}$
and
\item if $\mathcal A, \mathcal B \in \mathcal{N}_{\pi}$, then $\mathcal A \cup \mathcal B \in \mathcal{N}_{\pi}$.
\end{itemize}
Note that $\emptyset \in \mathcal{N}_{\pi}$ because the empty path is in $\mathcal{P}_{\pi}$.

To prove Lemma~\ref{thm.support-thm-1}, we would like to define a procedure that iteratively removes irreducible cycles from walks until the remaining walk is free of cycles (or trivial). Since the choice of cycle that is to be removed is not unique, we choose to always remove the first irreducible cycle by convention. To see that this choice is well-defined, consider a directed walk $\pi$ from $i$ to $j$ on a directed graph $\cG$ such that $\pi$ contains a cycle. The number $b= \min \{b' ~\vert~ \pi(1,b') \text{ contains an irreducible cycle}\}$ is then well-defined. Moreover the subwalk $\pi(1,b)$ contains a unique irreducible cycle that ends at $\pi(b)$. Indeed, by definition of $b$ any subcycle of $\pi(1,b)$ must end at $\pi(b)$ and if $c = \pi(a,b)$ and $c' = \pi(a',b)$ with $a<a'$ were different irreducible subcycles then $c$ could not be irreducible as it must contain $c'$. Therefore, the following procedure is well-defined:

\begin{definition} \label{def.nesting-resolution-walk}
Let $\pi$ be a directed walk from $i$ to $j$ on a directed graph $\cG$. Then, the \emph{cycle resolution} of  $\pi$ is the finite (possibly empty) sequence of irreducible cycles $(c_1,\dots, c_{\nu})$ defined as follows:
\begin{itemize} \setlength\itemsep{0.5em}
\item $c_1$ is the first irreducible subcycle of $d_0=\pi$, that is, $c_1 = \pi(a_1,b_1)$ for some $a_1, b_1 \in \{1,\dots, \length{d_0} \}$ with $a_1 < b_1$ is an irreducible subcycle of $d_0(1, b_1)$ and there is no irreducible subcycle $c'$ of $d_0(1,b_1)$ besides $c_1$.
\item $c_i$ for $i=2,\dots, \nu$ is the first irreducible subcycle of $d_{i} = d_{i-1} \backslash c_{i-1}$, that is, $c_i = d_i(a_i,b_i)$ for some $a_i, b_i \in \{1,\dots, \length{d_i} \}$ with $a_i < b_i$ is an irreducible subcycle of $d_i(1, b_i)$ and there is no irreducible subcycle $c'_i$ on $d_i(1,b_i)$ besides $c_i$.
\item $d_{\nu} \backslash c_{\nu}$ is a cycle-free directed walk from $i$ to $j$ if $i \neq j$ and the trivial walk $(i)$ without edges if $i=j$.
\end{itemize}
We call the resulting cycle-free walk $\mathrm{bs}(\pi) = d_{\nu} \backslash c_{\nu}$ the \emph{base walk} of the resolution and say that $\pi$ \emph{resolves to} $\pi'$ if $\pi' = \mathrm{bs}(\pi)$. 
\end{definition}

\begin{continuance}{ex.graph-of-cycles}
The walk $\pi = (1 \tailhead 2 \tailhead 3 \tailhead 4 \tailhead 3 \tailhead 2)$ in $\cG$ has the cycle resolution $(c_1,c_2)$ with $c_1 = (3,4,3)$ and $c_2 = (2,3,2)$ and the base walk $\mathrm{bs}(\pi) = (1,2)$.
\end{continuance}

Let us record how insertion of cycles interacts with cycle resolutions.

\begin{lemma} \label{lem.cycle_insertion}
Let $\pi$ be a directed walk in $\cW(\cG)$ with cycle resolution $(c_1,\dots, c_{\nu})$ and in-between walks $(d_1,\dots, d_{\nu})$ as in Definition~\ref{def.nesting-resolution-walk}. Let $\c$ be an irreducible cycle class with $\mathrm{nodes}(\pi) \cap \mathrm{nodes}(\c) \neq \emptyset$, let $j = \min \{ k ~\vert~ \pi(k) \in \mathrm{nodes}(\c) \} $ be the first node of $\pi$ that lies on $\c$ and let $c \in \c$ be the unique member of $\c$ starting and ending at $\pi(j)$. 

Consider the walk $\tilde{\pi} = \pi \cup_j c$ and its cycle resolution $(\tilde{c}_1,\dots, \tilde{c}_{\mu})$. Then, both of the following statements hold:
\begin{itemize}
\item[(1)] $\mu = \nu+1$ and for any $\tilde{c}_l$ we have that $\tilde{c}_l = c_k$ for some $k=1,\dots, \nu$ or $\tilde{c}_l = c$.
\item[(2)] $\mathrm{bs}(\tilde{\pi}) = \mathrm{bs}(\pi)$. 
\end{itemize}
\end{lemma}

\begin{proof}[Proof of Lemma~\ref{lem.cycle_insertion}]
Write $c = (c(1),\dots, c(s))$ and $\pi = (\pi(1),\dots, \pi(m))$ and $\tilde{\pi} = (\tilde{\pi}(1) = \pi(1),\dots, \tilde{\pi}(j) = c(1),\dots, \tilde{\pi}(j+s) = c(s),\dots, \tilde{\pi}(m+s) = \pi(m))$. Moreover, denote the in-between walks of $\tilde{\pi}$ by $\tilde{d}_0,\dots, \tilde{d}_{\mu}$. We prove the claim by induction over $\nu$.

\textbf{Induction base case:} Suppose that $\nu = 0$. In this case, $\pi$ does not have any cycles and hence $\pi = \mathrm{bs}(\pi)$. 
Consider the first irreducible cycle $\tilde{c}_1 = \tilde{\pi}(a,b) $ of $\tilde{\pi}$. Then $b \leq j + s$ as otherwise $\tilde{c}_1$ would not be the first irreducible subcycle of $\tilde{\pi}$. We distinguish two cases.
\begin{itemize}
\item Case 1: If $b <j+s$, then we must have $b>j$ as otherwise $\tilde{c}_1$ would be a subcycle of $\pi$, which was assumed to be cycle-free. In addition, we must have $a < j$ as otherwise $\tilde{c}_1$ would be a proper subcycle of $c$ which was assumed to be irreducible. But if $a<j$ and $j\leq b < j+s$, then $\tilde{\pi}(a) = \tilde{\pi}(b) \in \mathrm{nodes}(\c)$ contradicts the definition of $j$ as the first touch point. In other words, $b <j+s$ is not possible. 
\item Case 2: If $b = j+s$, then we must have $a \leq j$ as otherwise $\tilde{c}_1$ would be a proper subcycle of $c$ which was assumed to be irreducible. Similarly, we cannot have $a < j$ as in this case $c$ would be a proper subcycle of $\tilde{c}_1$ again contradicting irreducibility. Thus, $a = j$ such that $\tilde{c}_1 = c$ and $\tilde{d}_1 = \pi = \mathrm{bs}(\pi)$ is cycle-free. The cycle resolution of $\tilde{\pi}$ is thus $(c)$ and has length $\nu +1 = 1$. 
\end{itemize}

\textbf{Induction Hypothesis:} Suppose we have proven the claim for all $\nu$ with $0 \leq \nu \leq \nu^\prime$.

\textbf{Induction Step:} Suppose that $\nu = \nu^\prime + 1$. Consider the first irreducible cycle $\tilde{c}_1 = \tilde{\pi}(a,b) $ of $\tilde{\pi}$. Then, as for the induction base case, $b \leq j + s$ as otherwise $\tilde{c}_1$ would not be the first irreducible subcycle of $\tilde{\pi}$. We distinguish three cases.
\begin{itemize}
\item Case 1: If $j< b <j+s$, then we must have $a<j$ by irreducibility of $c$. But then we again find that $\tilde{\pi}(a) = \tilde{\pi}(b)$ violates the defining property of $j$. Thus, the case $j< b <j+s$ is impossible. 
\item Case 2: If $b = j +s$, then we can argue exactly as for $\nu = 0$ that $a=j$ and thus $\tilde{c}_1 = c_1$ and $\tilde{d}_1 = \pi$, such that the cycle resolution of $\tilde{\pi}$ is $(c,c_1,\dots,c_{\nu})$ and $\mathrm{bs}(\tilde{\pi}) = \mathrm{bs}(\pi)$.
\item Case 3: If $b \leq j$, then $\tilde{c}_1$ must have been a subcycle of $\pi$ already, more precisely the first subcycle of $\pi$, that is, $\tilde{c}_1 = c_1$. Thus, the path $\pi_{\circ} = d_1$ has a cycle resolution  of length $\nu - 1 = \nu^\prime$, its first touch point with $\c$ is $j_{\circ} = j- \length{c_1}$ and $\mathrm{bs}(\pi) = \mathrm{bs}(\pi_{\circ})$ and $\tilde{\pi}\backslash c_1 = \pi_{\circ} \cup_{j_{\circ}} c$. Applying the induction hypothesis with $\pi_{\circ} $ replacing $\pi$ then completes this case.
\end{itemize}
We have thus proven the claim.
\end{proof}

Cycle resolutions offer a convenient way to decompose the multi-weight of a directed walk. The proof of the following lemma is immediate.

\begin{lemma} \label{lem.weight-decomp}
Let $\pi$ be a directed walk from $i$ to $j$ with cycle resolution $(c_1,\dots, c_{\nu})$ and base walk $\pi'$. Then,
\begin{align*}
\weightset{\pi} = \weightset{\pi'} + \sum_{i=1}^{\nu} \weightset{c_i}\, .
\end{align*}
In addition, for any $\c \in \cC$ and any $c,c' \in \c$, we have $\w(c) = \w(c')$, so that $\w(\c)= \w(c)$ is well-defined, that is, $\w(\c)$ does not depend on the choice of $c \in \c$.
\hfill $\qedsymbol$
\end{lemma}

Consider now the following map on the set of directed walks on a directed graph $\cG$:
\begin{align*}
&\Gamma: \mathcal{W}(\cG) \to 2^{\cC} \, \text{, where} \\
&\Gamma(\pi) = \{ \c \in \cC ~\vert~ \ \text{there is } c \in \c \text{ that appears on the cycle resolution of } \pi \} \subseteq \cC \, .
\end{align*}

\noindent Then, with $\cW(i, j)$ as in Problem 3 and $\cW_0(i,j)$ as in Lemma 4.8, we readily obtain the set decomposition

\begin{align*}
\cW(i,j) &= \bigcup_{\pi' \in \cW_0(i,j)} \{ \pi \in \cW(i,j) ~\vert~ \mathrm{bs}(\pi) = \pi'\} \\
&=  \bigcup_{\pi' \in \cW_0(i,j)}\, \bigcup_{\mathcal S \subseteq \cC} \{ \pi \in \cW(i,j) ~\vert~ \mathrm{bs}(\pi) = \pi' \text{ and } \Gamma(\pi) = \mathcal S \}\\
&= \bigcup_{(\pi',\mathcal S) \in \cW_0(i,j) \times 2^{\cC}} \{ \pi \in \cW(i,j) ~\vert~ \mathrm{bs}(\pi) = \pi' \text{ and } \Gamma(\pi) = \mathcal S \} \, .
\end{align*}

Note that the index set of this union is finite, so with this decomposition of $\cW(i,j)$ we are starting to get close to the claim of Lemma~\ref{thm.support-thm-1}. However, some values $(\pi',\mathcal S ) \in \cW_0(i,j) \times 2^{\cC}$ might never appear when resolving walks. In other words, the restriction of $\Gamma$ to the set $\{ \pi \in \cW(i,j) ~\vert~ \mathrm{bs}(\pi) = \pi'\}$ for a fixed $\pi' \in \cW_0(i,j)$ might not be surjective. We therefore call a pair $(\pi',\mathcal S ) \in \cW_0(i,j) \times 2^{\cC}$ \emph{admissible} if 
\[ \mathcal S  \in \mathrm{range}\left(\Gamma|_{\{ \pi \in \cW(i,j) \vert \ \mathrm{bs}(\pi) = \pi'\}}\right) \equiv \mathrm{adm(\pi')} \, ,\] and our next step is to compute the set $\mathrm{adm(\pi')} \subseteq 2^{\cC}$ for any given element $\pi' \in \cW_0(i,j)$.

\begin{lemma} \label{lem.monoids}
 For given $\pi' \in \cW_0(i,j)$, the equality of sets $\mathrm{adm(\pi')} = \mathcal{N}_{\pi'}$ holds.
\end{lemma}

\begin{proof}[Proof of Lemma~\ref{lem.monoids}]
\textbf{Inclusion $\mathrm{adm(\pi')} \subseteq \mathcal{N}_{\pi'}$:} Let $\mathcal{S} \in \mathrm{adm(\pi')}$. If $\mathcal{S} = \emptyset$, then $\mathcal{S} \in \mathcal{N}_{\pi'}$ because, as we noted above, $\mathcal{N}_{\pi'}$ contains the empty set. Thus, suppose that $\mathcal{S} \neq \emptyset$. We need to show that $\mathcal S$ is the union of nodes on one or multiple paths on the graph of cycles $\cG_{\mathcal{C}}$ that start at the touch set $\mathrm{touch}(\pi')$. By definition of $\mathrm{adm(\pi')}$, there must be $\pi \in \cW(i, j)$ that resolves to $\pi'$ and such that $\Gamma(\pi) = \mathcal{S}$. Consider the cycle resolution $(c_1,\dots,c_\mu)$ of such $\pi$ with in-between walks $d_{i+1}= d_{i}\backslash c_i$ and cycle equivalence classes $\mathbf{c}_i \ni c_i$. Note that $\mu \geq 1$ because $\mathcal{S}$ is non-empty. When deleting $c_i$ from $d_i$, the endpoints of $c_i$ are glued together at the incision point $j_i$.  We now show that for any $c_i$ there exists a path $\chi_i$ on the graph of cycles such that
\begin{enumerate}[label=(\alph*)]
\item $\chi_i$ starts at $\mathbf{c}_i$ and ends at $\mathrm{touch}(\pi')$ and
\item all nodes of $\chi_i$ are elements of $\{ \mathbf{c}_1,\mathbf{c}_2,\dots, \mathbf{c}_{\mu}\} = \Gamma(\pi) = \mathcal{S}$.
\end{enumerate}
If this claim is proven, then $\cup_{i=1}^{\mu} \mathrm{nodes}(\overline{\chi_i}) = \Lambda_{\pi^\prime}(\overline{\chi_i}) = \mathcal{S} \in \mathcal{N}_{\pi^\prime}$, where $\overline{\chi_i}$ is $\chi_i$ read in the reverse direction. Indeed, the inclusion from left to right holds by property (b) whereas the inclusion from right to left holds because for all $i$ we have $\mathbf{c}_i \in \mathrm{nodes}(\overline{\chi_i})$ according to property (b).

To prove the claim we used, consider the cycle $c_i$ and its incision point $j_i = c_i(1)$. Then, there are two options: 
\begin{enumerate}[label=(\roman*)]
\item Either $j_i$ is a node of $\pi'$, in which case $\mathbf{c}_i \in \mathrm{touch}(\pi')$ and we can define $\chi'_i$ as the trivial path $\chi'_i = (\mathbf{c}_i)$ consisting of one node.
\item Or, $j_i \notin \mathrm{nodes}(\pi')$, in which case $j_i$ must have been removed at some later steps of the cycle resolution. Hence, there must be a cycle $c_{k_1}$ with $ k_1 > i$ such that $j_i \in \mathrm{nodes}(c_{k_1})$. In particular, $c_i$ and $c_{k_1}$ share a node, so that either $\mathbf{c}_i = \mathbf{c}_{k_1}$ or $\mathbf{c}_i - \mathbf{c}_{k_1}$ is an edge in the graph of cycles. 
\end{enumerate}
We can now repeat this argument for the incision point $j_{k_1}$ of $c_{k_1}$, leading us
\begin{enumerate}
\item either to a path $\chi_i' = (\mathbf{c}_i - \mathbf{c}_{k_1})$, in case $c_{k_1}(1)  = j_{k_1} \in \mathrm{nodes}(\pi')$,
\item or to another cycle $c_{k_2}$ with $ k_2 > k_1$ such that $j_{k_1} \in \mathrm{nodes}(c_{k_2})$, in case $j_{k_1} \notin \mathrm{nodes}(\pi')$.
\end{enumerate}
In the latter case, we continue recursively. This recursion terminates because the cycle resolution is a finite sequence and we obtain a walk $\chi_i' = (\mathbf{c}_i - \mathbf{c}_{k_1} - \cdots - \mathbf{c}_{k_l})$ with $\mathbf{c}_{k_l} \in \mathrm{touch}(\pi')$. These walks $\chi_i'$ satisfy both property (a) and (b), and by removing cycles from these walks we obtain paths $\xi_i$ that too satisfy both (a) and (b).

\textbf{Converse inclusion $\mathrm{adm(\pi')} \supseteq \mathcal{N}_{\pi'}$:} Since $\pi'$ resolves to itself, we get that $\emptyset \in \mathcal{N}_{\pi'}$. Therefore, we can start with (possibly trivial) paths $\xi_1,\dots,\xi_r$ on the graph of cycles $\cG_{\cC}$ with $\xi_k(1) \in \mathrm{touch}(\pi')$ and such that $\cup_{i = 1}^r \Lambda_{\pi^\prime}(\xi_i) = \mathcal{T} \in \mathcal{N}_{\pi^\prime}$. Let us write the nodes of $\xi_1,\dots,\xi_r$, which are cycle classes, sequentially as
\[ L = (\xi_1(1),\dots, \xi_1(\length{\xi_1}),\dots,\xi_r(1),\dots, \xi_r(\length{\xi_r}))\, . \]
We insert the cycle class $\xi_1(1)$ into $\psi_0 = \pi'$ at position $i_1 = \min \left\{ i ~\vert~ \pi'(i) \in \mathrm{nodes}(\xi_1(1))\right\}$ to get $\psi_1 = \psi_0 \cup_{i_1} \xi_1(1)$.\footnote{To be more exact, we insert the unique representative of $\xi_1(1)$ that starts at $\pi'(i_1)$, but we will hide this detail to ease notation.} Similarly, we insert the $k$-th element $L_k$ of $L$ into $\psi_{k-1}$ at position $i_1 = \min \left\{ i ~\vert~ \psi_{k-1}(i) \in \mathrm{nodes}(L_k)\right\}$. This insertion is well-defined as $L_k$ always touches $\pi'$ or shares a node with $L_{k-1}$. By Lemma~\ref{lem.cycle_insertion}, we have $\mathrm{bs}(\psi_k) = \mathrm{bs}(\psi_{k-1})$ and that the set $\mathcal{S}_k$ of cycles in the cycle resolution of $\psi_k$ is $\mathcal{S}_{k-1} \cup \{ L_k \}$ (starting with $\mathcal{S}_0 = \emptyset$). Therefore, after running through all of $L$, we end up with a path $\psi_N$ with $\mathrm{bs}(\psi_N) = \mathrm{bs}(\psi_{N-1}) = \dots = \mathrm{bs}(\psi_0) = \pi'$ and $\Gamma(\psi_N) = \mathcal S_N = \bigcup_{i=1}^r \mathrm{nodes}(\xi_i) = \mathcal{T}$ as desired.
\end{proof}

We are now ready to prove Lemma~\ref{thm.support-thm-1}.

\begin{proof}[Proof of Lemma~\ref{thm.support-thm-1}]
Lemma~\ref{lem.monoids} and the discussions preceding Lemma~\ref{lem.monoids} show that $\cW(i,j)$ admits the decomposition
\begin{align*}
\cW(i,j) = \bigcup_{(\pi',\mathcal S) \in \cW_0(i,j) \times \mathcal{N}_{\pi'}} \{ \pi \in \cW(i,j) ~\vert~ \mathrm{bs}(\pi) = \pi' \text{ and } \Gamma(\pi) = \mathcal S \} \, .
\end{align*}
Thus, the set  
\begin{align*}
\mathbf{W}_{\tau}(i,j) = \bigcup_{\pi \in \cW(i,j) } \weightset{\pi} + \{ \tau \}
\end{align*}
decomposes as
\begin{align*}
\mathbf{W}_{\tau}(i,j) = \bigcup_{(\pi',\mathcal S) \in \cW_0(i,j) \times \mathcal{N}_{\pi'}} B_{\tau}((\pi',\mathcal S))
\end{align*}
where 
\begin{align*}
B_{\tau}((\pi',\mathcal S)) = \bigcup_{ \{\pi \in \cW(i,j) ~\vert~ \mathrm{bs}(\pi) = \pi' \text{ and } \Gamma(\pi) = \mathcal S \}} \weightset{\pi} + \{ \tau \} \, .
\end{align*}

Recall from Lemma~\ref{lem.weight-decomp} that the multi-weight of a walk in $\cW(i,j)$ is the sum of the multi-weight of its base walk and the multi-weights of the cycles in its cycle resolution. For $ \pi \in \cW(i,j)$ such that $\mathrm{bs}(\pi) = \pi' \text{ and } \Gamma(\pi) = \mathcal S$, both the base walk and the equivalence classes that appear in the cycle resolution are fully specified, and the only free parameters are the non-negative integers $n_{\c}$ for $\c \in \mathcal{S}$ that specify how often the class $\c$ appears in the cycle resolution. Note also that once $\c \in \mathcal{S}$ appears, inserting $\c$ arbitrarily many times into the path $\pi$ immediately after its first appearance will lead to a new path $\tilde{\pi}$ such that $\mathrm{bs}(\pi) = \mathrm{bs}(\tilde{\pi})$ and $\Gamma(\pi) = \Gamma(\tilde{\pi})$. Thus,
\begin{align*}
B_{\tau}((\pi',\mathcal S)) &= \{ \tau \}+ \weightset{\pi'} + \sum_{\c \in \mathcal{S}} \{ n_{\c} \cdot w(\c) ~\vert~ n_{\c} \geq 1, \ w(\c) \in \weightset{\c} \}\\
&= \bigcup_{(a_0,\dots,a_{\mu}) \in E_{\tau}((\pi',\mathcal S))}\mathrm{con}_+(a_0;\dots,a_{\mu}) 
\end{align*}
where the indexing tuple $(a_0;\dots,a_{\mu})$ in the second line runs over the finite set 
\begin{align*}
 E_{\tau}((\pi',\mathcal S)) = (\{ \tau \} + \weightset{\pi'} )\times \prod_{\c \in \mathcal{S}} \weightset{\c} \, .
\end{align*}
We have thus completed the proof.
\end{proof}

\subsubsection{Deriving Lemma 4.8 from Lemma~\ref{thm.support-thm-1}} \label{subsec.deriving-Lemma}
While Lemma~\ref{thm.support-thm-1} is sufficient to provide a finitary decomposition of $\mathbf{W}_{\tau}(i,j) $, this decomposition in terms of positive cones is not necessarily economical as the set $\mathcal{N}_{\pi'}$ is rather large. For instance, if the graph of cycles does not have any edges and, say, $k$ nodes all of which touch $\pi'$, then $\mathcal{N}_{\pi'}$ contains $2^k$ elements, each with one positive cone associated to it. In this example, we can efficiently summarize all of these positive cones into a single non-negative cone, which significantly reduces the number of Diophantine equations to be considered at a later stage. To implement this and similar efficiency improvement, we will introduce new sets $\mathcal{M}_{\pi'}$ that replace the sets $\mathcal{N}_{\pi'}$. Before doing so, we will need the concept of access points (defined below).

\begin{definition} \label{def.access-point}
Let $\Graph$ be an undirected graph with vertex set $\mathcal V$, let $\mathcal{S} \subseteq \mathcal V$, and let $v \in \mathcal V$ and $w \in \mathcal V\backslash\mathcal{S}$ be two nodes. We say that 
 \begin{itemize}
     \item $v$ is an $\mathcal{S}$-access point for $w$ if there exists a path $\xi(1) - \ldots - \xi(r-1) -\xi(r)$, where $r \geq 2$, such that $\xi(1) \in \mathcal{S}$ and $\xi(r-1) = v $ and $\xi(r) = w$.
     \item $v$ is an $\mathcal{S}$-access point if it is an $\mathcal{S}$-access point for some $w \in \mathcal V\backslash\mathcal{S}$.  
 \end{itemize}
 We denote the set of all $\mathcal{S}$-access points by $\mathcal{A}_{(\Graph,\mathcal{S})}$. Note that elements of $\mathcal{S}$ can, but do not have to be, $\mathcal{S}$-access points.
\end{definition}
Recall that, given a graph of cycles $\cG_{\cC}$ and a touch set $\mathrm{touch}(\pi')$, we defined $\mathcal{P}_{\pi'}$ as
\begin{itemize}
\item the set of all paths on $\cG_{\cC}$ that start at the touch set $\mathrm{touch}(\pi')$
\item plus the empty path.
\end{itemize}
We now define $\mathcal{Q}_{\pi'} \subseteq \mathcal{P}_{\pi'}$ such that $\xi \in \mathcal{P}_{\pi^\prime}$ is in $\mathcal{Q}_{\pi'}$ if and only if $\xi$ is the empty path or
\begin{itemize}
    \item all nodes of $\xi$, including its starting node $\xi(1)$, are $\mathrm{touch}(\pi')$-access points
    \item and the starting node $\xi(1)$ is the \emph{only} node on $\xi$ that belongs to the touch set $\mathrm{touch}(\pi')$.
\end{itemize}
In particular, $\mathcal{Q}_{\pi'}$ contains the empty path as well as all trivial paths through nodes in the intersection of $\mathrm{touch}(\pi')$ and the $\mathrm{touch}(\pi')$-access points. Moreover, if $\xi \in \mathcal{Q}_{\pi'}$, then $\xi(1)$ is an element of the touch set $\mathrm{touch}(\pi')$.
Since $\mathcal{Q}_{\pi'} \subseteq \mathcal{P}_{\pi'}$, we can restrict the node set projection $\Lambda_{\pi^\prime}$ defined above to $\mathcal{Q}_{\pi'}$. Thus, we can consider the set monoid $\mathcal{M}_{\pi'}$ generated by $\mathcal{Q}_{\pi'}$ that is, the set monoid defined as the unique smallest subset $\mathcal{M}_{\pi'} \subseteq 2^{\cC}$ such that 
\begin{itemize}
\item $\Lambda_{\pi'}(\mathcal{Q}_{\pi'}) \subseteq \mathcal{M}_{\pi'}$ and
\item if $\mathcal A, \mathcal B \in \mathcal{M}_{\pi'}$, then $\mathcal A \cup \mathcal B \in \mathcal{M}_{\pi'}$.
\end{itemize}

\begin{remark}
The second condition in the definition of the set $\mathcal{Q}_{\pi'}$, which ensures that $\xi$ does not return to the touch set, is solely there to keep the generating set of the monoid $\mathcal{M}_{\pi'}$ as small as possible. If we would drop this condition and instead work with a set of paths $\mathcal{Q}_{\pi'}'$ without this requirement, then the same set monoid $\mathcal{M}_{\pi'}$ would be generated. This claim follows because, given a path that returns to the touch set, we can split this path into subpaths starting at every return point and ending at the predecessor of the next return point. For instance, if $\c_1 - \c_2 - \c_3 - \c_4 - \c_5$ is a path of access points with $\c_1,\c_3,\c_4 \in \mathrm{touch}(\pi')$, then we can split this path into its subpaths $\c_1 - \c_2$ and $(\c_3)$ and $\c_4-\c_5$. 
\end{remark}

In addition, we define the \emph{closure} $\mathrm{cl}(\mathcal S)$ of a set $\mathcal{S} \in \mathcal{M}_{\pi'}$ as
\begin{align*}
    \mathrm{cl}(\mathcal S) = \Big[\mathcal S &\cup \mathrm{touch}(\pi')\\
    &\cup \{ \c \in \cC\backslash\mathrm{touch}(\pi') ~\vert~ \mathcal S \text{ contains a } \mathrm{touch}(\pi')\text{-access point for } \c  \}\Big] \, .
\end{align*}
In particular, the closure of the empty set is $\mathrm{cl}(\emptyset) = \mathrm{touch}(\pi').$ The following observation is straightforward.

\begin{lemma} \label{lem.closures}
Given arbitrary $\mathcal{S}_i \in \mathcal{M}_{\pi'}$ sets for $i=1,\ldots,m$, the equality of sets $\cup_i \mathrm{cl}(\mathcal{S}_i) = \mathrm{cl} \left( \cup_i \mathcal{S}_i \right)$ holds. \hfill $\qedsymbol$
\end{lemma}

To prove Lemma 4.8, we use the following auxiliary result.

\begin{lemma} \label{lem.advanced-path-decomposition}
Let $\psi \in \mathcal{P}_{\pi'}$ be a non-empty path. Then, one of the two following statements is true:
\begin{enumerate}
\item $\mathrm{nodes}(\psi) \subseteq \mathrm{touch}(\pi') = \mathrm{cl}(\emptyset)$.
\item There are non-empty subpaths $\chi^1,\dots, \chi^m \subseteq \psi$ of $\psi$ such that $\chi^1,\dots, \chi^m \in \mathcal{Q}_{\pi'}$ and $\mathrm{nodes}(\psi) \subseteq \mathrm{cl}\left( \cup_{k=1}^m \mathrm{nodes}(\chi^i)  \right)$.
\end{enumerate}
\end{lemma}

\begin{proof}[Proof of Lemma~\ref{lem.advanced-path-decomposition}]
Suppose that $\mathrm{nodes}(\psi) \centernot\subseteq \mathrm{touch}(\pi')$, that is, suppose for $\psi$ the first of the two statements in Lemma~\ref{lem.advanced-path-decomposition} is not true, and write $\psi=(\psi(1),\dots, \psi(r))$. Then, we first note that the set $\mathrm{nodes}(\psi)$ contains at least one $\mathrm{touch}(\pi')$-access point. This claim follows from the following argument: By assumption, there must be $1 \leq k \leq r$ such that $\psi(k) \notin \mathrm{touch}(\pi')$. Moreover since $\psi$ starts in $\mathrm{touch}(\pi')$ by definition of $\mathcal{P}_{\pi^\prime}$, we must have that $k \geq 2$ and that $\psi(k-1)$ is a $\mathrm{touch}(\pi')$-access point for $\psi(k)$.\\

\textbf{Claim 1}: Let $\psi(k_1)$ be the first $\mathrm{touch}(\pi')$-access point on $\psi$. Then, all of the nodes $\psi(1),\dots, \psi(k_1-1), \psi(k_1)$ are elements of $\mathrm{touch}(\pi')$.\\ 

By definition, $\psi(1) \in \mathrm{touch}(\pi')$. Thus, the claim follows if $k_1=1$, otherwise let $2 \leq l \leq k_1$. If we had that $\psi(l) \notin \mathrm{touch}(\pi')$, then $\psi(l-1)$ were a $\mathrm{touch}(\pi')$-access point with $l-1 < k_1$, thus contradicting the fact that $\psi(k_1)$ is the first such access point on $\psi$. Consequently, we have proven the above Claim 1. \\

Next, returning to the overall proof of Lemma~\ref{lem.advanced-path-decomposition}, consider the set 
\begin{align*}
A = \{ l ~\vert~  \ k_1 < l \leq r, \ \psi(l) \in \mathrm{touch}(\pi') \text{ or } \psi(l) \text{ is not a $\mathrm{touch}(\pi')$-access point} \} \, . 
\end{align*}
There are two options:
\begin{itemize}
\item \textbf{Case 1}: $A = \emptyset$. In this case, the path $\chi^1 = (\psi(k_1),\dots, \psi(r))$ is an element of $\mathcal{Q}_{\pi'}$ and $\mathrm{nodes}(\psi) \subseteq \mathrm{cl}(\mathrm{nodes}(\chi^1) )$. Indeed, $\chi^1$ starts in the touch set by definition of $k_1$ and, since $A$ is empty, all nodes on $\chi^1$ are $\mathrm{touch}(\pi')$-access and $\chi^1$ never returns to $\mathrm{touch}(\pi')$ after the first node, so that $\chi^1 \in \mathcal{Q}_{\pi'}$. Moreover, since all nodes of $\psi$ up to $\psi(k_1)$ belong to the touch set by the above Claim 1 and since all nodes on $\psi$ after $\psi(k_1)$ belong to $\mathrm{nodes}(\chi^1)$, the inclusion $\mathrm{nodes}(\psi) \subseteq \mathrm{cl}(\mathrm{nodes}(\chi^1) )$ follows by definition of the closure. Therefore, in this case, the second statement in Lemma~\ref{lem.advanced-path-decomposition} holds for $\psi$.
\item \textbf{Case 2}: $A \neq \emptyset$. In this case, set $l_1 = \min A$, implying that $l_1 > k_1$ and that $\psi(l_1)$ is a touch point or not a $\mathrm{touch}(\pi')$-access point. Consider the path $\chi^1 = (\psi(k_1),\dots, \psi(l_1-1))$, which is non-empty since $l_1 > k_1$. Since $l_1$ is the minimal element of $A$, all nodes of $\chi_1$ must be $\mathrm{touch}(\pi')$-access points and none of these nodes other than $\psi(k_1)$ belongs to the touch set. Therefore, $\chi^1 \in \mathcal{Q}_{\pi'}$. Moreover, observe that $\psi(l_1) \in \mathrm{cl}(\mathrm{nodes}(\chi^1))$: either $\psi(l_1) \in \mathrm{touch}(\pi')$ and the claim follows, or not, in which case $\psi(l_1-1)$ is a $\mathrm{touch}(\pi')$-access point for $\psi(l_1)$. As a consequence, using that all nodes on $\psi$ up to $\psi(k_1)$ belong to the touch set, we have that $\{ \psi(1),\dots, \psi(l_1) \} \subseteq \mathrm{cl}(\mathrm{nodes}(\chi^1))$.

If $l_1 = r$, then we have thus proven that the second statement in Lemma~\ref{lem.advanced-path-decomposition} holds for $\psi$. Therefore, we suppose that $l_1<r$ and consider the remaining path $(\psi(l_1),\dots, \psi(r))$. We consider two subcases:
\begin{itemize}
\item \textbf{Case 2a}: If $\psi(l_1) \in \mathrm{touch}(\pi')$, then we set $\psi^1 = (\psi(l_1),\dots, \psi(r))$. This non-empty path $\psi^1$ is in $\mathcal{P}_{\pi'}$ because $\psi(l_1) \in \mathrm{touch}(\pi')$. Moreover, $\psi^1$ is strictly shorter than $\psi$.
\item \textbf{Case 2b}: If $\psi(l_1) \notin \mathrm{touch}(\pi')$, then $\psi(l_1)$ is not a $\mathrm{touch}(\pi')$-access point according to the above definition of the set $A$ and $l_1 \in A$. We then argue that $\psi(l_1+1)$ must in the touch set. Indeed, if $\psi(l_1+1)$ was not in the touch point, then the path $(\psi(k_1),\dots, \psi(l_1),\psi(l_1+1))$ would make $\psi(l_1)$ a $\mathrm{touch}(\pi')$-access point for $\psi(l_1+1)$, a contradiction. Thus, the non-empty path $\psi^1 = (\psi(l_1+1),\dots, \psi(r))$ is an element $\mathcal{P}_{\pi'}$ and strictly shorter than $\psi$.
\end{itemize}
To summarize, in both subcases we get that $\mathrm{nodes}(\psi) \subseteq \mathrm{cl}(\mathrm{nodes}(\chi^1)) \cup \mathrm{nodes}(\psi^1) $ with $\chi^1 \in \mathcal{Q}_{\pi'}$ and $\psi^1 \in \mathcal{P}_{\pi'}$ with $len(\psi^1) < len(\psi)$. We can now reapply the whole argument with $\psi^1$ replacing $\psi$ and continue recursively to obtain paths $\chi^1,\dots, \chi^m \in \mathcal{Q}_{\pi'}$ such that $\mathrm{nodes}(\psi) \subseteq \mathrm{cl}\left(\cup_{k=1}^m \mathrm{nodes}(\chi^i)  \right)$, thus showing that the second statement of Lemma~\ref{lem.advanced-path-decomposition} holds for $\psi$.
\end{itemize}
We have thus completed the proof.
\end{proof}

Returning to our eventual goal of proving Lemma 4.8, for any non-negative integer $\tau \geq 0$ and any $\mathcal{S} \in \mathcal{M}_{\pi'}$ with any $\pi' \in \cW_0(k,i)$, we define the set
\begin{equation}\label{def.ctau}
\begin{split}
C_{\tau}((\pi',\mathcal S)) &= \{ \tau \} + \weightset{\pi'} + \weightset{\mathcal{S}} + \sum_{\c \in \mathrm{cl}(\mathcal{S})} \{ n_{\c} \cdot w(\c) ~\vert~ n_{\c} \geq 0, \ w(\c) \in \weightset{\c} \}  \\
&= \bigcup_{(a_0,\dots,a_{\mu}) \in D_{\tau}((\pi',\mathcal S))}\mathrm{con}(a_0;\dots,a_{\mu}) 
\end{split}
\end{equation}
where 
\begin{align} \label{def.tuple-sets}
 D_{\tau}((\pi',\mathcal S))= \left( \{ \tau \} + \weightset{\pi'} + \weightset{\mathcal{S}} \right) \times \prod_{\c \in \mathrm{cl}(\mathcal{S})} \weightset{\c} 
 \end{align}
and $ \weightset{\mathcal{S}} = \sum_{\c \in \mathcal{S}} \weightset{\c}$. We have now defined all sets involved in the statement of Lemma 4.8 and can finish its proof.

\begin{proof}[Proof of Lemma 4.8]
To deduce Lemma 4.8 from Lemma~\ref{thm.support-thm-1}, we only need to show that for every $\pi' \in \cW_0(k,i)$ and $\tau \geq 0$ the equality of sets
\begin{align*}
 \bigcup_{\mathcal{T} \in \mathcal{N}_{\pi'}} \,\,\bigcup_{(a_0,\dots,a_{\mu})  \in E_{\tau}(\pi',\mathcal T)}\mathrm{con}_+(a_0;\dots,a_{\mu}) = \bigcup_{\mathcal{S} \in \mathcal{M}_{\pi'}}\,\, \bigcup_{(a_0,\dots,a_{\mu})   \in D_{\tau}(\pi',\mathcal S)}\mathrm{con}(a_0;\dots,a_{\mu})
\end{align*}
or, equivalently, that the equality of sets
\begin{align*}
 \bigcup_{\mathcal{T} \in \mathcal{N}_{\pi'}} B_{\tau}((\pi',\mathcal T)) = \bigcup_{\mathcal{S} \in \mathcal{M}_{\pi'}} C_{\tau}((\pi',\mathcal S))
\end{align*}
holds.We show the set inclusions $\subseteq$ and $\supseteq$ separately. 

\textbf{Inclusion $\subseteq$}: Let $\mathcal{T} \in \mathcal{N}_{\pi'}$. First, we observe that $B_{\tau}((\pi',\emptyset)) = \{\tau \} + \weightset{\pi'}$ and that $C_{\tau}((\pi',\emptyset)) = \{\tau \} + \weightset{\pi'} +\sum_{\c \in \mathrm{touch}(\pi')} \{ n_{\c} \cdot w(\c) ~\vert~ n_{\c} \geq 0, \ w(\c) \in \weightset{\c} \}$, so that $B_{\tau}((\pi',\emptyset)) \subseteq C_{\tau}((\pi',\emptyset))$. Therefore, we can now assume that $\mathcal{T} \neq \emptyset$.

Our goal is to construct a set $\mathcal{S} \in \mathcal{M}_{\pi'}$ such that $B_{\tau}((\pi',\mathcal T)) \subseteq C_{\tau}((\pi',\mathcal S))$. By definition of the monoid $\mathcal{N}_{\pi'}$, we can write $\mathcal{T}$ as $\mathcal{T} = \cup_{j=1}^r \Lambda_{\pi'}(\psi_j) = \cup_{j=1}^r \mathrm{nodes}(\psi_j)$ for some non-empty paths $\psi_1,\dots,\psi_r \in \mathcal{P}_{\pi'}$ on the graph of cycles $\cG_{\cC}$ that start at the touch set. Consider the path $\psi_j$. By Lemma \ref{lem.advanced-path-decomposition}, we have that $\mathrm{nodes}(\psi_j) \subseteq \mathrm{touch}(\pi')$ or there exists a finite sequence of non-empty paths $\chi_1^j,\dots,\chi_{m_j}^j \in \mathcal{Q}_{\pi'}$, all of which are subpaths of $\psi_j$, such that $\mathrm{nodes}(\psi_j) \subseteq \mathrm{cl}( \cup_{k=1}^{m_j} \mathrm{nodes}(\chi_k^j)) $. We set
\begin{align*}
    \mathcal{S}= \bigcup_{j} \, \bigcup_{k=1}^{m_j} \mathrm{nodes}(\chi_k^j) \in \mathcal{M}_{\pi'} \, ,
\end{align*}
where $j$ runs over all indices for which $\mathrm{nodes}(\psi_j) \centernot\subseteq \mathrm{touch}(\pi')$ (if no such indices exist, then we set $\mathcal{S} = \emptyset \in \mathcal{M}_{\pi'}$). Using Lemma~\ref{lem.closures}, it then follows that
\begin{align*}
\mathrm{cl}(\mathcal{S})
&= \mathrm{touch}(\pi') \cup \mathrm{cl}(\mathcal{S})\\
&= \mathrm{touch}(\pi') \cup \bigcup_j\, \mathrm{cl} \left( \bigcup_{k=1}^{m_j} \mathrm{nodes}(\chi_k^j)\right)
\supseteq \mathrm{touch}(\pi') \cup \bigcup_j \,\mathrm{nodes}(\psi_j)
\supseteq \mathcal{T} \, .
\end{align*}
Moreover, we find that
\begin{align*}
\mathcal{S} = \bigcup_{j} \, \bigcup_{k=1}^{m_j} \mathrm{nodes}(\chi_k^j) \subseteq \bigcup_{j} \, \mathrm{nodes}(\psi^j) \subseteq \mathcal{T} \, .
\end{align*}
Now, recalling the notation $n_\c \cdot \weightset{\c} = \{ n_{\c} \cdot w(\c) ~\vert~ w(\c) \in \weightset{\c} \}$, any element $y \in B_{\tau}((\pi',\mathcal T)) $ is an element of
\begin{align*}
   \{\tau \} +  \weightset{\pi'} + \sum_{\c \in \mathcal{T}} n_\c \cdot \weightset{\c}
\end{align*}
 for some combination $(n_{\c})_{\c\in \mathcal{T}}$ of positive integers $n_\c \geq 1.$ Hence, we can conclude that
 \begin{align*}
   \{\tau \} +  \weightset{\pi'} + \sum_{\c \in \mathcal{T}} n_\c \cdot \weightset{\c} &=  \{\tau \} +  \weightset{\pi'} + \weightset{\mathcal{S}} \\ &+ \sum_{\c \in \mathcal{S}} (n_\c -1) \cdot \weightset{\c} +  \sum_{\c \in \mathcal{T}\backslash\mathcal{S}} n_\c \cdot \weightset{\c} \\
   &\subseteq C_{\tau}((\pi',\mathcal S))
\end{align*}
and thus $B_{\tau}((\pi',\mathcal T)) \subseteq  C_{\tau}((\pi',\mathcal S))$. Here, the `$=$' is just a rewriting of the left-hand-side that uses $\mathcal{S} \subseteq \mathcal{T}$ (as shown above), and the `$\subseteq$' follows because $\mathcal{S} \subseteq \mathrm{cl}(\mathcal{S})$ (by definition of the closure) and $\mathcal{T} \subseteq \mathrm{cl}(\mathcal{S})$ (as shown above).

\textbf{Converse Inclusion $\supseteq$}: Let $\mathcal{S} \in \mathcal{M}_{\pi'}$ and write $C_{\tau}((\pi',\mathcal S))$ as the union of sets of the form $ \{\tau \} +  \weightset{\pi'} + \weightset{\mathcal{S}} + \sum_{\c \in \mathrm{cl}(\mathcal{S})} n_\c \cdot \weightset{\c}$ where, this time, $(n_{\c})_{\c \in \mathcal{S}}$ is a finite sequence of non-negative integers $n_\c \geq 0$ (instead of positive integers). We can rewrite such a set as 
\begin{align*}
\{\tau \} +  \weightset{\pi'} + \sum_{\c \in \mathcal{S}} (n_\c+1) \cdot \weightset{\c} +  \sum_{\c \in \mathrm{cl}(\mathcal{S}) \backslash\mathcal{S}} n_\c \cdot \weightset{\c} \, .
\end{align*}
Consider the set $\mathcal{T} = \mathcal{S} \cup \{\c\in \mathrm{cl}(\mathcal{S})\backslash\mathcal{S} ~\vert~ n_{\c} \geq 1 \}$ and set $m_{\c} = n_\c+1$ if $\c \in \mathcal{S}$ and $m_{\c} = n_\c$ if $\c \in \mathrm{cl}(\mathcal{S})\backslash\mathcal{S}$ and $n_{\c} \geq 1$. In other words, we simply drop all those summands in the second sum on the right-hand-side for which $n_\c = 0$. Thus, the above set is equal to $\{\tau \} +  \weightset{\pi'} + \sum_{\c \in \mathcal{T}} m_\c \cdot \weightset{\c}$ with positive coefficients $m_\c \geq 1$. Therefore, we have finished the proof if we can argue that $\mathcal{T}$ is an element of the monoid $\mathcal{N}_{\pi'}$, as we do now.

If $\mathcal{S}$ is empty, then $\mathcal{T}$ is a subset of $\mathrm{touch}(\pi') = \mathrm{cl}(\emptyset)$, such that we can write $\mathcal{T}$ as a finite union of trivial paths in $\mathcal{P}_{\pi'}$, and hence $\mathcal{T} \in \mathcal{N}_{\pi^\prime}$. If $\mathcal{S} \neq \emptyset$, then, by definition of $\mathcal{M}_{\pi^\prime}$, we can write $\mathcal{S} = \cup_{j=1}^r \mathrm{nodes}(\xi_j)$ for some non-empty paths $\xi_j \in \mathcal{Q}_{\pi'} \subseteq  \mathcal{P}_{\pi'}$ with  $j=1,\dots,r $. For $\c' \in  \mathrm{cl}(\mathcal{S})\backslash\mathcal{S}$ there are two cases:
\begin{itemize}
\item If $\mathbf{c}' \in \mathrm{touch}(\pi')$, then we let $\psi^{\mathbf c'}$ be the trivial path $(\mathbf c')$, which is an element of $\mathcal{P}_{\pi'}$.
\item If $\mathbf{c}' \notin \mathrm{touch}(\pi')$, then, using the definition of the closure, we see that $\mathcal{S}$ contains a $\mathrm{touch}(\pi')$-access point $\tilde{\c}$ for $\c'$. Because  $\mathcal{S} = \cup_{j=1}^r \mathrm{nodes}(\xi_j)$, this access point $\tilde{\mathbf c}$ lies on $\xi_j$ for some $1 \leq j \leq r$, say $\tilde{\c} = \xi_j(k)$. Consequently, the path $\psi^{\c'} = (\xi_j(1),\dots, \xi_j(k),\c')$ exists and is an element of $\mathcal{P}_{\pi'}$.
\end{itemize}
Finally, we obtain that $\mathcal{T} = [\cup_{j=1}^r \mathrm{nodes}(\xi_j)] \cup [\cup_{\c' \in \mathcal{T}\backslash\mathcal{S}} \mathrm{nodes}(\psi^{\c'})] \in \mathcal{N}_{\pi'}$.
\end{proof}

\subsection{Proofs of all theorems}

\begin{proof}[Proof of Theorem 1]
We have already proven Theorem 1 in the main paper under the assumption that Theorem 2 holds. Below, we prove Theorem 2 without making use of Theorem 1 in that proof.
\end{proof}

\begin{proof}[Proof of Theorem 2]
We prove the statement by considering the five cases separately:
\begin{enumerate}
\item \underline{$\mu = 0$ and $\nu = 0$.}
In this case, eq.~(3) reduces to the trivial equation $c = 0$.

\item \underline{$\mu = 0$ and $\nu \neq 0$ and $c > 0$.}
If $c \!\! \mod g_{a^\prime} \neq 0$ with $\gcd(a^\prime_1, \, \ldots, \, a^\prime_\nu)$, then according to Lemma 4.12 there is no integer solution. In particular, there is then no non-negative integer solution, thus proving the ``only-if'' statement in the first sentence. For the ``if-and-only-if'' statement in the second sentence, observe that eq.~(3) reduces to $c = \sum_{\beta \,=\, 1}^{\nu} n^\prime_\beta \cdot a^\prime_\beta$ where $c > 0$ and $a^\prime_\beta > 0$ for all $1 \leq \beta \leq \nu$. The statement follows because every term in the sum on the right-hand-side of this equation is positive.

\item \underline{($\mu = 0$ and $\nu \neq 0$ and $c \leq 0$) or ($\mu \neq 0$ and $\nu = 0$ and $c \geq 0$).} We first consider the subcase $\mu = 0$ and $\nu \neq 0$ and $c \leq 0$. Then, eq.~(3) reduces to $c = \sum_{\beta \,=\, 1}^{\nu} n^\prime_\beta \cdot a^\prime_\beta$ where $c \leq 0$ and $a^\prime_\beta > 0$. Recalling that every term in the sum on the right-hand-side of this equation is positive, there is a non-negative solution only if $c = 0$. Conversely, if $c = 0$, then $n^\prime_1 = \dots = n^\prime_\nu = 0$ is a non-negative solution.

The subcase $\mu \neq 0$ and $\nu = 0$ and $c \geq 0$ is equivalent to the first subcase up to replacing $c$ with $-c$ and replacing $\sum_{\beta \,=\, 1}^{\nu} n^\prime_\beta \cdot a^\prime_\beta$ with $\sum_{\alpha \,=\, 1}^{\mu} n_\beta \cdot a_\alpha$.

\item \underline{$\mu \neq 0$ and $\nu = 0$ and $c < 0$.} This case is equivalent to case 2 up to replacing $c$ with $-c$ and replacing $\sum_{\beta \,=\, 1}^{\nu} n^\prime_\beta \cdot a^\prime_\beta$ with $\sum_{\alpha \,=\, 1}^{\mu} n_\beta \cdot a_\alpha$.

\item \underline{$\mu \neq 0$ and $\nu \neq 0$.} If $c \!\! \mod \! g_{aa^\prime} \neq 0$ with $g_{aa^\prime} = \gcd(a_1, \, \ldots, \, a_\mu, a^\prime_1, \, \ldots, \, a^\prime_\nu,)$, then according to Lemma 4.12 there is no integer solution. In particular, there is then no non-negative integer solution, thus proving the ``only-if'' part.

For the ``if'' part, suppose that $c \!\! \mod \! g_{aa^\prime} = 0$. We can then divide eq.~(3) by $g_{aa^\prime}$ to get the modified linear Diophantine equation
\begin{equation}\label{eq:linear-diophantine-gcd-1-in-proof}
\tilde{c} + \sum_{\alpha \,=\, 1}^{\mu} n_\alpha \cdot \tilde{a}_\alpha = \sum_{\beta \,=\, 1}^{\nu} m_\beta \cdot \tilde{a}_\beta^\prime \, ,
\end{equation}
where $\tilde{c} = \tfrac{c}{g_{aa^\prime}}$ is an integer and $\tilde{a}_\alpha = \tfrac{a_\alpha}{g_{aa^\prime}}$ for all $1 \leq \alpha \leq \mu$ as well as $\tilde{a}_\beta^\prime = \tfrac{a^\prime_\beta}{g_{aa^\prime}}$ for all $1 \leq \beta \leq \nu$ are positive integers. Note that $\gcd(\tilde{a}_1, \, \ldots, \, \tilde{a}_\mu, \, \tilde{a}^\prime_1, \, \ldots, \, \tilde{a}^\prime_\nu) = 1$ by definition of the $\tilde{a}_\alpha$ and $\tilde{a}^\prime_\beta$, and that $\gcd(\tilde{a}_1, \, \ldots, \, \tilde{a}_\mu, \, \tilde{a}^\prime_1, \, \ldots, \, \tilde{a}^\prime_\nu)$ by associativity of the greatest common divisor equals $\gcd(\gcd(\tilde{a}_1, \, \ldots, \, \tilde{a}_\mu), \, \gcd(\tilde{a}^\prime_1, \, \ldots, \, \tilde{a}^\prime_\nu))$. Thus, letting $\tilde{g}_a = \gcd(\tilde{a}_1, \, \ldots, \, \tilde{a}_\mu)$ and $\tilde{g}_{a^\prime} = \gcd(\tilde{a}^\prime_1, \, \ldots, \, \tilde{a}^\prime_\nu)$, we get $\gcd(\tilde{g}_a, \, \tilde{g}_{a^\prime}) = 1$. In addition, since all $\tilde{a}_1, \, \ldots, \, \tilde{a}_\mu$ and all $\tilde{a}^\prime_1, \, \ldots, \, \tilde{a}^\prime_\nu$ are positive, also both $\tilde{g}_a$ and $\tilde{g}_{a^\prime}$ are positive. Now consider the linear Diophantine equation
\begin{equation}\label{eq:simple-linear-diophantine-in-proof}
\tilde{c} + \tilde{g}_a \cdot k_a = \tilde{g}_{a^\prime} \cdot k_b
\end{equation}
with unknowns $k_a$ and $k_b$. Since $\gcd(\tilde{g}_a, \, \tilde{g}_{a^\prime}) = 1$, according to Lemma 4.12 there is at least one integer solution $(k_a, k_b) = (k_a^0, k_b^0)$ to this equation. Then, for any integer $q$, also $(k_a, k_b) = (k_a^0 + q \cdot \tilde{g}_{a^\prime}, k_b^0 + q \cdot \tilde{g}_a)$ is an integer solution to eq.~\eqref{eq:simple-linear-diophantine-in-proof}. Since both $\tilde{g}_a$ and $\tilde{g}_{a^\prime}$ are positive, for any pair of integers $k_a^{\text{min}}$ and $k_b^{\text{min}}$, we can choose $q$ sufficiently large such that $k_a^0 + q \cdot \tilde{g}_{a^\prime} \geq k_a^{\text{min}}$ and at the same time $k_b^0 + q \cdot \tilde{g}_a \geq k_b^{\text{min}}$. Thus, there is an integer $\tilde{q}$ such that $k_a^0 + \tilde{q} \cdot \tilde{g}_{a^\prime} \geq \tilde{g}_a \cdot f(\tfrac{\tilde{a}_1}{\tilde{g}_a}, \, \ldots, \, \tfrac{\tilde{a}_\mu}{\tilde{g}_a})$ and at the same time $k_b^0 + \tilde{q} \cdot \tilde{g}_a \geq \tilde{g}_{a^\prime} \cdot f(\tfrac{\tilde{a}^\prime_1}{\tilde{g}_{a^\prime}}, \, \ldots, \,\tfrac{\tilde{a}^\prime_\nu}{\tilde{g}_{a^\prime}})$. Since $\gcd(\tfrac{\tilde{a}_1}{\tilde{g}_a}, \, \ldots, \, \tfrac{\tilde{a}_\mu}{\tilde{g}_a}) = 1$ and $\gcd(\tfrac{\tilde{a}^\prime_1}{\tilde{g}_{a^\prime}}, \, \ldots, \,\tfrac{\tilde{a}^\prime_\nu}{\tilde{g}_{a^\prime}}) = 1$, the Frobenius numbers $ f(\tfrac{\tilde{a}_1}{\tilde{g}_a}, \, \ldots, \, \tfrac{\tilde{a}_\mu}{\tilde{g}_a})$ and $f(\tfrac{\tilde{a}^\prime_1}{\tilde{g}_{a^\prime}}, \, \ldots, \,\tfrac{\tilde{a}^\prime_\nu}{\tilde{g}_{a^\prime}})$ exist according to Lemma 4.13. Again according to Lemma 4.13, there are non-negative integer $n_1, \, \ldots, \, n_\mu$ such that $k_a^0 + \tilde{q} \cdot \tilde{g}_{a^\prime} = \sum_{\alpha \, = \,1}^\mu n_\alpha \cdot \tfrac{\tilde{a}_\alpha}{\tilde{g}_{a}}$ as well as non-negative integers $n^\prime_1, \, \ldots, \, n^\prime_\nu$ such that $k_b^0 + \tilde{q} \cdot \tilde{g}_a = \sum_{\beta \, = \,1}^\nu n^\prime_\beta \cdot \tfrac{\tilde{a}^\prime_\beta}{\tilde{g}_{a^\prime}}$. With these choices we arrive at
\begin{align}
\tilde{c} + \tilde{g}_a \cdot \left(\sum_{\alpha \, = \,1}^\mu n_\alpha \cdot \frac{\tilde{a}_\alpha}{\tilde{g}_a}\right) &= \tilde{g}_{a^\prime} \cdot \left(\sum_{\beta \, = \,1}^\nu n^\prime_\beta \cdot \frac{\tilde{a}^\prime_\alpha}{\tilde{g}_{a^\prime}}\right) \\
\Leftrightarrow \qquad c +\sum_{\alpha \, = \,1}^\mu n_\alpha \cdot a_\alpha &= \sum_{\beta \, = \,1}^\nu n^\prime_\beta \cdot a^\prime_\beta\, .
\end{align}
The equivalence of these two equations follows by multiplication with $g_{aa^\prime}$ because $g_{aa^\prime} \cdot \tilde{a}_\alpha = a_\alpha$ and $g_{aa^\prime} \cdot \tilde{a}^\prime_\beta = a^\prime_\beta$ and $c = g_{aa^\prime} \cdot \tilde{c}$. Thus, $n_1, \, \ldots , \, n_\mu, \, m_1, \, \ldots , \, m_\nu$ is a non-negative integer solution to original linear Diophantine equation~(3).
\end{enumerate}
We have thus proven all of the five cases and thereby completed the proof.
\end{proof}

\begin{definition}[Notation for the below statements and proofs]\label{def:notation-proof-theorem-3}
Given a tuple $(a_0; a_1,\ldots,a_{\mu})$ with $\mu \geq 1$ where $a_0 \in \mathbb{N}_0$ and $a_\alpha \in \mathbb{N}$ for all $1 \leq \alpha \leq \mu$, we write $\mathbf{a} = (a_0;a_1,\dots,a_{\mu})$ and $g(\mathbf{a}) = \gcd(a_1,\ldots,a_{\mu})$. Note that the indices inside the $\gcd$ start with $1$, that is, the $\gcd$ does not include $a_0$ as argument.
\end{definition}

\begin{lemma} \label{lem.first-estimates}
For a tuple $\mathbf{a} = (a_0,a_1,\ldots,a_{\mu})$ as in Definition~\ref{def:notation-proof-theorem-3}, let the positive integer $D \in \mathbb N$ be such that
\begin{align*}
 (D-a_0) \!\!\!\!\!\mod g(\mathbf{a}) = 0 \quad \text{and} \quad D > a_0 + g(\mathbf{a})\cdot f\left(\tfrac{a_1}{g(\mathbf{a})},\ldots, \tfrac{a_{\mu}}{g(\mathbf{a})}\right)
\end{align*}
where $f(\cdot)$ denotes the Frobenius number (cf.~Lemma 4.13). Then, $D \in \mathrm{con}(a_0;a_1,\ldots,a_{\mu})$, with $\con{\cdot}$ as defined by eq.~(1).
\end{lemma}

\begin{proof}[Proof of Lemma~\ref{lem.first-estimates}]
This claim is an immediate consequence of the defining property of the Frobenius number: By assumption, $\tfrac{D-a_0}{g(\mathbf{a})}$ is a positive integer larger than $f(\tfrac{a_1}{g(\mathbf{a})},\ldots, \tfrac{a_{\mu}}{g(\mathbf{a})})$. Therefore, by the definition of the Frobenius property, there are non-negative integers $n_1,\ldots,n_{\mu}$ such that $\tfrac{D-a_0}{g(\mathbf{a})} = \sum_{\alpha=1}^{\mu} n_{\alpha} \cdot \tfrac{a_{\alpha}}{g(\mathbf{a})}$, which rearranges to $D = a_{0}+ \sum_{\alpha=1}^{\mu} n_{\alpha} \cdot a_{\alpha} \in \mathrm{con}(a_0;a_1,\ldots,a_{\mu})$.
\end{proof}

\begin{corollary}\label{cor:b1aaprime}
For tuples $\mathbf{a} = (a_0,a_1,\ldots,a_{\mu})$ and $\mathbf{a}' = (a'_0,a'_1,\ldots,a'_{\nu})$ as in Definition~\ref{def:notation-proof-theorem-3}, suppose there is a non-negative integer $D \in \mathbb N$ such that
\begin{align*}
(D-a_0) \!\!\!\!\!\mod g(\mathbf{a}) = 0 \quad \text{and} \quad (D-a^\prime_0) \!\!\!\!\!\mod g(\mathbf{a}^\prime) = 0
\end{align*}
as well as
\begin{align*}
  D > \max \left[ a_0 + g(\mathbf{a})\cdot f\left(\tfrac{a_1}{g(\mathbf{a})},\ldots, \tfrac{a_{\mu}}{g(\mathbf{a})}\right), \, a_0' + g(\mathbf{a}')\cdot f\left(\tfrac{a'_1}{g(\mathbf{a}')},\ldots, \tfrac{a'_{\nu}}{g(\mathbf{a})}\right)  \right] \, .
\end{align*}
Then, $D \in \mathcal{B}_1(\mathbf{a},\mathbf{a}') \equiv \mathrm{con}(a_1;a_1,\ldots,a_{\mu}) \cap \mathrm{con}(a'_0;a'_1,\ldots,a'_{\nu})$.
\end{corollary}

\begin{proof}[Proof of Corollary~\ref{cor:b1aaprime}]
This claim immediately follows from Lemma~\ref{lem.first-estimates}.
\end{proof}

\begin{lemma} \label{lem.D-upper-estimate}
For tuples $\mathbf{a} = (a_0,a_1,\ldots,a_{\mu})$ and $\mathbf{a}' = (a'_0,a'_1,\ldots,a'_{\nu})$ as in Definition~\ref{def:notation-proof-theorem-3}, suppose \newline $(a_0'-a_0)\!\!\! \mod g(\mathbf{a},\mathbf{a'})= 0$ where $g(\mathbf{a},\mathbf{a'}) = \gcd(a_1,\ldots,a_{\mu},a'_1,\ldots,a'_{\nu})$.
Then, there exists $D \in \mathcal{B}_1(\mathbf{a},\mathbf{a}')$ such that
\begin{align} \label{ineq.D-bound}
 D \leq \max(a_0,a_0') + \frac{g(\mathbf{a}')\cdot g(\mathbf{a})}{g(\mathbf{a},\mathbf{a'})} &\cdot \Bigg[\frac{|a_0'-a_0|}{g(\mathbf{a},\mathbf{a'})} + \nonumber \\
&\max\left[ f\left(\tfrac{a_1}{g(\mathbf{a})},\ldots, \tfrac{a_{\mu}}{g(\mathbf{a})}\right),f\left(\tfrac{a'_1}{g(\mathbf{a}')},\ldots, \tfrac{a'_{\nu}}{g(\mathbf{a})}\right)\right] +1 \Bigg] \, .
 \end{align}
\end{lemma}

\begin{proof}[Proof of Lemma~\ref{lem.D-upper-estimate}]
Due to symmetry, we can without loss of generality assume that $a_0 \leq a_0'$. Then, $\max(a_0,a_0') =a_0'$. By Lemma~\ref{lem.first-estimates}, we have
\begin{align*}
\underbrace{(D_0(\mathbf{a}),\infty)}_{\text{open interval}} \cap \left\{a_0 + g(\mathbf{a}\} \cdot \mathbb Z \right\} \subseteq \mathrm{con}(a_0;a_1,\ldots,a_{\mu})\, ,
\end{align*}
where 
\begin{align*}
D_0(\mathbf{a}) = a_0 + g(\mathbf{a})\cdot \left[ f\left(\tfrac{a_1}{g(\mathbf{a})},\ldots, \tfrac{a_{\mu}}{g(\mathbf{a})}\right) +1 \right] \, ,
\end{align*}
and an analogous result holds for $\mathbf{a}'$ instead of $\mathbf{a}$. Therefore, $\mathcal{B}_1(\mathbf{a},\mathbf{a}')$ contains the set
\begin{align*}
(D_0(\mathbf{a}) ,\infty) \cap (D_0(\mathbf{a}') ,\infty) \cap \mathcal{B}_2(\mathbf{a},\mathbf{a}) \subseteq \mathcal{B}_1(\mathbf{a},\mathbf{a}') \, ,
\end{align*}
where
\begin{align*}
\mathcal{B}_2(\mathbf{a},\mathbf{a}) = \left\{a_0 + g(\mathbf{a}) \cdot \mathbb Z \right\} \cap \left\{a'_0 + g(\mathbf{a}') \cdot \mathbb Z \right\}\, . 
\end{align*}
Using the assumption that $\gcd(g(\mathbf{a}),g(\mathbf{a}')) = g(\mathbf{a},\mathbf{a'})$ divides $(a_0'-a_0)$, we see that the set $\mathcal{B}_2(\mathbf{a},\mathbf{a})$ is non-empty because the equation
\begin{align} \label{eq.help-short-Diophantine}
g(\mathbf{a}) \cdot n - g(\mathbf{a}') \cdot m  = a_0'-a_0.
\end{align}
admits a solution $(n,m) \in \mathbb{Z} \times \mathbb{Z}$. Namely, if $u,v \in \mathbb Z$ are any integers for which
\begin{align*}
g(\mathbf{a}) \cdot u - g(\mathbf{a}') \cdot v  = g(\mathbf{a},\mathbf{a'})
\end{align*} 
(noting that $g(\mathbf{a},\mathbf{a'}) = \gcd(g(\mathbf a), g(\mathbf a'))$, such a pair of integers exists according to Bézout's identity), then the solutions of eq.~\eqref{eq.help-short-Diophantine} are pairs $(n,m)$ of integers of the form 
\begin{align*}
n = \frac{a_0'-a_0}{g(\mathbf{a},\mathbf{a'})}\cdot u + t \cdot \frac{g(\mathbf{a'})}{g(\mathbf{a},\mathbf{a'})} \, , \quad m = \frac{a_0'-a_0}{g(\mathbf{a},\mathbf{a'})}\cdot v + t \cdot \frac{g(\mathbf{a})}{g(\mathbf{a},\mathbf{a'})}
\end{align*}
for arbitrary $t \in \mathbb{Z}$. Note that for $v$ we can choose a value $v_0$ such that $0 \leq v_0 \leq \tfrac{g(\mathbf{a})}{g(\mathbf{a},\mathbf{a'})}$, which we will do going forward. Now choose
\begin{align*}
t_0 &= \max\left[ f\left(\tfrac{a_1}{g(\mathbf{a})},\ldots, \tfrac{a_{\mu}}{g(\mathbf{a})}\right),\,f\left(\tfrac{a'_1}{g(\mathbf{a}')},\ldots, \tfrac{a'_{\nu}}{g(\mathbf{a})}\right)\right] +1 \, ,\\
m_0 &= \frac{a_0'-a_0}{g(\mathbf{a},\mathbf{a'})}\cdot v_0 + t_0 \cdot \frac{g(\mathbf{a})}{g(\mathbf{a},\mathbf{a'})}\, , \\
D &= a_0' + g(\mathbf{a}')\cdot m_0 \, .
\end{align*}
Then, the choice of $m_0$ implies that $D \in \mathcal{B}_2(\mathbf{a},\mathbf{a})$. Moreover, since $m_0 \geq t_0 \cdot \tfrac{g(\mathbf{a})}{g(\mathbf{a},\mathbf{a'})}$, we have 
\begin{align*}
D \geq a_0' + \frac{g(\mathbf{a}')\cdot g(\mathbf{a})}{g(\mathbf{a},\mathbf{a'})} \cdot t_0 \geq D_0(\mathbf{a'})
\end{align*}
and 
\begin{align*}
D \geq a_0' + \frac{g(\mathbf{a}')\cdot g(\mathbf{a})}{g(\mathbf{a},\mathbf{a'})} \cdot t_0 \geq a_0 + \frac{g(\mathbf{a}')\cdot g(\mathbf{a})}{g(\mathbf{a},\mathbf{a'})} \cdot t_0 \geq  D_0(\mathbf{a}).
\end{align*}
We have thus shown that
\begin{align*}
D \in (D_0(\mathbf{a}) ,\infty) \cap (D_0(\mathbf{a}') ,\infty) \cap \mathcal{B}_2(\mathbf{a},\mathbf{a}) \subseteq \mathcal{B}_1(\mathbf{a},\mathbf{a}')\, .
\end{align*}
Finally, we need to show that $D$ satisfies inequality~\eqref{ineq.D-bound}. Indeed, since $v_0 \leq \tfrac{g(\mathbf{a})}{g(\mathbf{a},\mathbf{a'})}$, we have
\begin{align*}
g(\mathbf{a}')\cdot m_0 &\leq \frac{g(\mathbf{a})g(\mathbf{a}')}{g(\mathbf{a},\mathbf{a'})} \cdot \left( \frac{a_0'-a_0}{g(\mathbf{a},\mathbf{a'})} + t_0 \right) \\
&= \frac{g(\mathbf{a})g(\mathbf{a}')}{g(\mathbf{a},\mathbf{a'})} \cdot \left\{ \frac{a_0'-a_0}{g(\mathbf{a},\mathbf{a'})} + \max\left[ f\left(\tfrac{a_1}{g(\mathbf{a})},\ldots, \tfrac{a_{\mu}}{g(\mathbf{a})}\right),f\left(\tfrac{a'_1}{g(\mathbf{a}')},\ldots, \tfrac{a'_{\nu}}{g(\mathbf{a})}\right)\right] +1 \right\}.
\end{align*} 
Since $\max(a_0,a_0') =a_0'$ and $|a_0'-a_0| = a_0'-a_0$, inequality~\eqref{ineq.D-bound} follows.
\end{proof}

\begin{lemma} \label{lem.main-p-estimate}
Let $\tau,\tau' \geq 0$, let $i,j,k$ be (not necessarily distinct) vertices of a multi-weighted graph $(\cG,\w)$, let $\pi \in \cW_0(k,i)$ and $\pi' \in \cW_0(k,j)$ be cycle-free walks, and let $\cS \in \cM_{\pi}$ and $\cS' \in \cM_{\pi'}$. Consider tuples $\mathbf{a} = (a_0,a_1,\ldots,a_{\mu}) \in D_{\tau}(\pi,\cS)$ and $\mathbf{a}' = (a'_0,a'_1,\ldots,a'_{\nu}) \in D_{\tau'}(\pi',\cS')$ for which $\mu,\nu \geq 1$ and $ \mathcal{B}_1(\mathbf{a},\mathbf{a}') \neq \emptyset$. Then, there exists an element $D \in \mathcal{B}_1(\mathbf{a},\mathbf{a}')$ such that
\begin{align*}
D \leq \left(K^2+1\right) \cdot \left[\max\left(\tau,\tau'\right) + L +M \right] + K\left[\left(K-1\right)^2 +1\right] \, .
\end{align*}
\end{lemma}

\begin{proof}[Proof of Lemma~\ref{lem.main-p-estimate}]
Using Lemma 4.12, we see that the assumption $\mathcal{B}_1(\mathbf{a},\mathbf{a}') \neq \emptyset$ implies that $g(\mathbf{a},\mathbf{a}')$ divides $(a_0'-a_0)$. Then, by Lemma~\ref{lem.D-upper-estimate}, there exists $D \in \mathcal{B}_1(\mathbf{a},\mathbf{a}')$ that satisfies inequality~\eqref{ineq.D-bound}. We use the bounds (using the quantities defined in Theorem 3)
\begin{align*}
g(\mathbf{a}) \leq \max_{i=1,\ldots, \mu} a_i \leq K \qquad &\text{and} \qquad g(\mathbf{a}') \leq \max_{i=1,\ldots, \nu} a'_i  \leq K  \quad \text{and}\\
a_0 \leq \tau + L + M \qquad &\text{and} \qquad a_0' \leq \tau' + L +M \, ,
\end{align*}
which directly follow from the definition of the sets $D_{\tau}(\pi,\cS)$ and $D_{\tau'}(\pi',\cS')$. The latter two inequalities in particular imply that 
\begin{align*}
\max(a_0,a_0') \leq \max(\tau,\tau') + L +M
\end{align*}
and thus
\begin{align*}
&\max(a_0,a_0') + \frac{g(\mathbf{a}')\cdot g(\mathbf{a})}{g(\mathbf{a},\mathbf{a'})} \cdot \frac{|a_0'-a_0|}{g(\mathbf{a},\mathbf{a'})} \\
&\leq (K^2 +1) \cdot \max(a_0,a_0') \\
&\leq (K^2+1) \cdot \left(\max(\tau,\tau') + L+M\right)\, .
\end{align*}
Using the Brauer-Schur bound on the Frobenius number, see Remark 4.14, we get
\begin{align*}
 &\frac{g(\mathbf{a}')\cdot g(\mathbf{a})}{g(\mathbf{a},\mathbf{a'})} \cdot \left[f\left(\tfrac{a_1}{g(\mathbf{a})},\ldots, \tfrac{a_{\mu}}{g(\mathbf{a})}\right) +1 \right] \\
 &\leq g(\mathbf{a}') \cdot g(\mathbf{a}) \cdot \left[ \left( \max_{i=1,\ldots, \mu} \tfrac{a_i}{g(\mathbf{a})} -1\right) \left( \min_{i=1,\ldots, \mu} \tfrac{a_i}{g(\mathbf{a})} -1\right) +1\right] \\
 &\leq  g(\mathbf{a}') \cdot \left[ \left( \max_{i=1,\ldots, \mu} a_i -1\right) \left( \min_{i=1,\ldots, \mu} \tfrac{a_i}{g(\mathbf{a})} -1 \right)  +1\right] \\
  &\leq  g(\mathbf{a}') \cdot \left[ \left( \max_{i=1,\ldots, \mu} a_i -1 \right)^2  +1\right] \\
 &\leq K \cdot \left[\left(K-1\right)^2 +1\right] \, . 
\end{align*}
Combining inequality~\eqref{ineq.D-bound} and the previous bounds, the claim follows.
\end{proof}

\begin{proof}[Proof of Theorem 3]
The ``if'' direction is trivial. For the ``only-if'' direction, assume that $(i, t-\tau)$ and $(j,t)$ have a common ancestor $(k,t-\tau_k)$ in $\DAG$. Let $\mathcal{S}_{\w}(\DAG)$ be the multi-weighted summary graph of $\DAG$. According to Theorem 1, there are tuples $\mathbf{a} = (a_0,\ldots,a_{\mu}) \in D_{\tau}(\pi_0, \mathcal{S})$ and $\mathbf{a}' = (a'_0,\ldots,a'_{\nu}) \in D_{0}(\pi_0', \mathcal{S}')$ with $\pi_0$, $\mathcal{S}$, $\pi_0'$ and $\mathcal{S}'$ as in that theorem, such that $\mathcal{B}_1(\mathbf{a},\mathbf{a}') \neq \emptyset$. Here, $\mu =0$ and/or $\nu = 0$ is allowed, and we need to distinguish the following three mutually exclusive and collectively exhaustive cases:
\begin{enumerate}
    \item \underline{$\nu =0$:} In this case, $\tau_k = a_0' \in \w(\pi_0')+\w(\mathcal{S}')$. Thus, $\tau_k \leq L+M \leq p'(\DAG)$.
    \item \underline{$\nu \neq 0$ and $\mu = 0$:} In this case, $\tau_k = \tau + a_0 \in \{\tau\} + \w(\pi_0)+\w(\mathcal{S})$. Thus, $\tau_k \leq \tau + L+M \leq p'(\DAG)$.
    \item \underline{$\nu \neq 0$ and $\mu \neq 0$:} In this last case, Lemma~\ref{lem.main-p-estimate} applies (setting $\tau' = 0$), and we get the existence of a common ancestor $(k,t-\tau_k^\prime)$ of $(i, t-\tau)$ and $(j,t)$ with $\tau_k^\prime \leq p'(\DAG)$. 
\end{enumerate}
\end{proof}

\section{Pseudocode}\label{sec:pseudocode}
Here, we provide pseudocode for our solution of Problem 1 (construction of finite marginal ts-ADMGs) and Problem 2 (common-ancestor search in ts-DAGs) by means of the number-theoretic approach of Theorem 1 and Theorem 2. The pseudocode stays close to the theoretical results that it is based on and, for the purpose of simplicity, ignores practically important aspects such as memoization.

\begin{algorithm}[tbp]
    \caption{Pseudocode for the function \textsc{ts-admg}.}
    \label{algo:ts-ADMG}
    \begin{algorithmic}[1]
        \Require Time series ADMG $\ADMG$ with variable index set $\VarIndices$, non-empty subset $\VarIndicesObserved \subseteq \VarIndices$ corresponding to the observable component time series, length of the observed time window $\ptimewindow \in \mathbb{N}_0$
        \Ensure Marginal ts-ADMG $\tsADMG{\ObservedVertices}{\ADMG}$ where $\ObservedVertices = \VarIndicesObserved \times \TimeIndicesObserved$ with $\TimeIndicesObserved = \{t-\tau ~\vert~ 0 \leq \tau \leq \ptimewindow\}$
        \Function{ts-admg}{$\ADMG$, $\VarIndicesObserved$, $\ptimewindow$}
            \State $\ObservedVertices$ $\leftarrow$ $\VarIndicesObserved \times \TimeIndicesObserved$ with $\TimeIndicesObserved = \{t-\tau ~\vert~ 0 \leq \tau \leq \ptimewindow\}$ \Comment{Set of observed vertices, see Definition 3.1.}
            \State $\DAG$ $\leftarrow$ $\CanonicaltsDAG{\ADMG}$ \Comment{Canonical ts-DAG of $\ADMG$, see Definition 3.4.}
            \State $\ADMG^\prime$ $\leftarrow$ \textsc{simple-ts-admg($\DAG$, $\ptimewindow$)} \Comment{Simple marginal ts-ADMG of $\DAG$, see Section 3.3.2.}
            \State $\ADMG_{out}$ $\leftarrow$ $\ADMGprojection{\ObservedVertices}{\ADMG^\prime}$ \Comment{ADMG latent projection of $\ADMG^\prime$ to $\ObservedVertices$, see Section 3.3.2.}
            \State \Return $\tsADMG{\ObservedVertices}{\ADMG} = \ADMG_{out}$
        \EndFunction
    \end{algorithmic}
\end{algorithm}

Algorithm~\ref{algo:ts-ADMG} provides pseudocode for the top-level function \textsc{ts-admg} that returns the finite marginal ts-ADMG $\tsADMG{\ObservedVertices}{\ADMG}$ of an infinite ts-ADMG $\ADMG$, where $\ObservedVertices = \VarIndicesObserved \times \TimeIndicesObserved$ with $\TimeIndicesObserved = \{t-\tau ~\vert~ 0 \leq \tau \leq \ptimewindow\}$. This function implements the reduction steps explained in Sections 3.3.1 and 3.3.2, see also Figure 4, and calls the function \textsc{simple-ts-admg} that performs the simple marginal ts-ADMG projection.

\begin{algorithm}[tbp]
    \caption{Pseudocode for the function \textsc{simple-ts-ADMG}.}
    \label{algo:simple-ts-ADMG}
    \begin{algorithmic}[1]
        \Require Time series DAG $\DAG$ with variable index set $\VarIndices$, observed time window length $\ptimewindow \in \mathbb{N}_0$
        \Ensure Simple marginal ts-ADMG $\tsADMG{\ObservedVertices^\prime}{\DAG}$ with $\ObservedVertices^\prime = \VarIndicesObserved \times \TimeIndicesObserved$ with $\TimeIndicesObserved = \{t-\tau ~\vert~ 0 \leq \tau \leq \ptimewindow\}$
        \Function{simple-ts-ADMG}{$\DAG$, $\ptimewindow$}
            \State $\Graph$ $\leftarrow$ copy of the segment of $\DAG$ on $[t-\ptimewindow, \ldots, t]$ \Comment{Copy directed edges from $\DAG$ according$\ldots$}
            \Statex \Comment{$\ldots$ to Proposition 3.7.}
            \ForAll{$(i, j, \Delta\tau) \in \VarIndices \times \VarIndices \times [0, \ldots, \ptimewindow]$} \Comment{Run through all pairs of$\ldots$}
                \If{$i < j$ or $\Delta\tau > 0$} \Comment{$\ldots$}
                    \ForAll{$\tau_j = 0$ to $\tau_j = \ptimewindow-\Delta\tau$} \Comment{$\ldots$}
                        \State $v_1, v_2$ $\leftarrow$ $(i, t-\tau_j-\Delta\tau), (j, t-\tau_j)$ \Comment{$\ldots$ vertices.}
                        \ForAll{$(k,t-\tau_k), (l,t-\tau_l) \in \parents{\DAG}{v_1} \times \parents{\DAG}{v_2}$} \Comment{Run through all pairs of$\ldots$}
                            \If{$\tau_k > \ptimewindow$ and $\tau_l > \ptimewindow$} \Comment{$\ldots$ parents of $(v_1, v_2)$$\ldots$}
                                \Statex \Comment{$\ldots$ before the observed time window.}
                                \If{$\tau_k > \tau_l$ or ($\tau_k = \tau_l$ and $k<l$)} \Comment{Shift the temporally later parent$\ldots$}
                                    \State $w_1, w_2$ $\leftarrow$ $(k, t-\tau_k+\tau_l), (l, t)$ \Comment{$\ldots$ vertex to time $t$, using$\ldots$}
                                \Else \Comment{$\ldots$ lexicographical order$\ldots$}
                                    \State $w_1, w_2$ $\leftarrow$ $(k, t), (l, t-\tau_l+\tau_k)$ \Comment{$\ldots$ if the parent vertices$\ldots$}
                                \EndIf \Comment{$\ldots$ are concurrent.}
                                \If{\textsc{have-common-ancestor}($\DAG$, $w_1$, $w_2$)} \Comment{Use Proposition 3.8 and the$\ldots$}
                                    \Statex \Comment{$\ldots$ repeating edges property of$\ldots$ }
                                    \Statex \Comment{$\ldots$ ts-DAGs to decide whether $v_1 \headhead v_2$.}
                                    \ForAll{$\tilde{\tau}_j \in [\tau_j, \ldots, \ptimewindow-\Delta \tau]$} \Comment{Add bidirected edges between $(v_1, v_2)$$\ldots$}
                                        \State $\tilde{v}_1, \tilde{v}_2$ $\leftarrow$ $(i, t-\tilde{\tau}_j-\Delta\tau), (j, t-\tilde{\tau}_j)$ \Comment{$\ldots$ and between all backwards$\ldots$}
                                        \State $\Graph$ $\leftarrow$ Add $\tilde{v}_1 \headhead \tilde{v}_2$ to $\Graph$ \Comment{$\ldots$ shifted pairs of vertices according$\ldots$}
                                        \Statex \Comment{$\ldots$ to the last paragraph of Section 3.3.3.}
                                    \EndFor
                                \EndIf
                            \EndIf
                        \EndFor
                    \EndFor
                \EndIf
            \EndFor
            \State \Return $\tsADMG{\ObservedVertices^\prime}{\DAG} = \Graph$
        \EndFunction
    \end{algorithmic}
\end{algorithm}

Algorithm~\ref{algo:simple-ts-ADMG} provides pseudocode for the function \textsc{simple-ts-ADMG} that returns the simple finite marginal ts-ADMG $\tsADMG{\ObservedVertices^\prime}{\DAG}$ of an infinite ts-DAG $\DAG$, where $\ObservedVertices^\prime = \VarIndices \times \TimeIndicesObserved$. This function implements the results in Section 3.3.3 and calls a function \textsc{have-common-ancestor} that decides about common-ancestorship in ts-DAGs.

\begin{algorithm}[tbp]
    \caption{Pseudocode for the function \textsc{have-common-ancestor}.}
    \label{algo:have-common-ancestor}
    \begin{algorithmic}[1]
        \Require Time series DAG $\DAG$ with variable index set $\VarIndices$, vertices $(i,t-\tau)$ and $(j, t)$ with $\tau \geq 0$ in $\DAG$ (not necessarily distinct) 
        \Ensure As Boolean value: Whether or not $(i,t-\tau)$ and $(j, t)$ have a common ancestor
        \Function{have-common-ancestor}{$\DAG$, $(i, t-\tau)$, $(j, t)$}
            \State $\Graph_{w}$ $\leftarrow$ $\mathcal{S}_{\w}(\DAG)$ \Comment{Multi-weighted summary graph of $\DAG$, see Definition 4.5.}
            \State $\cC$ $\leftarrow$ \textsc{get-cycle-classes}($\Graph_{w}$) \Comment{See Section~\ref{subsubsec.cycles}.}
            \State $\w $-$\mathrm{dict}$ $\leftarrow$ \textsc{get-cycle-weights}($\cC, \Graph_{w}$) \Comment{See Section 4.3.1.}
            \ForAll{$k \in \VarIndices$} \Comment{Run through all vertices in $\Graph_{w}$.}
                \State $\mathcal{W}_0(k,i) \ \leftarrow$ \textsc{get-cycle-free-paths}($k,i, \Graph_{w}$)
                \State $\mathcal{W}_0(k,j) \ \leftarrow$ \textsc{get-cycle-free-paths}($k,j, \Graph_{w}$)  \Comment{See Lemma 4.8 and the discussion after.}
                \ForAll{$(\pi,\pi') \in \mathcal{W}_0(k,i) \times \mathcal{W}_0(k,j)$}
                    \State $\weightset{\pi}$ $\leftarrow$ \textsc{get-weightset}($\pi, \Graph_{w}$)
                    \State $\weightset{\pi'}$ $\leftarrow$ \textsc{get-weightset}($\pi', \Graph_{w}$) \Comment{See Section 4.3.1.}
                    \State $\mathcal{M}_{\pi} \ \leftarrow$ \textsc{get-monoid}($\pi, \cC$)
                    \State $\mathcal{M}_{\pi'} \ \leftarrow$ \textsc{get-monoid}($\pi', \cC$)  \Comment{See Section~\ref{subsec.deriving-Lemma}.}
                    \ForAll{$(\mathcal{S},\mathcal{S}') \in \mathcal{M}_{\pi} \times \mathcal{M}_{\pi'}$}
                        \State $D_{\tau}(\pi,\mathcal S)$ $\leftarrow$ \textsc{tuple-set}($\tau$, $\pi$, $\mathcal{S}, \w$-$\mathrm{dict}$, $\weightset{\pi}$)
                        \State $D_{0}(\pi',\mathcal S')$ $\leftarrow$ \textsc{tuple-set}($0$, $\pi'$, $\mathcal{S}', \w$-$\mathrm{dict}$, $\weightset{\pi}$) \Comment{See equation~\eqref{def.tuple-sets}.}
                        \ForAll{$(a_0;a_1, \ldots, a_\mu), (a^\prime_0;a_1^\prime,\ldots,a^\prime_\nu) \in D_{\tau}(\pi,\mathcal S) \times D_{0}(\pi',\mathcal S')$}
                            \If{\textsc{has-nni-solution}($(a_0;a_1, \ldots, a_\mu)$, $(a^\prime_0;a_1^\prime,\ldots,a^\prime_\nu)$)}
                                \State \Return \textsc{True} \Comment{See Proposition 4.6 and Theorem 1.}
                            \EndIf
                        \EndFor
                    \EndFor
                \EndFor
            \EndFor
            \State \Return \textsc{False} \Comment{See Proposition 4.6 and Theorem 1.}
        \EndFunction
    \end{algorithmic}
\end{algorithm}

\begin{algorithm}[tbp]
    \caption{Pseudocode for the function \textsc{get-monoid}.}
    \label{algo:get-monoid}
    \begin{algorithmic}[1]
        \Require Cycle-free path $\pi$, set of irreducible cycle classes $\cC$.
        \Ensure Set monoid $\mathcal{M}_{\pi}$, a subset of the power set of $\cC$.
        \Function{get-monoid}{$\pi,\cC$}
            \State $\mathrm{touch}(\pi)$ $\leftarrow$ \textsc{compute-touch-set}$(\pi,\cC)$ \Comment{See Section~\ref{subsubsec.cycles}.}
            \State $\Graph_{\cC}$ $\leftarrow$ \textsc{compute-graph-of-cycles}$(\pi,\cC)$ \Comment{See Definition~\ref{def.graph-of-cycles}.}
            \State $\mathcal{A}_{\cC,\pi}$ $\leftarrow$ \textsc{compute-access-points}$(\Graph_{\cC},\mathrm{touch}(\pi))$ \Comment{See Definition~\ref{def.access-point}.}
            \State $\mathcal{Q}_{\pi}$ $\leftarrow$ \textsc{compute-generating-set}$(\Graph_{\cC},\mathrm{touch}(\pi),\mathcal{A}_{\cC,\pi})$ \Comment{See Section~\ref{subsec.deriving-Lemma}.}
            \State $\mathcal{Q}_{\pi}'$ $\leftarrow$ \textsc{set-project}$(\mathcal{Q}_{\pi})$ \Comment{See Section~\ref{subsec.deriving-Lemma}.}
            \State $\mathcal{M}_{\pi}$ $\leftarrow$ \textsc{compute-monoid-from-generating-set}$(\mathcal{Q}_{\pi}')$ \Comment{See Section~\ref{subsec.deriving-Lemma}.}
            \State \Return $\mathcal{M}_{\pi}$
        \EndFunction
    \end{algorithmic}
\end{algorithm}

Algorithm~\ref{algo:have-common-ancestor} provides pseudocode for the function \textsc{have-common-ancestor} that returns whether or not the vertices $(i,t-\tau)$ and $(j, t)$ have a common ancestor in an infinite ts-DAG $\DAG$. Here, common-ancestorship is understood in the sense of Remark 3.9. This function jointly implements Proposition 4.6 and Theorem 1 and relies on the subroutines \textsc{get-cycle-classes}, \textsc{get-cycle-free-paths}, \textsc{get-cycle-weights}, \textsc{get-weightset}, \textsc{get-monoid}, \textsc{tuple-set} and \textsc{has-nni-solution}. The first function \textsc{get-cycle-classes} computes a list of all equivalence classes of irreducible cycles in the multi-weighted summary graph, and the second function \textsc{get-cycle-free-paths} takes two indices $k$ and $i$ and the multi-weighted summary graph as an input and returns all cycle-free directed paths from $k$ to $i$.\footnote{Note that both of these function disregard the multi-weights because the summary graph itself is sufficient for their purposes.} Since implementations of irreducible cycle enumeration and search routines for cycle-free paths are freely available, we will not provide further details on these two subroutines.
Next, the function \textsc{get-cycle-weights} computes a dictionary that assigns the weight-set $\weightset{\c}$ to every $\c \in \cC$ from the edge weight-sets of the multi-weighted graph $\Graph_{w} = \mathcal{S}_{\w}(\DAG)$. Since this procedure amounts to a simple execution of the definition of $\weightset{\c}$ as the sums over the weight-sets of the edges on $\c$, we omit the pseudocode for this subroutine. The function \textsc{get-weightset} is very similar to \textsc{get-cycle-weights} as it computes the weight-set of a cycle-free path $\pi$ as the sum of the weight-sets of the path's edges. Again, we we do not provide pseudocode for this simple operation.

The function \textsc{get-monoid} is the most complicated subroutine, and we provide pseudocode for it in Algorithm~\ref{algo:get-monoid} below.

The function \textsc{tuple-set} takes the delay $\tau$, a cycle-free path $\pi$, a set of irreducible cycle classes $\mathcal{S}$ as well as the weight-set dictionary for cycle classes $\w$-$\mathrm{dict}$ and the weight-set $\weightset{\pi}$ as input and computes the direct product set $D_{\tau}(\pi,\mathcal S)$ from its defining equation~\eqref{def.tuple-sets}. Again, this procedure is a simple computation of a product set from simpler sets and is straightforward to execute, which is why we do not provide pseudocode. 

Finally, the function \textsc{has-nni-solution} returns (as Boolean value) whether a linear Diophantine equation as in eq.~(3) as non-negative integer solution. We omit pseudocode for this function, since, at least in its most basic implementation, such pseudocode amounts to a one-to-one translation of Theorem 2. See also the paragraph below Theorem 2 for a discussion on how to improve on that basic implementation.

Thus, we are left with describing pseudocode for the subroutine \textsc{get-monoid}, which we do in Algorithm~\ref{algo:get-monoid}. This algorithm reduces to the subroutines \textsc{compute-touch-set}, \textsc{compute-graph-of-cycles}, \textsc{compute-access-points}, \textsc{compute-generating-set} as well as \textsc{compute-monoid-from-generating-set}.

The function \textsc{compute-touch-set} for every cycle class $\c$ checks where $\c$ shares a node with $\pi$ and, if this is the case, then adds $\c$ to the touch set. Similarly, \textsc{compute-graph-of-cycles} checks for all pairs $\c \neq \c' \in \cC$ whether the two cycles share any nodes and, if this is the case, then adds an edge $(\c,\c')$ to the graph of cycles. The function \textsc{set-project} takes in a set of paths $\{ \xi \}$ and computes the set of node sets $\{ \mathrm{nodes}(\xi) \}$, which is a straighforward operation for which we do not provide pseudocode. 

\begin{algorithm}[tbp]
    \caption{Pseudocode for the function \textsc{compute-access-points}.}
    \label{algo:compute-access-points}
    \begin{algorithmic}[1]
        \Require Undirected Graph $\Graph$ with vertex set $\mathcal{V}$, subset $\mathcal{S} \subseteq \mathcal{V}$.
        \Ensure Set of $\mathcal{S}$-access points on $\Graph$. \Comment{See Definition \ref{def.access-point}.}
        \Function{compute-access-points}{$\Graph,\mathcal{S}$}
            \State $\mathcal{A}(\Graph,\mathcal{S})$ $\leftarrow$ $\emptyset$
            \ForAll{$v$ in $\mathcal{V}$}
                \If{\textsc{neighbors}$(v,\Graph) \neq \emptyset$}
                    \ForAll{$v' $ in \textsc{neighbors}($v,\Graph$)$\setminus \mathcal{S}$} 
                        \ForAll{$v''$ in $\mathcal{S}$}
                        \If{\textsc{is-path}($v,v''$,\textsc{subset-delete}($\{v'\},\Graph$))}
                            \State $\mathcal{A}(\Graph,\mathcal{S})$ $\leftarrow$ Add $v$ to $\mathcal{A}(\Graph,\mathcal{S})$
                            \State \textbf{break loop starting in line 5}
                            \EndIf
                        \EndFor
                    \EndFor
                \EndIf
            \EndFor
            \State \Return $\mathcal{A}(\Graph,\mathcal{S})$.   
        \EndFunction
    \end{algorithmic}
\end{algorithm}

Algorithm~\ref{algo:compute-access-points} implements \textsc{compute-access-points}. This algorithms assumes (1) the availability of a function \textsc{neighbors}, which computes the set of direct neighbors of a node $v$ on an undirected graph $\Graph$, (2) a function \textsc{subset-delete} that inputs a subset of nodes $\mathcal{V}'$ on an undirected graph $\Graph$ with vertex set $\mathcal{V} \supseteq \mathcal{V}'$ and returns the subgraph $\Graph\backslash\mathcal{V}'$ from which all nodes in $\mathcal{V}'$ and their adjacent edges have been deleted, and (3) the availability of a function \textsc{is-path}, which decides whether there is a (possibly trivial) path between two given nodes $v$ and $v'$ in an undirected graph $\Graph$. All of these functions are available in standard graph theory packages and easily implementable, which is why we do not provide pseudocode.

\begin{algorithm}[tbp]
    \caption{Pseudocode for the function \textsc{compute-generating-set}.}
    \label{algo:compute-generating-set}
    \begin{algorithmic}[1]
        \Require An undirected graph $\Graph$ with vertex set $\mathcal{V}$, a subset of nodes $\mathcal S \subseteq \mathcal{V}$, the set $\mathcal{A}(\Graph,\mathcal{S})$ of $\mathcal{S}$-access points on $\Graph$. 
        \Ensure The set of paths $\mathcal{Q}$ starting in and not returning to $\mathcal{S}$ with nodes in  $\mathcal{A}(\Graph,\mathcal{S})$ .
        \Function{compute-generating-set}{$\Graph, \mathcal{S},\mathcal{A}(\Graph,\mathcal{S})$}  \Comment{See Section~\ref{subsec.deriving-Lemma}.}
        \State $\mathcal{H}$ $\leftarrow$ \textsc{subset-delete}($\mathcal{V}\backslash\mathcal{A}(\Graph,\mathcal{S}),\mathcal{G}$)
        \State $\mathcal{Q}$ $\leftarrow$ $\{\emptyset \}$
        \ForAll{$v$ in $\mathcal{S} \cap \mathcal{A}(\Graph,\mathcal{S})$}
            \State add trivial path $(v)$ to $\mathcal{Q}$
            \State $\mathcal{H}_v$ $\leftarrow$ \textsc{subset-delete}($\left(\mathcal{S} \cap \mathcal{A}(\Graph,\mathcal{S})\right)\backslash\{ v \},\mathcal{H}$)
            \ForAll{$v'$ in $\mathcal{A}(\Graph,\mathcal{S})\backslash\mathcal{S}$}
                \State $\mathcal{Q}$ $\leftarrow$ $\mathcal{Q} \cup $ \textsc{get-paths}($v,v',\mathcal{H}_v$ )
            \EndFor
        \EndFor
        \State \Return $\mathcal{Q}$
        \EndFunction
    \end{algorithmic}
\end{algorithm}

Next, in Algorithm~\ref{algo:compute-generating-set}, we provide pseudocode for the function \textsc{compute-generating-set}. This algorithm employs a subroutine \textsc{get-paths}, which takes in two nodes $v$ and $v'$ and a finite undirected graph $\Graph$ (to which the nodes belong) and returns the set of all paths from $v$ to $v'$ on $\Graph$. Again, such a routine is available in standard packages, which is why we do not provide pseudocode.

\begin{algorithm}[tbp]
    \caption{Pseudocode for the function \textsc{compute-monoid-from-generating-set}.}
    \label{algo:compute-monoid-from-generating-set}
    \begin{algorithmic}[1]
        \Require $\mathcal{Q}$, a finite set of sets.
        \Ensure The set monoid generated by $\mathcal{Q}$.
        \Function{compute-monoid-from-generating-set}{$\mathcal{Q}$}  \Comment{See Section~\ref{subsec.deriving-Lemma}.}
        \If{$\mathcal{Q} = \{\emptyset \}$}
            \State \Return $\mathcal{M} = \{\emptyset \}$  
        \Else
            \State $\mathcal{M}^1$ $\leftarrow$ $\mathcal{Q}$
            \For{$i = 1 \dots, |\mathcal{Q}|-1$}
                \State $\mathcal{M}^{i+1}$ $\leftarrow$ $\mathcal{M}^{i}$
                \ForAll{$(\mathcal{S},\mathcal{T})$ in $\mathcal{Q} \times \mathcal{M}^{i}$} 
                        \If{$\mathcal{S}\cup \mathcal{T}$ is not in $\mathcal{M}^{i}$}
                            \State add $\mathcal{S}\cup \mathcal{T}$ to $\mathcal{M}^{i+1}$
                         \EndIf
                \EndFor
                \If{$\mathcal{M}^{i} = \mathcal{M}^{i+1}$}
                    \State \Return  $\{\emptyset \} \cup \mathcal{M}^{i+1}$
                    \State \textbf{break}
                \EndIf
            \EndFor
            \State \Return $\{\emptyset \} \cup \mathcal{M}^{|\mathcal{Q}|}$
        \EndIf 
        \EndFunction
    \end{algorithmic}
\end{algorithm}

We finish by describing pseudocode for the last remaining subroutine \newline \textsc{compute-monoid-from-generating-set} of Algorithm~\ref{algo:get-monoid}. This subroutine computes a set monoid from a given generating set, for which we provide pseudocode in Algorithm~\ref{algo:compute-monoid-from-generating-set}.

\end{appendix}

\newpage

\bibliographystyle{apalike}
\bibliography{library}

\begin{thebibliography}{}

\bibitem[Andreescu et~al., 2010]{andreescu_introduction_2010}
Andreescu, T., Andrica, D., and Cucurezeanu, I. (2010).
\newblock {\em An {Introduction} to {Diophantine} {Equations}: {A}
  {Problem}-{Based} {Approach}}.
\newblock Birkhäuser, Boston.

\bibitem[Angrist and Pischke, 2009]{angrist2009mostly}
Angrist, J.~D. and Pischke, J.-S. (2009).
\newblock {\em Mostly harmless econometrics: An empiricist's companion}.
\newblock Princeton university press.

\bibitem[Assaad et~al., 2022]{assaad2022survey}
Assaad, C.~K., Devijver, E., and Gaussier, E. (2022).
\newblock Survey and evaluation of causal discovery methods for time series.
\newblock {\em Journal of Artificial Intelligence Research}, 73:767--819.

\bibitem[Bareinboim and Pearl, 2016]{bareinboim2016causal}
Bareinboim, E. and Pearl, J. (2016).
\newblock Causal inference and the data-fusion problem.
\newblock {\em Proceedings of the National Academy of Sciences},
  113(27):7345--7352.

\bibitem[Bollen, 1989]{bollen1989structural}
Bollen, K.~A. (1989).
\newblock {\em {Structural Equations with Latent Variables}}.
\newblock John Wiley \& Sons, New York, NY, USA.

\bibitem[Bongers et~al., 2018]{bongers2018causal}
Bongers, S., Blom, T., and Mooij, J.~M. (2018).
\newblock Causal modeling of dynamical systems.
\newblock {\em arXiv preprint arXiv:1803.08784}.

\bibitem[Bowden and Turkington, 1990]{bowden1990instrumental}
Bowden, R.~J. and Turkington, D.~A. (1990).
\newblock {\em Instrumental variables}.
\newblock Cambridge university press.

\bibitem[Brauer, 1942]{brauer_problem_1942}
Brauer, A. (1942).
\newblock On a {Problem} of {Partitions}.
\newblock {\em American Journal of Mathematics}, 64(1):299--312.
\newblock Publisher: Johns Hopkins University Press.

\bibitem[Bringmann, 2017]{bringmann_near-linear_2017}
Bringmann, K. (2017).
\newblock A {Near}-{Linear} {Pseudopolynomial} {Time} {Algorithm} for {Subset}
  {Sum}.
\newblock In {\em Proceedings of the 2017 {Annual} {ACM}-{SIAM} {Symposium} on
  {Discrete} {Algorithms} ({SODA})}, Proceedings, pages 1073--1084. Society for
  Industrial and Applied Mathematics.

\bibitem[Camps-Valls et~al., 2023]{camps2023discovering}
Camps-Valls, G., Gerhardus, A., Ninad, U., Varando, G., Martius, G.,
  Balaguer-Ballester, E., Vinuesa, R., Diaz, E., Zanna, L., and Runge, J.
  (2023).
\newblock Discovering causal relations and equations from data.
\newblock {\em arXiv preprint arXiv:2305.13341}.

\bibitem[Chu and Glymour, 2008]{chu2008search}
Chu, T. and Glymour, C. (2008).
\newblock Search for additive nonlinear time series causal models.
\newblock {\em Journal of Machine Learning Research}, 9(32):967--991.

\bibitem[Dahlhaus and Eichler, 2003]{dahlhaus2003causality}
Dahlhaus, R. and Eichler, M. (2003).
\newblock Causality and graphical models in time series analysis.
\newblock In Green, P.~J., Hjort, N.~L., and Richardson, S., editors, {\em
  Highly Structured Stochastic Systems}. Oxford University Press, Oxford.

\bibitem[Danks and Davis, 2023]{danks2023causal}
Danks, D. and Davis, I. (2023).
\newblock Causal inference in cognitive neuroscience.
\newblock {\em Wiley Interdisciplinary Reviews: Cognitive Science}, page e1650.

\bibitem[Eichler, 2010]{eichler2010graphical}
Eichler, M. (2010).
\newblock Graphical gaussian modelling of multivariate time series with latent
  variables.
\newblock In {\em Proceedings of the Thirteenth International Conference on
  Artificial Intelligence and Statistics}, pages 193--200. JMLR Workshop and
  Conference Proceedings.

\bibitem[Eichler and Didelez, 2007]{eichler2007causal}
Eichler, M. and Didelez, V. (2007).
\newblock Causal reasoning in graphical time series models.
\newblock In Parr, R. and van~der Gaag, L., editors, {\em Proceedings of the
  Twenty-Third Conference on Uncertainty in Artificial Intelligence}, UAI'07,
  page 109–116, Arlington, Virginia, USA. AUAI Press.

\bibitem[Eichler and Didelez, 2010]{eichler2010granger}
Eichler, M. and Didelez, V. (2010).
\newblock On granger causality and the effect of interventions in time series.
\newblock {\em Lifetime data analysis}, 16(1):3--32.

\bibitem[Entner and Hoyer, 2010]{entner2010causal}
Entner, D. and Hoyer, P.~O. (2010).
\newblock On causal discovery from time series data using fci.
\newblock In Myllym{\"a}ki, P., Roos, T., and Jaakkola, T., editors, {\em
  Proceedings of the 5th European Workshop on Probabilistic Graphical Models},
  pages 121--128, Helsinki, FI. Helsinki Institute for Information Technology
  HIIT.

\bibitem[Geiger et~al., 1990]{geiger1990identifying}
Geiger, D., Verma, T., and Pearl, J. (1990).
\newblock Identifying independence in bayesian networks.
\newblock {\em Networks}, 20(5):507--534.

\bibitem[Gerhardus, 2023]{gerhardus2021characterization}
Gerhardus, A. (2023).
\newblock Characterization of causal ancestral graphs for time series with
  latent confounders.
\newblock {\em arXiv:2112.08417v2 (accepted at The Annals of Statistics)}.

\bibitem[Gerhardus and Runge, 2020]{gerhardus2020high}
Gerhardus, A. and Runge, J. (2020).
\newblock High-recall causal discovery for autocorrelated time series with
  latent confounders.
\newblock In Larochelle, H., Ranzato, M., Hadsell, R., Balcan, M., and Lin, H.,
  editors, {\em Advances in Neural Information Processing Systems 34 (NeurIPS
  2021)}, volume~33, pages 12615--12625.

\bibitem[Granger, 1969]{granger1969investigating}
Granger, C. W.~J. (1969).
\newblock Investigating causal relations by econometric models and
  cross-spectral methods.
\newblock {\em Econometrica}, 37:424--438.

\bibitem[Hernan and Robins, 2020]{hernan2020causal}
Hernan, M. and Robins, J. (2020).
\newblock {\em Causal Inference: What if.}
\newblock Chapman \& Hill/CRC.

\bibitem[Huang and Valtorta, 2006]{huang2006pearls}
Huang, Y. and Valtorta, M. (2006).
\newblock Pearl's calculus of intervention is complete.
\newblock In Dechter, R. and Richardson, T., editors, {\em Proceedings of the
  Twenty-Second Conference on Uncertainty in Artificial Intelligence}, UAI'06,
  page 217–224, Arlington, Virginia, USA. AUAI Press.

\bibitem[Hyv{{\"a}}rinen et~al., 2010]{hyvarinen2010estimation}
Hyv{{\"a}}rinen, A., Zhang, K., Shimizu, S., and Hoyer, P.~O. (2010).
\newblock Estimation of a structural vector autoregression model using
  non-gaussianity.
\newblock {\em Journal of Machine Learning Research}, 11(56):1709--1731.

\bibitem[Imbens and Rubin, 2015]{imbens2015causal}
Imbens, G.~W. and Rubin, D.~B. (2015).
\newblock {\em Causal inference in statistics, social, and biomedical
  sciences}.
\newblock Cambridge University Press.

\bibitem[Maathuis and Colombo, 2015]{maathuis2015generalized}
Maathuis, M.~H. and Colombo, D. (2015).
\newblock {A generalized back-door criterion}.
\newblock {\em The Annals of Statistics}, 43(3):1060 -- 1088.

\bibitem[Malinsky and Spirtes, 2018]{malinsky2018causal}
Malinsky, D. and Spirtes, P. (2018).
\newblock Causal structure learning from multivariate time series in settings
  with unmeasured confounding.
\newblock In Le, T.~D., Zhang, K., K{\i}c{\i}man, E., Hyv\"{a}rinen, A., and
  Liu, L., editors, {\em Proceedings of 2018 ACM SIGKDD Workshop on Causal
  Disocvery}, volume~92 of {\em Proceedings of Machine Learning Research},
  pages 23--47, London, UK. PMLR.

\bibitem[Mastakouri et~al., 2021]{mastakouri2020necessary}
Mastakouri, A.~A., Sch{\"o}lkopf, B., and Janzing, D. (2021).
\newblock Necessary and sufficient conditions for causal feature selection in
  time series with latent common causes.
\newblock In Meila, M. and Zhang, T., editors, {\em Proceedings of the 38th
  International Conference on Machine Learning}, volume 139 of {\em Proceedings
  of Machine Learning Research}, pages 7502--7511. PMLR.

\bibitem[Mooij and Claassen, 2020]{mooij2020constraint}
Mooij, J.~M. and Claassen, T. (2020).
\newblock Constraint-based causal discovery using partial ancestral graphs in
  the presence of cycles.
\newblock In Peters, J. and Sontag, D., editors, {\em Proceedings of the 36th
  Conference on Uncertainty in Artificial Intelligence (UAI)}, volume 124 of
  {\em Proceedings of Machine Learning Research}, pages 1159--1168. PMLR.

\bibitem[Nowack et~al., 2020]{nowack2020causal}
Nowack, P., Runge, J., Eyring, V., and Haigh, J.~D. (2020).
\newblock {Causal networks for climate model evaluation and constrained
  projections}.
\newblock {\em Nature communications}, 11(1):1--11.

\bibitem[Pearl, 1988]{pearl1988}
Pearl, J. (1988).
\newblock {\em {Probabilistic Reasoning in Intelligent Systems: Networks of
  Plausible Inference}}.
\newblock Morgan Kaufmann Publishers Inc., San Francisco, CA, USA.

\bibitem[Pearl, 1993]{pearl1993bayesian}
Pearl, J. (1993).
\newblock [bayesian analysis in expert systems]: Comment: Graphical models,
  causality and intervention.
\newblock {\em Statistical Science}, 8(3):266--269.

\bibitem[Pearl, 1995]{pearl1995causal}
Pearl, J. (1995).
\newblock Causal diagrams for empirical research.
\newblock {\em Biometrika}, 82(4):669--688.

\bibitem[Pearl, 2009]{pearl2009causality}
Pearl, J. (2009).
\newblock {\em {Causality: Models, Reasoning, and Inference}}.
\newblock Cambridge University Press, Cambridge, UK, 2nd edition.

\bibitem[Pearl and Verma, 1995]{pearl1995theory}
Pearl, J. and Verma, T.~S. (1995).
\newblock A theory of inferred causation.
\newblock In {\em Studies in Logic and the Foundations of Mathematics}, volume
  134, pages 789--811. Elsevier.

\bibitem[Perkovi\'c et~al., 2018]{perkovic2018complete}
Perkovi\'c, E., Textor, J., Kalisch, M., and Maathuis, M.~H. (2018).
\newblock Complete graphical characterization and construction of adjustment
  sets in markov equivalence classes of ancestral graphs.
\newblock {\em Journal of Machine Learning Research}, 18(220):1--62.

\bibitem[Peters et~al., 2017]{peters2017elements}
Peters, J., Janzing, D., and Sch{\"{o}}lkopf, B. (2017).
\newblock {\em {Elements of Causal Inference: Foundations and Learning
  Algorithms}}.
\newblock MIT Press, Cambridge, MA, USA.

\bibitem[Pisinger, 1999]{pisinger_linear_1999}
Pisinger, D. (1999).
\newblock Linear {Time} {Algorithms} for {Knapsack} {Problems} with {Bounded}
  {Weights}.
\newblock {\em Journal of Algorithms}, 33(1):1--14.

\bibitem[Pya~Arnqvist et~al., 2019]{arnqvist2019nilde}
Pya~Arnqvist, N., Voinov, V., and Voinov, Y. (2019).
\newblock nilde: Nonnegative integer solutions of linear diophantine equations
  with applications. r package version 1.1-3.

\bibitem[Ramírez~Alfonsín, 2005]{ramirez_alfonsin_diophantine_2005}
Ramírez~Alfonsín, J.~L. (2005).
\newblock {\em The {Diophantine} {Frobenius} {Problem}}.
\newblock Oxford University Press.

\bibitem[Reiter et~al., 2023]{reiter2023formalising}
Reiter, N.-D., Gerhardus, A., Wahl, J., and Runge, J. (2023).
\newblock Formalising causal inference in time and frequency on process graphs
  with latent components.
\newblock {\em arXiv preprint arXiv:2305.11561}.

\bibitem[Richardson, 2003]{richardson2003markov}
Richardson, T. (2003).
\newblock Markov properties for acyclic directed mixed graphs.
\newblock {\em Scandinavian Journal of Statistics}, 30(1):145--157.

\bibitem[Richardson and Spirtes, 2002]{richardson2002ancestral}
Richardson, T. and Spirtes, P. (2002).
\newblock Ancestral graph markov models.
\newblock {\em The Annals of Statistics}, 30(4):962--1030.

\bibitem[Richardson et~al., 2023]{richardson2023nested}
Richardson, T.~S., Evans, R.~J., Robins, J.~M., and Shpitser, I. (2023).
\newblock {Nested Markov properties for acyclic directed mixed graphs}.
\newblock {\em The Annals of Statistics}, 51(1):334 -- 361.

\bibitem[Runge, 2020]{runge2020discovering}
Runge, J. (2020).
\newblock Discovering contemporaneous and lagged causal relations in
  autocorrelated nonlinear time series datasets.
\newblock In Peters, J. and Sontag, D., editors, {\em Proceedings of the 36th
  Conference on Uncertainty in Artificial Intelligence (UAI)}, volume 124 of
  {\em Proceedings of Machine Learning Research}, pages 1388--1397. PMLR.

\bibitem[Runge, 2021]{runge2021necessary}
Runge, J. (2021).
\newblock Necessary and sufficient graphical conditions for optimal adjustment
  sets in causal graphical models with hidden variables.
\newblock In Ranzato, M., Beygelzimer, A., Dauphin, Y., Liang, P., and Vaughan,
  J.~W., editors, {\em Advances in Neural Information Processing Systems 34
  (NeurIPS 2021)}.

\bibitem[Runge, 2023]{runge2023modern}
Runge, J. (2023).
\newblock Modern causal inference approaches to investigate
  biodiversity-ecosystem functioning relationships.
\newblock {\em nature communications}, 14(1):1917.

\bibitem[Runge et~al., 2019a]{runge2019inferring}
Runge, J., Bathiany, S., Bollt, E., Camps-Valls, G., Coumou, D., Deyle, E.,
  Glymour, C., Kretschmer, M., Mahecha, M.~D., Mu{\~n}oz-Mar{\'\i}, J., et~al.
  (2019a).
\newblock Inferring causation from time series in earth system sciences.
\newblock {\em Nature communications}, 10(1):1--13.

\bibitem[Runge et~al., 2015]{runge2015optimal}
Runge, J., Donner, R.~V., and Kurths, J. (2015).
\newblock {Optimal model-free prediction from multivariate time series}.
\newblock {\em Physical Review E}, 91(5):052909.

\bibitem[Runge et~al., 2023]{runge2023causal}
Runge, J., Gerhardus, A., Varando, G., Eyring, V., and Camps-Valls, G. (2023).
\newblock Causal inference for time series.
\newblock {\em Nature Reviews Earth \& Environment}, pages 1--19.

\bibitem[Runge et~al., 2012]{runge2012escaping}
Runge, J., Heitzig, J., Petoukhov, V., and Kurths, J. (2012).
\newblock Escaping the curse of dimensionality in estimating multivariate
  transfer entropy.
\newblock {\em Physical Review Letters}, 108:258701.

\bibitem[Runge et~al., 2019b]{runge2019detecting}
Runge, J., Nowack, P., Kretschmer, M., Flaxman, S., and Sejdinovic, D. (2019b).
\newblock Detecting and quantifying causal associations in large nonlinear time
  series datasets.
\newblock {\em Science advances}, 5(11):eaau4996.

\bibitem[Sargan, 1958]{sargan1958estimation}
Sargan, J.~D. (1958).
\newblock The estimation of economic relationships using instrumental
  variables.
\newblock {\em Econometrica: Journal of the econometric society}, pages
  393--415.

\bibitem[Sch{\"o}lkopf et~al., 2021]{schoelkopf2021toward}
Sch{\"o}lkopf, B., Locatello, F., Bauer, S., Ke, N.~R., Kalchbrenner, N.,
  Goyal, A., and Bengio, Y. (2021).
\newblock Toward causal representation learning.
\newblock {\em Proceedings of the IEEE}, 109(5):612--634.

\bibitem[Shpitser and Pearl, 2006a]{shpitser2006identification}
Shpitser, I. and Pearl, J. (2006a).
\newblock Identification of conditional interventional distributions.
\newblock In Dechter, R. and Richardson, T., editors, {\em Proceedings of the
  Twenty-Second Conference on Uncertainty in Artificial Intelligence}, UAI'06,
  page 437–444, Arlington, Virginia, USA. AUAI Press.

\bibitem[Shpitser and Pearl, 2006b]{shpitser2006identification_2}
Shpitser, I. and Pearl, J. (2006b).
\newblock Identification of conditional interventional distributions.
\newblock In Dechter, R. and Richardson, T., editors, {\em Proceedings of the
  Twenty-Second Conference on Uncertainty in Artificial Intelligence}, UAI'06,
  page 437–444, Arlington, Virginia, USA. AUAI Press.

\bibitem[Shpitser and Pearl, 2008]{shpitser2008complete}
Shpitser, I. and Pearl, J. (2008).
\newblock Complete identification methods for the causal hierarchy.
\newblock {\em Journal of Machine Learning Research}, 9:1941--1979.

\bibitem[Shpitser et~al., 2010]{shpitser2010validity}
Shpitser, I., VanderWeel, T., and Robins, J.~M. (2010).
\newblock On the validity of covariate adjustment for estimating causal
  effects.
\newblock In Grunwald, P. and Spirtes, P., editors, {\em UAI'10: Proceedings of
  the Twenty-Sixth Conference on Uncertainty in Artificial Intelligence}. AUAI
  Press.

\bibitem[Spirtes et~al., 2000a]{Spirtes2000}
Spirtes, P., Glymour, C., and Scheines, R. (2000a).
\newblock {\em {Causation, Prediction, and Search}}.
\newblock MIT Press, Boston.

\bibitem[Spirtes et~al., 2000b]{spirtes2000causation}
Spirtes, P., Glymour, C.~N., Scheines, R., and Heckerman, D. (2000b).
\newblock {\em Causation, prediction, and search}.
\newblock MIT press, Cambridge, MA, USA, 2nd edition.

\bibitem[Thams et~al., 2022]{thams2022identifying}
Thams, N., S{\o}ndergaard, R., Weichwald, S., and Peters, J. (2022).
\newblock Identifying causal effects using instrumental time series: Nuisance
  iv and correcting for the past.
\newblock {\em arXiv preprint arXiv:2203.06056}.

\bibitem[Tian and Pearl, 2002]{tian2002general}
Tian, J. and Pearl, J. (2002).
\newblock A general identification condition for causal effects.
\newblock In {\em Eighteenth National Conference on Artificial Intelligence},
  page 567–573, USA. American Association for Artificial Intelligence.

\bibitem[Triantafillou and Tsamardinos, 2015]{triantafillou2015constraint}
Triantafillou, S. and Tsamardinos, I. (2015).
\newblock Constraint-based causal discovery from multiple interventions over
  overlapping variable sets.
\newblock {\em J. Mach. Learn. Res.}, 16(1):2147–2205.

\bibitem[Verma and Pearl, 1990]{verma1990causal}
Verma, T. and Pearl, J. (1990).
\newblock Causal networks: Semantics and expressiveness.
\newblock In Shachter, R.~D., Levitt, T.~S., Kanal, L.~N., and Lemmer, J.~F.,
  editors, {\em Uncertainty in Artificial Intelligence}, volume~9 of {\em
  Machine Intelligence and Pattern Recognition}, pages 69--76. North-Holland.

\bibitem[Zhang, 2008]{zhang2008causal}
Zhang, J. (2008).
\newblock Causal reasoning with ancestral graphs.
\newblock {\em Journal of Machine Learning Research}, 9(47):1437--1474.

\end{thebibliography}

\end{document}